\documentclass[a4paper,12pt,twoside]{amsart}
\usepackage{amsmath,amssymb}
\usepackage[all,cmtip]{xy}
\usepackage{mathtools, stmaryrd}
\usepackage[OT2,T1]{fontenc}
\usepackage{substitutefont}
\substitutefont{OT2}{\rmdefault}{wncyr}

\usepackage{url}
\usepackage{datetime}
\usepackage{abraces}
\usepackage[T1]{fontenc}
\usepackage{vmargin}
\usepackage{amssymb}
\usepackage{csquotes}
\usepackage{datetime}
\usepackage{amsmath}
\numberwithin{equation}{subsection}
\usepackage{amsthm}
\usepackage{mathrsfs}
\usepackage[dvips]{graphicx}
\usepackage{amscd}
\usepackage{wasysym}
\usepackage{float}
\usepackage{stmaryrd} 
\usepackage[colorlinks=true,pdfpagemode=FullScreen]{hyperref}
\DeclareGraphicsExtensions{.pdf, .png, .jpg, .mps}
\usepackage{mathtools}
\usepackage{color}
\definecolor{airforceblue}{rgb}{0.36, 0.54, 0.66}
\definecolor{bured}{rgb}{0.8, 0.0, 0.0}
\definecolor{burntsienna}{rgb}{0.91, 0.45, 0.32}
\definecolor{burntsienna}{rgb}{0.91, 0.45, 0.32}
\definecolor{burntsienna}{rgb}{0.91, 0.45, 0.32}
\definecolor{chromeyellow}{rgb}{1.0, 0.65, 0.0}
\definecolor{cobalt}{rgb}{0.0, 0.28, 0.67}
\definecolor{coralred}{rgb}{1.0, 0.25, 0.25}
\definecolor{coralred}{rgb}{1.0, 0.25, 0.25}
\definecolor{coralred}{rgb}{1.0, 0.25, 0.25}
\definecolor{ferrarired}{rgb}{1.0, 0.11, 0.0}
\definecolor{forestgreen}{rgb}{0.0, 0.27, 0.13}
\definecolor{cream}{rgb}{1.0, 0.99, 0.82}
\definecolor{cream}{rgb}{1.0, 0.99, 0.82}
\definecolor{cream}{rgb}{1.0, 0.99, 0.82}
\definecolor{cream}{rgb}{1.0, 0.99, 0.82}
\definecolor{candy}{rgb}{0.64, 0.0, 0.0}
\definecolor{darkblue}{rgb}{0.0, 0.0, 0.55}
\definecolor{crimson}{rgb}{0.86, 0.08, 0.24}
\definecolor{carmine}{rgb}{0.94, 0.19, 0.22}
\definecolor{ashgrey}{rgb}{0.7, 0.75, 0.71}
\definecolor{tomato}{rgb}{1.0, 0.39, 0.28}
\definecolor{awesome}{rgb}{1.0, 0.13, 0.32}
\definecolor{cadmiumorange}{rgb}{0.93, 0.53, 0.18}
\definecolor{darkseagreen}{rgb}{0.56, 0.74, 0.56}
\definecolor{darkspringgreen}{rgb}{0.09, 0.45, 0.27}
\definecolor{antiquefuchsia}{rgb}{0.57, 0.36, 0.51}
\definecolor{amaranth}{rgb}{0.9, 0.17, 0.31}
\definecolor{myc}{cmyk}{0.0009,0.8,0.8,0.00}
\newtheorem{theorem}{Theorem}[section]
\newtheorem{proposition}[theorem]{Proposition}
\newtheorem{corollary}[theorem]{Corollary}
\newtheorem{lemma}[theorem]{Lemma}
\newtheorem{claim}[theorem]{Claim}

\newtheorem*{nb}{\footnotesize N.B}
\theoremstyle{definition}
\newtheorem{definition}[theorem]{Definition}
\newtheorem{remark}[theorem]{Remark}

\newtheorem{ex}[theorem]{Example}
\newtheorem{exs}[theorem]{Examples}
\newcommand\p{\partial}
\newcommand\io{{\infty}}
\newcommand\inde{\operatorname{index}}

\newcommand\range{\operatorname{ran}}
\newcommand\diag{\operatorname{diag}}

\newcommand\re{\operatorname{Re}}
\newcommand\card{\operatorname{Card}}
\newcommand\spectrum{\operatorname{Spectrum}}

\newcommand\Id{\operatorname{Id}}

\newcommand\N{\mathbb N}
\newcommand\Z{\mathbb Z}

\newcommand\R{\mathbb R}

\newcommand{\poscal}[2]{\langle#1,#2\rangle}
\newcommand{\Poscal}[2]{\left\langle#1,#2\right\rangle}

\newcommand{\poi}[2]{\left\{#1,#2\right\}}
\newcommand{\Poi}[2]{\Bigl\{#1,#2\Bigr\}}
\newcommand{\norm}[1]{\Vert#1\Vert}
\newcommand{\Norm}[1]{\bigl\Vert#1\bigr\Vert}
\newcommand{\trinorm}[1]{\vert\hskip-1pt\norm{#1}\hskip-1pt\vert}

\newcommand{\tend}[4]{{{#1}\begin{matrix}\\ \longrightarrow\\ \hfill{\scriptstyle #3\rightarrow #4}\end{matrix}#2}}
\newcommand{\val}[1]{\vert#1\vert}
\newcommand{\Val}[1]{\left\vert#1\right\vert}

\newcommand{\sign}{\operatorname{sign}}
\newcommand{\tr}[1]{{^t}#1}

\newcommand{\ops}[2]{{\text{Op}_{#1}(#2)}}
\newcommand{\OPS}[2]{{\text{Op}_{#1}\bigl(#2\bigr)}}
\newcommand{\opw}[1]{{\text{\textrm Op}_{\text{\textrm w}}(#1)}}
\newcommand{\OPW}[1]{{\text{\textrm Op}_{\text{\textrm w}}\bigl(#1\bigr)}}

\newcommand{\mett}[4]{M_{#1,#2,#3}^{\scriptscriptstyle\{#4\}}}
\newcommand{\Mett}[4]{\mathcal M_{#1,#2,#3}^{\scriptscriptstyle\{#4\}}}
\newcommand{\RZ}{\R^{2n}}
\newcommand{\trace}{\operatorname{trace}}
\newcommand{\hs}{{\hskip15pt}}
\newcommand{\vs}{\vskip.3cm}

\let\no=\noindent

\newcommand{\wt}[1]{\widetilde{#1}}

\def\mat22#1#2#3#4{\begin{pmatrix}#1&#2\\ #3&#4\end{pmatrix}}

\def\symm{\text{\sl Sym}(n,\R)}
\def\sym+{\text{\sl Sym}_{+}(n,\R)}
 
\def\mett#1#2#3#4{M_{#1,#2,#3}^{\scriptscriptstyle\{#4\}}}
\def\Mett#1#2#3#4{\mathcal M_{#1,#2,#3}^{\scriptscriptstyle\{#4\}}}
\def\expo#1{^{\scriptscriptstyle\{#1\}}}

\begin{document}
\title[Uncertainty principle for metaplectic transformations]{On the Uncertainty Principle \\for Metaplectic Transformations}
\date{\today}
\author[Nicolas Lerner]{Nicolas Lerner}
\address{\noindent \textsc{N. Lerner, Institut de Math\'ematiques de Jussieu,
Sorbonne Universit\'e (formerly Paris VI),
Campus Pierre et Marie Curie,
4 Place Jussieu,
75252 Paris cedex 05,
France}}
\email{nicolas.lerner@imj-prg.fr, nicolas.lerner@sorbonne-universite.fr}
\numberwithin{equation}{subsection}
\begin{abstract}
We explore the new proofs and extensions of the Heisenberg Uncertainty Principle introduced by 
A.~Widgerson \& Y.~Widgerson in \cite{MR4229152},
deve\-loped in \cite{MR4453622} by 
N.C.~Dias,
F.~Luef
and J.N.~Prata
and also in \cite{MR4337266}
by Y. Tang.
In particular we give here a proof of the Uncertainty Principle for operators in the Metaplectic group
in any dimension.
\end{abstract}
\keywords{Uncertainty principle, Wigner distribution, Metaplectic group}
\subjclass[2010]{}
\dedicatory{\empty}
\maketitle
$$\color{magenta}
\boxed{\boxed{\textbf{\today,\quad\currenttime}}}
$$
{\color{ferrarired}\footnotesize
\tableofcontents}
\baselineskip=1.2\normalbaselineskip
\section{Introduction}
\subsection{Foreword}
The Heisenberg Uncertainty Principle,
a cornerstone of Quantum Mechanics,
could somehow be summarized by the fact that \emph{Operators do not usually commute}
and the calculation of the commutator of two operators $A,B,$
say from a Hilbert space into itself,
\begin{equation}\label{comm01}
[A,B]=AB-BA,
\end{equation}
can unravel important properties linked to stability of matter. In particular the fact that, in one dimension,
\begin{equation}\label{sharp1}
\Bigl[\frac{d}{dx}, x\Bigr]=1,
\end{equation}
implies readily that 
\begin{multline}\label{flat02}
\Bigl\Vert{\frac{du}{dx}}\Bigr\Vert_{L^{2}(\R)}\norm{xu}_{L^{2}(\R)}\ge
\re\poscal{\frac{du}{idx}}{ixu}_{L^{2}(\R)}=\frac12\re\poscal{[\frac{d}{dx},x] u}{u}_{L^{2}(\R)}\\=\frac12\norm{u}_{L^{2}(\R)}^{2},
\end{multline}
and is closely related to the fact that the Harmonic
Oscillator
$$
\mathcal H=-\hbar^{2}\frac{d^{2}}{dx^{2}}+x^{2},
$$
is bounded from below by the positive constant $\hbar$, whereas the Hamiltonian function $\xi^{2}+x^{2}$
(\emph{quantized} by $\mathcal H$)
is actually vanishing at the origin.
Similarly, we have in $n$ dimensions with $n\ge 2$, $\kappa>0$, 
$$
\mathcal L_{\kappa}=-\frac{\hbar^{2}\Delta}2-\frac{\kappa}{\val x}\ge -\frac{2\kappa^{2}}{(n-1)^{2}\hbar^{2}}>-\infty,
$$
whereas the Hamiltonian function 
$
\frac{\val{\xi}^{2}}2-\frac{\kappa}{\val x}
$
(quantized by $\mathcal L_{\kappa}$)
is unbounded from below.
The \emph{Heisenberg inequality} \eqref{flat02} can be written in $n$ dimensions as 
\begin{equation}\label{natura}
\Norm{\val D u}_{L^2(\R^n)}\Norm{\val x u}_{L^2(\R^n)}\ge \frac{n}{4\pi}\Norm{ u}_{L^2(\R^n)}^2,
\end{equation}
and follows easily from the one-dimensional  \eqref{flat02}.
(see e.g. Section \ref{sec.ugtf88} in our Appendix.) One could say as well as for \eqref{flat02}
that \eqref{natura} is a consequence of the non-commutation of the Fourier multiplier $\val D$
with the operator of multiplication by $\val x$.
It is also interesting to notice
the following result (see e.g. Lemma 1.2.29 in \cite{MR4726947}).
\begin{lemma}
Let $\mathbb H$ be a Hilbert space and let $J,K\in \mathcal B(\mathbb H)$; then the commutator $[J,K]\not=\Id$.
\end{lemma}
That simple result, whose proof is elementary, is showing that 
Quantum Mechanics must deal with unbounded operators to have properties like \eqref{sharp1}
or creation-annihilation properties such as
$$
\frac12\Bigl[ \frac{d}{idx} -ix, \frac{d}{idx}+ix \Bigr]=1.
$$
E.~Lieb's article \cite{MR2766495} and E.~Lieb \& R.~Seiringer' book \cite{MR2583992}
are providing a whole theory of stability of matter based upon the non-commutation of operators quantifying the kinetic energy with the operators of multiplication by a potential depending on the configuration variables. 
\par
The Uncertainty Principle can be also seen as a \emph{limitation on the sharp localization of both $u$ and its Fourier transform\footnote{See Section \ref{sec.app001} in the Appendix for our normalization in the definition of the Fourier transformation.}
$\widehat u$} in the sense that, for $u$ normalized in $L^{2}(\R^{n})$, 
\begin{equation}\label{dag001}
\inf_{(y,\eta)\in \RZ}\Norm{\val{\eta-\xi} \widehat u(\xi)}_{L^2(\R^n_{\xi})}\Norm{\val{y-x} u(x)}_{L^2(\R^n_{x})}\ge \frac{n}{4\pi},
\end{equation}
which is an immediate consequence of \eqref{natura}.
The paper \cite{MR1993414}
by
A.~Bonami, 
B.~Demange
and P.~Jaming,
(see also \cite{MR2264204} by the first two previous authors),
generalizing the article \cite{MR1150375} of L.~H\"ormander,
proved in particular that
\begin{equation}\label{dag002}
\iint_{\R^{n}\times \R^{n}}\frac{\val{u(x)}\val{\widehat u(\xi)}}{
(1+\val x+\val \xi)^{n}} e^{2\pi \val{\poscal{x}{\xi}}} dx d\xi<+\io
\Longrightarrow u=0.
\end{equation}
Also the papers  \cite{MR0780328} by M.~Benedicks,
\cite{MR0461025} by W.O.~Amrein \&  A.M.~Berthier, proved that both supports of $u,\widehat u$ cannot have a finite Lebesgue measure unless $u=0$.
That result was extended to more general operators than the Fourier transformation  in the article \cite{MR3173045}
by S.~Ghobber \& P.~Jaming.
\par
We must also quote the wide  and sophisticated program around the Uncertainty Principle
given in the article  \cite{MR0611747}
by C.~Fefferman \& D.H.~Phong
and in the C.~Fefferman's survey paper \cite{MR0707957}, where the authors  were able to describe an
explicit procedure to \emph{obtain a sharp  lowerbound  for large classes of pseudo-differential operators}.
In fact, given a Hamiltonian function $a(x,\xi)$,
a smooth real-valued function verifying some rather natural estimates,
they would consider the set
\begin{equation}\label{diam01}
\mathcal N_{a}=\{(x,\xi)\in \RZ, a(x,\xi)<0\},
\end{equation}
and evaluate its symplectic importance by counting the number of disjoint symplectic images
of the unit ball contained in it.
For instance, if $\mathcal N_{a}$ does not contain any symplectic image of the unit ball of $\RZ$, then the operator $\opw{a}$
(the Weyl quantization of the Hamiltonian function $a$, see e.g. Section \ref{sec.weyl})
should be bounded from below.
Note that there is a striking link with the results of symplectic rigidity of M.~Gromov
\cite{MR0809718}
(see also the articles by I.~Ekeland \& H.~Hofer
\cite{MR0978597},
\cite{MR1044064})
where it is proven that the very existence of a symplectic mapping $\chi$,
such that
\[
\chi: \underbrace{\{(x,\xi)\in \RZ, {\val{x}}^{2}+{\val\xi}^{2}\le 1\}}_{\text{the unit ball}}\longrightarrow  \underbrace{\{(x,\xi)\in \RZ, x_{1}^{2}+\xi_{1}^{2}\le r^{2}\}}_{\text{vertical cylinder with radius $r$}}
\]
implies that $r$ should be $\ge 1$.
Of course that obstruction is not due to the volume when $n\ge 2$, since the cylinder has an infinite volume.
In particular, the Fefferman-Phong procedure should lead to the fact that  the operator 
$$
\mathcal U^{*}\bigl[\opw{x_{1}^{2}+\xi_{1}^{2}}\bigr]\mathcal U,
$$ 
should be bounded from below by a positive constant,
for a very large class of non-classical elliptic Fourier Integral Operators $\mathcal U$,
quantifying a canonical transformation.
The explicit mathematical link between the results of M.~Gromov, I.~Ekeland \& H.~Hofer mentioned above and the Fefferman-Phong procedure is not so easy to derive
rigourously,
since it does not seem obvious to quantize any smooth canonical transformation $\chi$ by a Fourier Integral Operator
without having some explicit bounds on the derivatives of $
\chi$.
We may nevertheless point out that in the article \cite{MR2664605},
J.-M.~Bony is providing quite general formulas for the quantization of canonical transformations
into tractable Fourier Integral Operators, using some partitions of unity related to metrics on the phase space.
\par
More recently, the paper 
 \cite{MR4229152}
 by 
 A.~Wigderson and Y.~Wigderson
 provided a renewed point of view on uncertainty principles
 along with new proofs. Their paper was  deve\-loped in \cite{MR4453622} by 
N.C.~Dias,
F.~Luef
and J.N.~Prata
and also in \cite{MR4337266}
by Y.~Tang.
\par
This article is organized in four sections: in Section 1, after the above Foreword, we formulate the main result of the article, a version of the uncertainty principle for metaplectic transformations.
We hope that having a statement at the beginning of the article could serve as a motivation for the reader to go on. However, it seems necessary to spend some time at reminding a couple of classical  facts  on the symplectic and metaplectic groups: this is what we do in Section 2,
in particular with the perspective to using a special set of generators for the metaplectic group.
We end that section with some basic elements on the Weyl quantization.
In Section 3, we give the proof of the main theorem of the paper,
using the symplectic covariance of the Weyl quantization as a key tool; we split our discussion in two different cases, related to the invertibility of a submatrix of the symplectic transformation linked to the metaplectic mapping under scope.
Our article contains also an Appendix as Section 4, essentially devised to ensure a reasonable degree of 
self-containedness for our arguments.
\subsection{The main result: uncertainty principle for metaplectic transformations}
\begin{definition}
 Let $f$ be  a function in $L^{2}_{\textrm{loc}}(\R^{n})$. We define the \emph{variance} of $f$ as
 \begin{equation}\label{variance}
\mathcal V(f)=\int_{\R^{n}}{\val x}^{2}{\val{f(x)}}^{2}dx.
\end{equation}
\end{definition}
\begin{remark}\label{rem.0013}
 Inequality \eqref{natura} implies that, for $u\in \mathscr S(\R^{n})$, 
\begin{equation}\label{var001}
\mathcal V(\mathcal Fu)\mathcal V(u)\ge \frac{n^{2}}{2^{4}\pi^{2}}\norm{u}_{L^{2}(\R^{n})}^{4},
\end{equation}
where $\mathcal F$ stands for the Fourier transform. 
We may wonder if there are other operators satisfying this type of  inequality, besides the Fourier transform.
 Obviously the Identity does not satisfy any inequality of type \eqref{var001}: we have 
 \begin{equation}\label{213213}
\inf_{\substack{u\in \mathscr S(\R^{n})\\\norm{u}_{L^{2}(\R^{n})=1}}}\mathcal V(u)=0.
\end{equation}
To prove this, we just need to concentrate a function near the origin: take with $\varepsilon>0$,
$
u_{\varepsilon}(x)=\phi(x\varepsilon^{-1})\varepsilon^{-n/2}, \ \phi\in \mathscr S(\R^{n}),
\norm{\phi}_{L^{2}(\R^{n})}=1.
$
We obtain $\norm{u_{\varepsilon}}_{L^{2}(\R^{n})}=1$ and 
\begin{equation*}
\mathcal V(u_{\varepsilon})=\int \val x^{2}
{\val{\phi\bigl(\frac x\varepsilon\bigr)}}^{2}\varepsilon^{-n} dx=\varepsilon^{2}\int
{ \val y}^{2}{\val{\phi(y)}}^{2} dy,
\end{equation*}
implying \eqref{213213}.
\end{remark}
\begin{definition}
Let $M$ be a bounded  operator on $L^{2}(\R^{n})$, sending $\mathscr S(\R^{n})$ into itself. We define 
\begin{equation}\label{mesure++}
\mu(M)=\inf_{\substack{u\in \mathscr S(\R^{n})\\\norm{u}_{L^{2}(\R^{n})=1}}}
\sqrt{\mathcal V(Mu)\mathcal V(u)}.
\end{equation}
\end{definition}
We have proven above\footnote{The constant $\frac{n^{2}}{4\pi^{2}}$ in \eqref{var001} is sharp
as it is seen by checking the variances of $u, \hat u$ with  $u(x)=2^{n/4} e^{-\pi \val x^{2}}$,
 as in Section \ref{sec.ugtf88} of our Appendix.} that 
\begin{equation}\label{first1}
\mu(\mathcal F)=\frac{n}{4\pi}, \qquad \mu(\Id)=0.
\end{equation}
We would like to introduce the symplectic group $\text{\sl Sp}(2n,\R)$
(which is a subgroup of  $\text{\sl Sl}(2n,\R)$)
and its two-fold covering the metaplectic group
$\text{\sl Mp}(n,\R)$,
with
\begin{equation}\label{}
\Psi:\text{\sl Mp}(n,\R)\longrightarrow \text{\sl Sp}(2n,\R),
\end{equation}
as the covering map.
We shall provide in the next sections a couple of reminders about these two groups, but at the moment, we need only to say that
with the canonical symplectic form $\sigma$ being given 
on $\R^{n}\times \R^{n}$  by 
\begin{equation}\label{sympfor}
\poscal{\sigma X}{Y}=\bigl[X,Y\bigr]=\xi\cdot y-\eta\cdot x, \quad \text{with $X=(x,\xi), Y=(y,\eta)$,}
\end{equation}
the symplectic group
$\text{\sl Sp}(2n,\R)$ is the subgroup of $S\in \text{\it Gl}(2n,\R)$
such that\footnote{We shall see below that the determinant of symplectic matrices is actually $1$.} 
\begin{equation}\label{55vhqs}
\forall X, Y\in \RZ, \quad [S X, S Y]=[X,Y], \qquad\text{i.e.\quad} S^{*}\sigma S=\sigma,
\end{equation}
where $S^{*}$ is the transpose and 
\begin{equation}\label{ffqq99}
\sigma=\mat22{0}{I_{n}}{-I_{n}}{0}.
\end{equation}
On the other hand, the metaplectic group
is  a subgroup of the unitary group of $L^{2}(\R^{n})$,
whose elements are also isomorphisms of $\mathscr S(\R^{n})$ and $\mathscr S'(\R^{n})$.
We can then formulate without further postponing 
the main result of this paper.
\begin{theorem}\label{thm.21poii}
Let $\mathcal M$ be an element of the metaplectic group $\text{\sl Mp}(n,\R)$  such that $\Psi(\mathcal M)=\Xi,$
where the $(2n)\times (2n)$ matrix $\Xi\in \text{\sl Sp}(2n,\R)$
is written in four  $(n\times n)$ blocks as 
\begin{equation}\label{blocks}
 \Xi=\mat22{\Xi_{11}}{\Xi_{12}}{\Xi_{21}}{\Xi_{22}}.
\end{equation}
Then we have, with $\mu$ defined in \eqref{mesure++}, 
\begin{equation}\label{main01}
 \mu(\mathcal M)=\frac1{4\pi}\trace\bigl((
 \Xi_{12}^{*}\Xi_{12}
 )^{1/2}\bigr)=\frac1{4\pi}\trinorm{\Xi_{12}}_{{\mathtt{S}^{1}}},
\end{equation}
where $\trinorm{\Xi_{12}}_{{\mathtt{S}^{1}}}$ is the Schatten norm of index 1.
\end{theorem}
\begin{nb}
 For $n=1$, the inequality  $\mu(\mathcal M)\ge \val{\Xi_{12}}/(4\pi)$
 is already proven  for a dense subset of $\text{\sl Mp}(1,\R)$
in the paper 
 \cite{MR4229152}
by 
 A.~Wigderson and Y.~Wigderson.
 We provide here a multi-dimensional version
 along with an algebraic identification of the lower bound $\mu(\mathcal M)$.
\end{nb}
\begin{remark}
 The equalities \eqref{first1} are coherent with the above theorem since we have 
\begin{align*}
\Psi(e^{-\frac{i\pi n}4}\mathcal F)&=\mat22{0}{I_{n}}{-I_{n}}{0}, \quad
\trinorm{I_{n}}_{{\mathtt{S}^{1}}}=n,\quad \mu(z\mathcal F)=\mu(\mathcal F)\quad \text{for $\val z=1$},
\\
 \Psi(\Id_{L^{2}(\R^{n})})&=I_{2n}=\mat22{I_{n}}{0}{0}{I_{n}}.
\end{align*}
\end{remark}
\section{Symplectic and metaplectic groups, Weyl quantization}
We recall in this section some classical facts about the symplectic and metaplectic  groups
as well as various
matters on the Weyl quantization.
We rely essentially on our Memoir \cite{MR4726947},
where complete proofs are given.The statements and examples which follow   will be hopefully  useful
for the reader,   but we shall of course omit the proofs appearing in  \cite{MR4726947}.
We have used 
E.P.~Wigner's works \cite{MR1635991},
J.~Leray's book \cite{MR644633}
and other Lecture Notes of this author at the {\sl Coll\`ege de France} such as \cite{MR0501198},
L.~H\"ormander's
four-volume treatise, {\sl The Analysis of Linear Partial Differential operators} and in particular Volume III,
as well as {K.~Gr\"{o}chenig's }\cite{MR1843717} {\sl Foundations of time-frequency analysis},
along with G.B.~Folland's \cite{MR983366}, 
A.~Unterberger's \cite{MR552965}, 
and N.~Lerner's \cite{MR2599384}.
\subsection{The symplectic group}
The canonical symplectic form $\sigma$
on $\R^{n}\times \R^{n}$ is defined in \eqref{sympfor},
the symplectic group
$\text{\sl Sp}(2n,\R)$ 
is defined in \eqref{55vhqs}.
It is easy to prove directly from \eqref{55vhqs} that
$
\text{\sl Sp}(2,\R)=Sl(2, \R).
$
In the sequel we shall denote by 
\begin{equation}\label{sym001}
\text{\sl Sym}(n,\R)=\{(n\times n) \text{ real symmetric matrices}\}.
\end{equation}
\begin{theorem}\label{4.thm.gensym}
Let $n$ be an integer $\ge 1$.
The group
$\text{\sl Sp}(2n,\R)$ is included in $Sl(2n, \R)$
and generated by the following mappings
\begin{align}
\mat22{I_{n}}{0}{A}{I_{n}}&,\ \text{where $A\in \text{\sl Sym}(n,\R)$,}
\label{syty01}\\
\mat22{B^{-1}}{0}{0}{B^*}&,\quad \text{where }B\in \text{\it Gl}(n,\R),
\label{syty02}\\
\mat22{I_{n}}{-C}{0}{I_{n}}&,\ \text{where $C\in \text{\sl Sym}(n,\R)$.}
\label{syty03}
\end{align} 
For $A, B, C$ as above, the mapping
\begin{multline}\label{gensym}
\Xi_{A,B,C}=
{\arraycolsep=8pt     
\begin{pmatrix*}[l]
{B^{-1}}&{-B^{-1}C}\\
{AB^{-1}}&{B^*-AB^{-1}C}
\end{pmatrix*}}
=
\mat22{I_{n}}{0}{A}{I_{n}}
\mat22{B^{-1}}{0}{0}{B^*}
\mat22{I_{n}}{-C}{0}{I_{n}}.
\end{multline}
belongs to $\text{\sl Sp}(2n,\R)$.
\index{generating function}
Moreover, we define on $\R^{n}\times \R^{n}$ the {\it generating function} $S$
of the symplectic mapping
$\Xi_{A,B,C}$
by the identity
\begin{equation}\label{4.genfun}
S(x,\eta)=\frac12\bigl(\poscal{Ax}{x}+2\poscal{Bx}{\eta}
+\poscal{C\eta}{\eta}
\bigr)
\
\text{so that}
\hs
\Xi \Bigl(\frac{\p S}{\p\eta}\oplus \eta\Bigr)=
x\oplus\frac{\p S}{\p x}.
\end{equation}
For a symplectic mapping $\Xi$, to be of the form \eqref{gensym}
is equivalent
to the assumption that the mapping
$x\mapsto \pi_{\R^{n}\times\{0\}}\Xi(x\oplus 0)$
is invertible from $\R^{n}$ to $\R^{n}$;
moreover, if this mapping is not invertible,
the symplectic mapping $\Xi$ is the product of two mappings of the type
$\Xi_{A,B,C}$.
\end{theorem}
This result is proven in 
Theorem 1.2.6 of  \cite{MR4726947}.
\begin{remark}
Let us consider $\Xi$ in $\text{\sl Sp}(2n,\R)$: we have
\begin{equation}\label{notsym}
\Xi=\mat22{P}{Q}{R}{S},\quad\text{where
$P,
Q,
R,
S,$
are $n\times n$
matrices.
}
\end{equation}
The equation
\begin{equation}\label{4.eqsygr}
{\Xi^{*}} \sigma\Xi=\sigma
\end{equation} is satisfied
with 
$
\sigma=\mat22{0}{I_{n}}{-I_{n}}{0}, 
$
which means
\begin{equation}\label{4.eqsyma}
{P^{*}} R={({P^{*}} R)}^{*},\quad
{Q^{*}} S={({Q^{*} S})}^{*},\quad
{P^{*}} S-{R^{*}} Q=I_{n}.
\end{equation}
 \end{remark}
The following claim is proven in (1.2.23-24) of \cite{MR4726947}.
\begin{claim}\label{claim.isosta}
 The mapping $\Xi\mapsto (\Xi^{*})^{-1}$ is an isomorphism
of the group $\text{\sl Sp}(2n,\R)$.
As a consequence, 
  \begin{gather}
  \Xi=\mat22{P}{Q}{R}{S} \in \text{\sl Sp}(2n,\R)
  \quad\text{ is also equivalent to}
\label{fc45y9}
\\
 P Q^{*}={(P Q^{*})}^{*},\quad
  R S^{*}={(R S^{*})}^{*},\quad
  P S^{*}-Q R^{*}=I_{n}.
  \label{4.eqsym'}
 \end{gather}
\end{claim}
\begin{exs}\label{exa.444}
 \rm The $4\times 4$ matrix
 \begin{equation}\label{sinblo}
 \begin{pmatrix}
\begin{matrix}
0&0\\
0&1
\end{matrix}
&\begin{matrix}
1&0\\
0&0
\end{matrix}\\
\begin{matrix}
-1&0\\
0&0
\end{matrix}&
\begin{matrix}
0&0\\
0&1
\end{matrix}
\end{pmatrix}
=\begin{pmatrix}
P&Q
\\
R&S
\end{pmatrix},
\end{equation}
belongs to $Sp(4, \R)$
 although the block $2\times 2$
 matrices $P,Q,R, S,$
 are all singular (with rank 1).
The symplectic matrix
\begin{equation}\label{noti22}
\mat22{0}{I_{n}}{-I_{n}}{0}=2^{-1/2}
\mat22{I_{n}}{I_{n}}{-I_{n}}{I_{n}}2^{-1/2}
\mat22{I_{n}}{I_{n}}{-I_{n}}{I_{n}}=\Xi_{-I_{n}, 2^{1/2}I_{n}, -I_{n}}^{2},
\end{equation}
is not of the form 
$\Xi_{A,B,C}$ but\footnote{Note that for a symplectic mapping to be of type $\Xi_{A,B,C}$ requires (in fact is equivalent) to the fact that $\Xi_{11}$ is invertible, using the notation \eqref{blocks}.} is the square of such a matrix.
It is also the case of all the mappings
$
(x_{k},\xi_{k})\mapsto (\xi_{k},-x_{k}),$
with the other coordinates fixed.
Similarly the symplectic  matrix
\begin{equation}\label{}
\mat22{0}{-I_{n}}{I_{n}}{I_{n}}=
\mat22{I_{n}}{-I_{n}}{0}{I_{n}}
\mat22{I_{n}}{0}{I_{n}}{I_{n}},
\end{equation}
is not of the form 
$\Xi_{A,B,C}$ but is the product
$\Xi_{0,I,I}\Xi_{I,I,0}$.
\end{exs}
\begin{remark}
We have seen in \eqref{4.eqsym'}, \eqref{4.eqsyma} some equivalent conditions for a matrix
\begin{equation}\label{fgw123}
\Xi=\begin{pmatrix}
P&Q\\
R&S
\end{pmatrix}
\quad\text{where $P,Q,R, S$ are  $n\times n$ real matrices,}
\end{equation}
to be symplectic. 
We note here that when $\Xi\in \text{\sl Sp}(2n,\R)$, we have 
\begin{equation}\label{arfd58}
\Xi^{-1}=\begin{pmatrix}
S^{*}&-Q^{*}\\
-R^{*}&P^{*}
\end{pmatrix},
\end{equation}
as it is easily checked from \eqref{4.eqsym'}, \eqref{4.eqsyma}. 
When $\det P\not=0$, we can prove  that $\Xi=\Xi_{A,B,C}$ as\footnote{If $\Xi$ is given by 
\eqref{fgw123} with $P$ invertible, we define $B=P^{-1}, C=-P^{-1}Q, A=RP^{-1}$
and we get from \eqref{4.eqsyma} that $\Xi=\Xi_{A,B,C}$.
} defined in  \eqref{gensym}.
Also from  \eqref{arfd58}, we get 
that if $\det S\not =0$ we have 
$$
\Xi^{-1}=\Xi_{A,B,C},
$$
so that 
\begin{equation}\label{}
\Xi=\begin{pmatrix}
I_{n}&C\\
0&I_{n}
\end{pmatrix}
\begin{pmatrix}
B&0\\
0&B^{*-1}
\end{pmatrix}
\begin{pmatrix}
I_{n}&0\\
-A&I_{n}
\end{pmatrix}.
\end{equation}
Some other properties of the same type  are available when $\det Q$ or $\det R$ are different from $0$. Indeed we have for $\Xi\in \text{\sl Sp}(2n,\R)$ and $\sigma$ given by \eqref{ffqq99},
\begin{equation}\label{singsing}
\Xi \sigma=
\begin{pmatrix}
P&Q\\
R&S
\end{pmatrix} \sigma
=\begin{pmatrix}
-Q&P\\
-S&R
\end{pmatrix}\underbrace{=}_{\text{if $\det Q\not=0$}}\Xi_{A,B,C},
\end{equation}
so that 
\begin{equation}\label{}
\Xi =-\Xi_{A,B,C} \sigma=
\begin{pmatrix}
I_{n}&0\\
A&I_{n}
\end{pmatrix}
\begin{pmatrix}
B^{-1}&0\\
0&B^{*}
\end{pmatrix}
\begin{pmatrix}
I_{n}&-C\\
0&I_{n}
\end{pmatrix}
\begin{pmatrix}
0&-I_{n}\\
I_{n}&0
\end{pmatrix}.
\end{equation}
If we have $\det R\not=0$, using the two first equalities in \eqref{singsing}, we get that 
$
(\Xi \sigma)^{-1}=\Xi_{A,B,C} 
$
which gives
\begin{equation}\label{}
\Xi=
\begin{pmatrix}
I_{n}&C\\
0&I_{n}
\end{pmatrix}
\begin{pmatrix}
B&0\\
0&B^{*-1}
\end{pmatrix}
\begin{pmatrix}
I_{n}&0\\
-A&I_{n}
\end{pmatrix}
\begin{pmatrix}
0&-I_{n}\\
I_{n}&0
\end{pmatrix}.
\end{equation}
However, it is indeed possible when $n\ge 2$ to have a symplectic matrix in $\text{\sl Sp}(2n,\R)$ in the form \eqref{fgw123} such that all the blocks are singular,
as shown by  the example \eqref{sinblo}.
\end{remark}
\subsection{The metaplectic group}
\subsubsection{Basics}
With
$\Z_{4}$ standing for $\Z/4\Z$, we consider the mapping
\begin{equation}\label{311311}
\symm\times \text{\it Gl}(n,\R)\times \symm\times\Z_{4}
\ni (A,B,C, m)\mapsto \omega(A,B,C,m),
 \end{equation}
 where $\omega(A,B,C,m)$ is a function defined on $\R^{n}\times \R^{n}$ by
 \begin{equation}\label{}
\omega(A,B,C,m)(x,\eta)
=e^{\frac{i\pi m}2} {\val{\det B}}^{1/2}e^{i\pi\{\poscal{Ax}{x}+2\poscal{Bx}{\eta}+\poscal{C\eta}{\eta}\}}.
\end{equation}
We note in particular that 
$\Z_{4}\ni m\mapsto e^{\frac{i\pi m}2}$ is well-defined.
We shall always require the following link between $m$ and $B$:
we want that $m$ be \emph{compatible}  with $B$,
which means
\begin{equation}\label{compat}
e^{i\pi m}=\sign(\det B),
\end{equation}
in such a way that we have $\bigl(e^{\frac{i\pi m}2} {\val{\det B}}^{1/2}\bigr)^{2}=\det B.$
In fact the above mapping \eqref{311311}
is defined for
$$
(B,m)\in \bigl(\text{\it Gl}_{+}(n,\R)\times \{0,2\}\bigr)\cup \bigl(\text{\it Gl}_{-}(n,\R)\times \{1,3\}\bigr),
\quad
(A,C)\in \symm^{2},
$$
where we consider $\{0,2\}$ and $\{1,3\}$ as  subsets of $\Z_{4}$.
We shall denote 
\begin{equation}\label{314314}
m(B)=\{m\in \Z_{4}, e^{i\pi m}=\sign(\det B)\}=
\begin{cases}
\{0,2\}&\text{if $\det B>0$,}
\\
\{1,3\}&\text{if $\det B<0$.}
\end{cases}
\end{equation}
The next proposition is Proposition 1.2.11 in \cite{MR4726947}.
\begin{proposition}\label{pro.kj77qq}
Let $A,B,C$ be as in Theorem \ref{4.thm.gensym},
 let  $S$ be the generating function of $\Xi_{A,B,C}$ (cf. \eqref{4.genfun})
 and let $m\in \Z_{4}$ such that \eqref{compat} is satisfied.
We define the operator $M_{A,B,C}^{\scriptscriptstyle\{m\}}$ on $\mathscr S(\R^{n})$ by
\begin{equation}\label{4.genmet}
(M_{A,B,C}^{\scriptscriptstyle\{m\}}v)(x)=e^{\frac{i\pi m}2}
{\val{\det B}}^{1/2}
\int_{\R^{n}}e^{2i\pi S(x,\eta)}\hat v(\eta) d\eta.
\end{equation}
This operator is an automorphism of 
$\mathscr S'(\R^{n})$ and of $\mathscr S(\R^{n})$
which is unitary on $L^2(\R^{n})$,
and such that, for all $a\in \mathscr S'(\RZ)$, 
\begin{equation}\label{1239}
\bigl(M_{A,B,C}^{\scriptscriptstyle\{m\}}\bigr)^{\!*}\ \opw{a}M_{A,B,C}^{\scriptscriptstyle\{m\}}=\opw{a\circ\Xi_{A,B,C}},
\end{equation}
where $\Xi_{A,B,C}$ is defined   in Theorem \ref{4.thm.gensym}.
\end{proposition}
\begin{remark}
 We have for $A,B,C$ as above, 
 \begin{align}
 (M_{A, I, 0}^{\scriptscriptstyle\{0\}} v)(x)&= e^{i\pi \poscal{Ax}{x}} v(x),\label{mety01}\\
 (M_{0,B,0}^{\scriptscriptstyle\{m\}} v)(x)&=e^{\frac{i\pi m}2}
{\val{\det B}}^{1/2}
v(Bx),\ m \text{ satisfying \eqref{compat}},\label{mety02}\\
   (M_{0, I, C}^{\scriptscriptstyle\{0\}} v)(x)&= \bigl(e^{i\pi \poscal{CD_{x}}{D_{x}}} v\bigr)(x),\label{mety03}
\end{align}
three operators which are obviously automorphisms of $\mathscr S(\R^{n})$ and of 
$\mathscr S'(\R^{n})$ as well as unitary operators in $L^{2}(\R^{n})$.
In particular we find that, keeping in mind \eqref{1254},
\begin{gather}
\bigl(M_{A, I, 0}^{\scriptscriptstyle\{m\}}\bigr)^{-1}=
M_{-A, I, 0}^{\scriptscriptstyle\{-m\}}, \quad
\bigl(M_{0, I, C}^{\scriptscriptstyle\{m\}}\bigr)^{-1}=M_{0, I, -C}^{\scriptscriptstyle\{-m\}},\quad m\in m(I_{n})=\{0,2\},
\label{311313}\\
\bigl(M_{0,B,0}^{\scriptscriptstyle\{m\}}\bigr)^{-1}=M_{0,B^{-1},0}^{\scriptscriptstyle\{-m\}},
\quad m\in m(B).
\label{311414}
\end{gather}
\end{remark}
\begin{proof}
 Formula \eqref{1239} is easily checked  for each operator \eqref{mety01}, \eqref{mety02}
 and \eqref{mety03}.
 We have also
 \begin{align}
\mathcal F^{*}\opw{a}\mathcal F&=\opw{a(\xi,-x)},\label{segal01}\\
e^{-i\pi Ax^{2}}\opw{a}e^{i\pi Ax^{2}}&=\opw{a(x,Ax+\xi)},\ \text{\footnotesize $A$ real symmetric $n\times n$ matrix,}\label{segal02}\\
e^{-i\pi CD^{2}}\opw{a}e^{i\pi CD^{2}}&=\opw{a(x-C\xi,\xi)},\ \text{\footnotesize $C$ real symmetric $n\times n$ matrix,}\label{segal03}\\
\mathcal N_{B}^{*}\opw{a}
\mathcal N_{B}&=\opw{a(B^{-1}x,B^{*}\xi)}, \quad B\in \text{\it Gl}(n,\R),\label{segal04}
\end{align}
where 
\begin{equation}\label{}
\bigl(\mathcal N_{B} v\bigr)(x)={\val{\det B}}^{1/2} v(Bx).
\end{equation}
Since we have $\Xi_{A,B,C}=\Xi_{A,I,0}\ \Xi_{0,B,0}\ \Xi_{0,I,C}$ and 
 \begin{align}\label{54gvrs}
 M_{A,B,C}^{\scriptscriptstyle\{m\}} =M_{A,I,0}^{\scriptscriptstyle\{0\}} M_{0,B,0}^{\scriptscriptstyle\{m\}} 
 M_{0,I,C}^{\scriptscriptstyle\{0\}},
\end{align}
we get \eqref{1239} and the proposition.
\end{proof}
%%%%%%%%%%%%%%%%%%%%%%%%%%%%%%%%%%%%%%%%%%
\begin{remark}\label{footsign}
With the set $m(B)$ defined in \eqref{314314} and with obvious notations we have
\begin{equation}\label{1254}
m(-B)= m(B)+n,\quad m(B^{*})= m(B), \quad m(B^{-1})=-m(B)=m(B).
\end{equation}
\rm
Indeed, $m\in m(-B)$ is equivalent to having
$$
e^{i\pi m}=\sign(\det(-B))=(-1)^{n}\sign(\det B)=e^{i\pi(n+m')}, \ m'\in m(B),
$$
which is equivalent to  $(m-n-m')\in \{0,2\}$ modulo 4, proving the first result,
since $m(B)+\{0,2\}=m(B)$.
On the other hand,
$m\in m(B^{*})$ is equivalent to 
$$
e^{i\pi m}=\sign(\det(B^{*}))=\sign(\det B)=e^{i\pi m'}, \ m'\in m(B),
$$
which means
$m(B^{*})=m(B)+\{0,2\}=m(B).$
Also the
equality $\sign(\det B)=\sign(\det (B^{-1}))$ implies that $m(B)=m(B^{-1})$ and moreover since the mapping
$$\Z_{4}\ni m\mapsto -m\in \Z_{4},$$ leaves invariant $\{0,2\}$ and stable
$\{1,3\}$,
we get the last equality in \eqref{1254}.
We note in particular that we have
\begin{equation}\label{sig+++}
\mett{0}{I_{n}}{0}{0}=\Id_{L^{2}(\R^{n})},\ 
\mett{0}{I_{n}}{0}{2}=-\Id_{L^{2}(\R^{n})},
\end{equation}
and also with the notation \eqref{sigma0}, using the now proven  \eqref{1254},
\begin{equation}\label{3115++}
\mett{0}{-I_{n}}{0}{n}=e^{\frac{i\pi n}2}\sigma_{0},\
\mett{0}{-I_{n}}{0}{n+2}=-e^{\frac{i\pi n}2}\sigma_{0},
\end{equation}
noting that 
$
(\pm e^{\frac{i\pi n}2})^{2}=(-1)^{n}=\det(-I_{n})
$
so that \eqref{compat} is satisfied.
More generally, we have 
$$
 \text{for $\det B>0$, }\mett{A}{B}{C}{0}=-\mett{A}{B}{C}{2},
 \qquad
 \text{for $\det B<0$, }\mett{A}{B}{C}{1}=-\mett{A}{B}{C}{3}.
$$
\end{remark}

%%%%%%%%%%%%%%%%%%%%%%%%%%%%%%%%%%%%%%%%%%
\begin{exs}\rm
In \eqref{noti22}, we have seen in particular that, for $n=1$,  
$$
\mat22{0}{1}{-1}{0}=\left\{2^{-1/2}\mat22{1}{1}{-1}{1}\right\}^{2}, \quad
2^{-1/2}\mat22{1}{1}{-1}{1} =\Xi_{-1,2^{1/2}, -1}.
$$
We have also with the notations of Theorem \ref{4.thm.gensym},
$$
(M_{-1, 2^{1/2},-1}\expo{0}v)(x)=\int_{\R}e^{2i\pi\frac12(-x^{2}+2^{3/2}x\eta-\eta^{2})}\hat v(\eta) d\eta 2^{1/4},
$$
so that the kernel $k_{1}(x,y)$ of the operator $M_{-1, 2^{1/2},-1}\expo{0}$ is 
$$
k_{1}(x,y)=2^{1/4}\int e^{i\pi(
-x^{2}+2^{3/2}x\eta-\eta^{2}
)} e^{-2i\pi y\eta} d\eta\underbrace{=}_{\text{use  \eqref{foimga}}}
2^{1/4}e^{-i\pi/4} e^{i\pi(x^{2}+y^{2})}e^{-2^{3/2} i\pi xy},
$$
so that the kernel $k_{2}$
of the operator $(M_{-1, 2^{1/2},-1}\expo{0})^{2}$ is (using again  \eqref{foimga}),
\begin{multline*}
k_{2}(x,y)=\int k_{1}(x,z) k_{1}(z,y) dz\\=2^{1/2} e^{-i\pi/2} e^{i\pi(x^{2}+y^{2})}
\int e^{2i\pi z^{2}} e^{-2i\pi z2^{1/2}(x+y)}dz
=e^{-i\pi/4} e^{-2i\pi xy},
\end{multline*}
so that 
\begin{equation}\label{}
(M_{-1, 2^{1/2},-1}\expo{0})^{2}=e^{-i\pi/4}\mathcal F_{1}, \end{equation}
with $\mathcal F_{1}$ standing for the  1D Fourier transformation.
We get by tensorisation that in $n$ dimensions, 
\begin{equation}\label{kjh123}
(M_{-I_{n}, 2^{1/2}I_{n},-I_{n}}\expo{0})^{2}=e^{-i\pi n/4}\mathcal F_{n}, \end{equation}
with $\mathcal F_{n}$ standing for the  Fourier transformation in $n$ dimensions.
Similar expressions can be obtained for $\mathcal F_{[k;n]}$, the Fourier transformation with respect to the variable $x_{k}$ in $n$ dimensions, $k\in \llbracket1,n\rrbracket$ with 
\begin{equation}\label{}
(M_{A_{k}, B_{k},C_{k}}\expo{0})^{2}=
e^{-i\pi/4}\mathcal F_{[k;n]},
\end{equation}
where $B_{k}$ is the $n\times n$ diagonal matrix with diagonal entries equal to $1$ except for the $k$th equal to $2^{1/2}$, the $n\times n$ diagonal
matrices $A_{k}=C_{k}$ with diagonal entries equal to 0, except for the $k$th equal to $-1$.
\end{exs}
\index{metaplectic group}
\begin{definition}\label{def.41fdhh}\rm
 The metaplectic group $\text{\sl Mp}(n)$
 is defined  as the subgroup of the group of unitary operators on $L^{2}(\R^{n})$ generated by
 \begin{align}
&M_{A,I,0}\expo{m'},\text{\footnotesize where $A\in\symm$, $m'\in\{0,2\}=m(I)$, cf. \eqref{mety01}},\label{hhjj01}\\
 &M_{0,B,0}\expo{m},\text{\footnotesize
 with $B\in \text{\it Gl}(n,\R)$, $m\in m(B)$,
 cf. \eqref{mety02},
 }\label{hhjj02}\\
 &M_{0,I,C}\expo{m''},\text{\footnotesize where $C\in \symm$, $m''\in\{0,2\}=m(I)$, cf. \eqref{mety03}}.
 \label{hhjj03}
\end{align}
\end{definition}
\begin{remark}\label{rem.44oorr}
 The elements of $\text{\sl Mp}(n)$
are also automorphisms of $\mathscr S(\R^{n})$ and of $\mathscr S'(\R^{n}).$
Moreover, from \eqref{311313},\eqref{311414}, the elements of $\text{\sl Mp}(n)$ are products
of operators of type \eqref{hhjj01},\eqref{hhjj02},\eqref{hhjj03}.
\end{remark}
\begin{claim}\label{claim.444}
 If $M$ belongs to $\text{\sl Mp}(n)$, then $-M$  belongs to $\text{\sl Mp}(n)$.
\end{claim}
\begin{proof}
 According to Remark \ref{footsign}, we have $\mett{0}{I_{n}}{0}{2}=-\mett{0}{I_{n}}{0}{0}=-\Id_{L^{2}(\R^{n})}$
 so that $-\Id_{L^{2}(\R^{n})}$ belongs to $\text{\sl Mp}(n)$, proving the claim.
\end{proof}
The next proposition is Proposition 1.2.15 in \cite{MR4726947}.
\begin{proposition}
The metaplectic group $\text{\sl Mp}(n)$
 is  generated by
 \begin{align}
&M_{A,I,0}\expo{m'},\text{\footnotesize where $A\in\symm$, $m'\in\{0,2\}$, cf. \eqref{mety01}},
\label{778801}\\
 &M_{0,B,0}\expo{m},\text{\footnotesize
 with $B\in \text{\it Gl}(n,\R)$, $m\in m(B)$,
 cf. \eqref{mety02},
 }\label{778802}\\
 &e^{-\frac{i\pi n}4}\mathcal F, \text{\footnotesize where $\mathcal F$ is the Fourier transformation.}
 \label{778803}
\end{align}
 \end{proposition}
 \begin{remark} We have, with $\sigma_{0}$ defined in \eqref{sigma0}, 
 \begin{equation}\label{azqs22}
\sigma_{0}\mathcal F^{2}=\Id,\quad[\mathcal F,\sigma_{0}]=0, \quad{\text{so that}\quad \mathcal F^{*}=\mathcal F^{-1}=\sigma_{0}\mathcal F=\mathcal F\sigma_{0}.}
\end{equation}
 \end{remark}
  \begin{remark}
 According to \eqref{azqs22}  and to Remark \ref{footsign}, we find
 \begin{equation}\label{}
(e^{-i\pi n/4} \mathcal F)^{*}=e^{i\pi n/4} \mathcal F\sigma_{0}=e^{-i\pi n/4} \mathcal Fe^{i\pi n/2} \sigma_{0}
=e^{-i\pi n/4} \mathcal F M_{0,-I_{n}, 0}^{\scriptscriptstyle\{n\}}.
\end{equation}
As a consequence, 
$
e^{-i\pi n/4}\mathcal F, e^{-i\pi n/2}\sigma_{0}, e^{i\pi n/2}\sigma_{0}
$
belong to the metaplectic group.
\end{remark}
\begin{lemma}\label{lem.54dftz}
 For $Y\in \RZ$, we define the linear form $L_{Y}$ on $\RZ$ by
 \begin{equation}\label{linear}
L_{Y}(X)=\poscal{\sigma Y}{X}=[Y,X].
\end{equation}
For any $M\in \text{\sl Mp}(n)$ there exists a unique $\chi\in \text{\sl Sp}(2n,\R)$ such that 
\begin{equation}\label{hg54rr}\forall Y\in \RZ, \quad 
M^{*}\opw{L_{Y}}M=\opw{L_{\chi^{-1} Y}}.
\end{equation}
\end{lemma}
This is Lemma 1.2.17 in \cite{MR4726947}.
We can thus define a mapping 
\begin{equation}\label{hqma28}
\Psi:\text{\sl Mp}(n)\rightarrow \text{\sl Sp}(2n,\R), \quad\text{with $\Psi(M)=\chi$ satisfying
\eqref{hg54rr}.} 
\end{equation}
In particular we have from \eqref{1239} in Proposition \ref{pro.kj77qq} and \eqref{kjh123}, that
\begin{equation}\label{1270}
\Psi(M_{A,B,C})=\Xi_{A,B,C}, \qquad \Psi\bigl(e^{-\frac{i\pi n}4}\mathcal F\bigr)=\sigma=\mat22{0}{I_{n}}{-I_{n}}{0}.
\end{equation}
The following result is Theorem 1.2.18 in \cite{MR4726947}.
\begin{theorem}\label{thm.groups}
 The mapping $\Psi$ defined in \eqref{hqma28} is a surjective homomorphism of groups with kernel $\{\pm 
 \Id_{L^{2}(\R^{n})}
 \}$.
\end{theorem}
\begin{remark}
  As a consequence, the mapping $\Psi$
 is a two-fold covering and $\text{\sl Mp}(n)$ appears as $Sp_{2}(2n,\R)$,
 the two-fold covering of $\text{\sl Sp}(2n,\R)$.
 In fact, our results related to the uncertainty principle, such as Theorem \ref{thm.21poii},
 can be applied
 to the 
  subgroup of the unitary group of $L^{2}(\R^{n})$ generated by
 \eqref{hhjj01}, \eqref{hhjj02}, \eqref{hhjj03} and the operators of multiplication by a complex number of modulus 1
 ($L^{2}(\R^{n})\ni u\mapsto z u\in L^{2}(\R^{n})$ where $\val z=1$).
 Indeed, for a metaplectic operator $\mathcal M$ and $z\in {\mathbb C}$,
 we have obviously
 \begin{equation}\label{}
\mu\bigl(z\mathcal M\bigr)=\val z\mu(\mathcal M).
\end{equation}
 \end{remark}
\subsubsection{Another set of generators for the metaplectic group}
\begin{definition}\label{def.lersym} 
Let $P, L, Q$ be $n\times n$ real matrices such that $P=P^{*}, Q=Q^{*}$,
$\det L\not=0$ and $m\in \Z_{4}$ compatible with $L$ (cf.\eqref{compat}).
We define the operator $\mathcal M_{P,L,Q}\expo{m}$ by the formula
\begin{equation}\label{322322}
(\mathcal M_{P,L,Q}\expo{m} u)(x)=e^{-\frac{i\pi n}4}e^{\frac{i\pi m}2}{\val{\det L}}^{1/2}\int_{\R^{n}} e^{i\pi\{
\poscal{Px}{x}-2\poscal{Lx}{y}+\poscal{Q y}{y}\}}
u(y) dy.
\end{equation}
\end{definition}
\begin{nb}\rm
 Let us note that for $m$ compatible with $L$, we have 
 $$
 \left(e^{\frac{i\pi m}2}\val{\det L}^{1/2}\right)^{2}= \det L.
 $$
\end{nb}
The next result is Proposition A.2.2 in the Appendix  of \cite{MR4726947}.
\begin{proposition}\label{pro.989898}
 The operator $\mathcal M_{P,L,Q}\expo{m}$ given in Definition \ref{def.lersym} is an automorphism of $\mathscr S(\R^{n})$ and of $\mathscr S'(\R^{n})$ which is a unitary operator on $L^{2}(\R^{n})$ belonging to the metaplectic group
 (cf. Definition \ref{def.41fdhh}). Moreover the metaplectic group is generated by the set $\{\mathcal M_{P,L,Q}\expo{m}\}_{\substack{P=P^{*},\hskip0.7pt Q=Q^{*}\\
 \det L\not=0,\hskip0.7pt m\in m(L)}}$.
\end{proposition}
\begin{remark}
Using the Notation \eqref{4.genmet} and \eqref{1254}, we see that\footnote{
 We note that $m(B)+n\in\{m(-B),m(-B)+2\}$ modulo 4:
 indeed we have modulo 4
 $$
\begin{cases} 
\text{for $n$ even,}&\underbrace{\{0,2\}}_{\det B >0}+n=\underbrace{\{0,2\}}_{\det(-B)>0},
\qquad \underbrace{\{1,3\}}_{\det B<0}+n=\underbrace{\{1,3\}}_{\det (-B)<0},
\\
\text{for $n$ odd,}&\underbrace{\{0,2\}}_{\det B >0}+n=\underbrace{\{1,3\}}_{\det(-B)<0},
\qquad \underbrace{\{1,3\}}_{\det B<0}+n=\underbrace{\{0,2\}}_{\det (-B)>0}.
 \end{cases}
 $$
 We have also 
 $
 m(L)-n\in \{m(-L), m(-L)+2\}
 $
 since we know already (from the above in that footnote)
 that  $m(L)-n\in\{m(-L),m(-L)+2\}-2n$,
which gives 
 $m(L)-n\in\{m(-L),m(-L)+2\}$ for $n$ even; for $n=2l+1$ odd we get
 the same result since 
 $$
 m(L)-n\in\{m(-L),m(-L)+2\}-4l-2=\{m(-L)-2,m(-L)\}=\{m(-L)+2,m(-L)\}.
 $$
 }
 \begin{equation}\label{linkme}
 M_{A,B,C}^{\scriptscriptstyle\{m(B)\}}=\mathcal M_{A,-B,C}^{\scriptscriptstyle\{m(B)+n\}}\mathcal F e^{-i\pi n/4},\quad \mathcal M_{P,L,Q}^{\scriptscriptstyle\{m(L)\}}=M_{P,-L,Q}^{\scriptscriptstyle\{m(L)-n\}}
 \left(\mathcal F e^{-i\pi n/4}\right)^{-1},
\end{equation}
\end{remark}
The following lemma is Lemma A.2.4 in \cite{MR4726947}.
\begin{lemma}\label{lem.kfd89e} With the homomorphism $\Psi$ defined in \eqref{hqma28} and defining 
\begin{equation}\label{trua88}
\Lambda_{P,L,Q}=\Psi(\mathcal M_{P,L,Q}\expo{m}),
\end{equation}
we find that 
\begin{equation}\label{327327}
\Lambda_{P,L,Q}=\mat22{L^{-1}Q}{\hs L^{-1}}{PL^{-1}Q-L^{*}}{\hs PL^{-1}}.
\end{equation}
\end{lemma}
\begin{claim}\label{claim.911}
   Let $P, L, Q, m$ be as in Definition \ref{def.lersym}. Then we have 
   \begin{equation}\label{929929}
(\mathcal M_{P,L,Q}^{\scriptscriptstyle\{m\}})^{-1}=\mathcal M_{-Q,-L^{*},-P}^{\scriptscriptstyle\{n-m\}},
\end{equation}
and moreover $n-m\in m(-L^{*})$.
\end{claim}
\no
Indeed, calculating the kernel $\kappa$ of $\mathcal M_{P,L,Q}^{\scriptscriptstyle\{m\}}
\mathcal M_{-Q,-L^{*},-P}^{\scriptscriptstyle\{n-m\}}$,
we get
\begin{multline*}
\kappa(x,y)
=e^{\frac{i\pi}2(m+n-m-n)}
\val{\det L}
 \int e^{i\pi\{Px^{2}-2Lx\cdot z+Qz^{2}
-Qz^{2}+2L^{*}z\cdot y-Py^{2}\}} dz
\\
= \val{\det L}
e^{i\pi\{Px^{2}-P y^{2}\}} \delta_{0}(Lx-Ly)=
 \delta_{0}(x-y),
\end{multline*}
so that 
$\mathcal M_{P,L,Q}^{\scriptscriptstyle\{m\}}
\mathcal M_{-Q,-L^{*},-P}^{\scriptscriptstyle\{n-m\}}=\Id_{L^{2}(\R^{n})}$ and since 
$\mathcal M_{P,L,Q}\expo{m}$ is unitary, this proves \eqref{929929}.
The last assertion is equivalent to
$
n-m\in m(L)+n,
$
i.e. to $m\in-m(L)$.
Since the  mapping 
$$\Z/4\Z\ni x\mapsto-x\in \Z/4\Z,$$
leaves stable the sets $\{0,2\}, \{1,3\}$,
we obtain the sought result,
concluding the proof of the claim.\qed
\begin{remark}
 We may also note that since \eqref{929929} holds true, having
 $$
 (\mathcal M_{P,L,Q}^{\scriptscriptstyle\{m\}})^{-1}=\Mett{P'}{L'}{Q'}{m'}
 $$
 would imply that the kernels of the above operator should coincide, i.e.
 $$
e^{-\frac{i\pi n}4} e^{\frac{i\pi}2\{n-m\}} e^{i\pi\{-Q x^{2}+2L^{*}x\cdot y-P y^{2}\}}
=e^{-\frac{i\pi n}4} e^{\frac{i\pi}2\{m'\}} e^{i\pi\{P'x^{2}-2L'x\cdot y+Q' y^{2}\}},
 $$
 implying readily that
 $$
 P'=-Q, L'=-L^{*}, Q'=-P, \quad \pi\frac{(n-m-m')}{2}\in 2\pi \Z,
 $$
 and thus $n-m=m'\mod 4$.
 More generally the mapping \eqref{311311}
 $$
 \text{\sl Sym}(n,\R)\times \text{\it Gl}(n,\R)
 \times  \text{\sl Sym}(n,\R)\times \Z_{4}\ni(P,L,Q,m)\mapsto \Mett{P}{L}{Q}{m}\in  \mathcal U\bigl(L^{2}(\R^{n})
 \bigr),
 $$
 is injective since for a symmetric $(2n)\times (2n)$ matrix $\mathcal A$ ,
 \begin{multline}\label{quafor}
\forall X\in \RZ, 
e^{\frac{i\pi}2\{\mathcal A X^{2}+\mu\}}=1\overbrace{\Longrightarrow}^{\text{\tiny differentiation}}
\forall X\in \RZ, e^{\frac{i\pi}2\{\mathcal A X^{2}+\mu\}}\mathcal AX=0
\\\Longrightarrow
\forall X\in \RZ, \mathcal A X=0 
 \quad\text{i.e. }\mathcal A=0,\quad\text{and thus $\mu\in 4\Z$.}
\end{multline}
 \end{remark} 
 The next proposition is Proposition A.2.10 in \cite{MR4726947}.
\begin{proposition}\label{pro.9159++}
 The metaplectic group $\text{\sl Mp}(2n)$ is equal to the set
\begin{equation}\label{}
 \left\{\mathcal M_{P_{1}, L_{1}, Q_{1}}\expo{m_{1}}\mathcal M_{P_{2}, L_{2}, Q_{2}}\expo{m_{2}}\right\}_{
 \substack
 {P_{j}=P_{j}^{*},\hskip0.6pt Q_{j}=Q_{j}^{*}\\m_{j}\in m(L_{j}),\hskip0.6pt  \det L_{j}\not=0}
 }.
\end{equation}
In other words, every metaplectic operator of $\text{\sl Mp}(2n)$ is the product of two operators of type 
$\mathcal M_{P, L, Q}\expo{m}$ as given by Definition \ref{def.lersym}.
\end{proposition}
The next theorem is Theorem A.2.11 in \cite{MR4726947}.
\begin{theorem}\label{thm.phaseb}
 Let $M$ be an element of $\text{\sl Mp}(n)$ such that $M=e^{i\phi}\Id_{L^{2}(\R^{n})}, \phi\in \R.$
 Then $e^{i\phi}$ belongs to the set $\{-1,1\}$. In other words, the intersection of the metaplectic group with the unit circle (identified to the unitary operators in $L^{2}(\R^{n})$ defined by the mappings $v\mapsto z v$ where $z\in\mathbb S^{1}\subset {\mathbb C}$)
 is reduced to the set $\{-1,1\}$.
 \end{theorem}
We may go back to the description given by Proposition \ref{pro.kj77qq} and Definition \ref{def.41fdhh}.
The next proposition is Proposition A.2.12 in \cite{MR4726947}.
\begin{proposition}\label{pro.917917}
 The metaplectic group $\text{\sl Mp}(n)$ is equal to the set
\begin{equation}\label{}
 \left\{M_{A_{1},B_{1},C_{1}}\expo{m_{1}}
 M_{A_{2},B_{2},C_{2}}\expo{m_{2}}\right\}_{
 \substack
 {A_{j}=A_{j}^{*}, C_{j}=C_{j}^{*}\\ m_{j}\in m(B_{j}),\       \det B_{j}\not=0}
 }
\end{equation}
In other words, every metaplectic operator of $\text{\sl Mp}(n)$ is the product of two operators of type 
$M_{A, B, C}\expo{m}$ as given by Proposition \ref{pro.kj77qq}.
\end{proposition}
\subsection{Weyl quantization}\label{sec.weyl}
\subsubsection{Wigner distribution}
Let $u,v$ be given functions in $L^2(\R^n)$.
The function $\Omega$, defined on $\R^n\times \R^n$ by
\begin{equation}\label{funcome}
\R^{n}\times \R^{n}\ni(z,x)\mapsto u(x+\frac z2) \bar v(x-\frac z2) =\Omega(u,v)(x,z),
\end{equation}
belongs to $L^{2}(\RZ)$ from the identity
\begin{equation}\label{newide}
\int_{\RZ}{\val{\Omega(u,v)(x,z)}}^{2} dxdz=\norm{u}^{2}_{L^{2}(\R^{n})}\norm{v}^{2}_{L^{2}(\R^{n})}.
\end{equation}
We have also
\begin{equation}\label{113uhb}
\sup_{x\in \R^n}\int_{\R^n}\val{\Omega(x,z)} dz\le 2^n \norm{u}_{L^{2}(\R^{n})}\norm{v}_{L^{2}(\R^{n})}.
\end{equation}
We may then give the following definition.
\begin{definition}\label{wigne+}
 Let $u,v$ be given functions in $L^2(\R^n)$. We define the joint Wigner distribution
 $\mathcal W(u,v)$ 
as the partial Fourier transform
with respect to $z$ of the function $\Omega$ defined in
\eqref{funcome}.
We have for  $(x,\xi)\in\R^n_{x}\times \R^n_{\xi}$, using \eqref{113uhb},
\index{$\mathcal W(u,v)$}
\begin{equation}\label{wigner}
\mathcal W(u,v)(x,\xi)=\int_{\R^n} e^{-2i\pi z\cdot \xi} u(x+\frac z2) \bar v(x-\frac z2) dz.
\end{equation}
The Wigner distribution of $u$ is defined as $\mathcal W(u,u)$.
\end{definition}
\begin{nb}
 By inverse Fourier transformation we get, in a weak sense,
 \index{reconstruction formula}
  \begin{equation}\label{623new}
u(x_{1})\otimes \bar v(x_{2})=\int \mathcal W(u,v)(\frac{x_{1}+x_{2}}2, \xi)\ e^{2i\pi(x_{1}-x_{2})\cdot \xi}d\xi.
\end{equation}
\end{nb}
\begin{lemma}
 Let $u,v$ be given functions in $L^2(\R^n)$.  The function 
 $\mathcal W(u,v)$  belongs to $L^2(\RZ)$ and we have
 \begin{equation}\label{norm}
\norm{\mathcal  W(u,v)}_{L^{2}(\RZ)}=\norm{u}_{L^{2}(\R^{n})}\norm{v}_{L^{2}(\R^{n})}.
\end{equation} 
We have  also
\begin{equation}\label{complex}
\overline{\mathcal W (u,v)(x,\xi)}=\mathcal W (v,u)(x,\xi),
\end{equation}
so that 
 $\mathcal W(u,u)$ is real-valued.
\end{lemma}
\begin{proof}
Note that the function $\mathcal W(u,v)$ is in $L^{2}(\RZ)$ and satisfies
\eqref{norm}
from \eqref{newide} and the definition of $\mathcal W$ as the partial Fourier transform of $\Omega$.
Property \eqref{complex} is immediate and entails that $\mathcal W(u,u)$ is real-valued.
\end{proof}
\begin{remark}\label{rem13}
\textrm
We note also that the real-valued function $\mathcal W (u,u)$ can take negative values,
choosing for instance $u_{1}(x)=x e^{-\pi x^{2}}$ on the real line, we get
\[
\mathcal W (u_{1},u_{1})(x,\xi)=2^{1/2} e^{-2\pi(x^{2}+\xi^{2})}\Bigl(x^{2}+\xi^{2}-\frac1{4\pi}\Bigr).
\]
\end{remark}
\subsubsection{The Weyl quantization}
The main goal of Hermann Weyl in his seminal paper \cite{MR0450450} was to give a simple formula, also providing symplectic covariance, ensuring that  real-valued Hamiltonians $a(x,\xi)$ get quantized by formally self-adjoint operators.  The standard way of dealing with differential operators does not achieve that goal since for instance the standard quantization of the Hamiltonian
$x\xi$ (indeed real-valued) is the operator $xD_{x}$, which is not symmetric 
($D_{x}=\frac1{2i\pi}\frac{\p}{\p x})$; H.~Weyl's choice in that case was
\[
x\xi\ \text{should be quantized by  the operator\quad} \frac{1}{2}(xD_{x}+D_{x}x),
\quad\text{(indeed symmetric),}
\]
and more generally, say for 
$a\in \mathscr S(\RZ), u\in \mathscr S(\R^n)$, the quantization of the Hamiltonian $a(x,\xi)$, denoted by $\opw{a}$, 
should be 
given by the formula
\index{$\opw{a}$}
\begin{equation}
(\opw{a} u)(x)=\iint e^{2i\pi (x-y)\cdot \xi} a\Bigl(\frac{x+y}2, \xi\Bigr) u(y) dy d\xi.
\end{equation}
For $v\in\mathscr S(\R^n)$, we may consider
\begin{multline*}
\poscal{\opw{a} u}{v}_{L^2(\R^n)}
=\iiint a(x,\xi)e^{-2i\pi z\cdot \xi} 
u(x+\frac z2)\bar v(x-\frac z2) dzdxd\xi 
\\
=\iint_{\R^n\times \R^n} a(x,\xi) \mathcal W(u,v)(x,\xi) dx d\xi,
\end{multline*}
and the latter formula allows us to give the following definition.
\begin{definition}
Let $a\in \mathscr S'(\RZ)$. We define the Weyl quantization $\opw{a}$ of the Hamiltonian $a$, by the formula
\begin{equation}\label{weylq}
(\opw{a}u)(x)=\iint e^{2i\pi (x-y)\cdot \xi} a\Bigl(\frac{x+y}2, \xi\Bigr) u(y) dy d\xi,
\end{equation}
to be understood weakly as 
\begin{equation}\label{eza654}
\poscal{\opw{a}u}{\bar v}_{\mathscr S'(\R^{n}), \mathscr S(\R^{n})}=\poscal{a}{\mathcal W(u,v)}_{\mathscr S'(\RZ), \mathscr S(\RZ)}.
\end{equation}
\end{definition}
We note that the sesquilinear  mapping 
\[
\mathscr S(\R^{n})\times \mathscr S(\R^{n})\ni(u,v)\mapsto \mathcal W(u,v) \in \mathscr S(\RZ),
\]
is continuous so that the above bracket of duality $
\poscal{a}{\mathcal W(u,v)}_{\mathscr S'(\R^{2n}), \mathscr S(\R^{2n})}
$
makes sense.
We note as well that a temperate distribution $a\in\mathscr S'(\RZ)$ gets quantized
by a continuous  operator $\opw a$ from 
$\mathscr S(\R^{n})$ 
into
$\mathscr S'(\R^{n})$.
This very general framework  is not really useful since we want to compose our operators $\opw{a}\opw{b}$.
A first step in this direction is to look for sufficient conditions
ensuring that the operator $\opw{a}$ is bounded on $L^{2}(\R^{n})$. 
Moreover, for $a\in \mathscr S'(\RZ)$ and $b$ a \emph{polynomial} in ${\mathbb C}[x,\xi]$,
we have the composition formula,
\begin{align}
\opw{a}\opw{b}&=\opw{a\sharp b},\label{gfcd44}
\\
(a\sharp b)(x,\xi)&= \sum_{k\ge 0}\frac1{(4i\pi)^{k}}\sum_{\val \alpha+\val \beta=k}
\frac{(-1)^{\val \beta}}{\alpha!\beta!}(\p_{\xi}^{\alpha}\p_{x}^{\beta}a)(x,\xi)
(\p_{x}^{\alpha}\p_{\xi}^{\beta}b)(x,\xi),\label{gfcd44+}
\end{align}
which involves here a finite sum.
This follows from (2.1.26) in \cite{MR2599384}
where several generalizations can be found (see in particular in that reference the integral formula (2.1.18) which can be given a meaning for quite general classes of symbols).
We shall use the following notation
\begin{equation}\label{}
 \omega_{k}(a,b)=\frac1{(4i\pi)^{k}}\sum_{\val \alpha+\val \beta=k}
\frac{(-1)^{\val \beta}}{\alpha!\beta!}(\p_{\xi}^{\alpha}\p_{x}^{\beta}a)(x,\xi)
(\p_{x}^{\alpha}\p_{\xi}^{\beta}b)(x,\xi).
\end{equation}
As a consequence of \eqref{gfcd44+}, we get that 
\begin{align}
(a\sharp b)&=\sum_{k\ge 0}\omega_{k}(a,b), \quad \omega_{0}(a,b)=ab, \quad \omega_{1}(a,b)=\frac1{4i\pi}
\poi{a}{b},
\label{gfl227}\\
\poi{a}{b}&=\p_{\xi}a\cdot \p_{x}b-\p_{x}a\p_{\xi}b,
\end{align}
where $\poi{a}{b}$ is called the \emph{Poisson bracket} of $a$ and $b$.
\begin{claim}\label{cla.lin001}
 Let $L$ be a linear form on $\RZ$ and let $N\ge 1$ be an integer. Then we have 
 \begin{equation}\label{for545}
\underbrace{L\sharp \dots \sharp L}_{\text{\rm $N$ factors}}=L^{N}.
\end{equation}
\end{claim}
\begin{proof}[Proof of the claim]
 The formula is true for $N=1,2,$ from Formulas \eqref{gfcd44}, \eqref{gfcd44+} since 
 $\poi{L}{L}=0$. Let $N\ge 2$ such that \eqref{for545} holds true.
 Then we calculate
 $$
 \underbrace{L\sharp \dots \sharp L}_{\text{\rm $N$ factors}}\sharp L=L^{N}\sharp L
 =L^{N}L+\frac1{4i\pi}\poi{L^{N}}{L}=L^{N+1},
 $$
 concluding the induction proof.
\end{proof}
The next proposition is Proposition 1.2.2 in \cite{MR4726947}.
\begin{proposition}\label{pro1717}
Let $a$ be a tempered distribution on $\RZ$. Then we have 
 \begin{equation}\label{norm01}
\norm{\opw{a}}_{\mathcal B(L^{2}(\R^{n}))}\le \min\bigl(2^{n}\norm{a}_{L^{1}(\RZ)}, \norm{\hat a}_{L^{1}(\RZ)}\bigr).
\end{equation}
\end{proposition}
The next statement is Claim 1.2.3 in \cite{MR4726947}.
\begin{claim}\label{foot2}
Let $(x,\xi)$ be in $\R^{n}\times \R^{n}$.
We define 
the operator $\sigma_{x,\xi}$ (called \emph{phase symmetry}), on $\mathscr S'(\R^{n})$
by
\begin{equation}\label{phsymm}
(\sigma_{x,\xi} u)(y)=u(2x-y) e^{-4i\pi(x-y)\cdot \xi}.
\end{equation}
Then 
$\sigma_{x,\xi}$
is unitary and self-adjoint on $L^{2}(\R^{n})$.
We have in particular
\begin{equation}\label{sigma0}
(\sigma_{0,0}w)(y)=w(-y).
\end{equation}
The operator $\sigma_{0,0}$ will be denoted by $\sigma_{0}$.
\end{claim}
\begin{proposition}\label{pro1716}
Let $a\in \mathscr S'(\RZ)$. The distribution kernel $k_{a}(x,y)$ of the operator $\opw{a}$ is
\begin{equation}\label{reteey}
k_{a}(x,y)=\widehat a^{[2]}(\frac{x+y}2, y-x),
\end{equation}
where $a^{[2]}$ stands for the Fourier transform of $a$ with respect to the second variable.
Let $k\in \mathscr S'(\RZ)$ be the distribution kernel of a continuous operator $A$ from $\mathscr S(\R^n)$ into   $\mathscr S'(\R^n)$. Then the Weyl 
symbol $a$
of $A$ is 
\begin{equation}\label{}
a(x,\xi)= \int e^{-2\pi i t\cdot \xi}
k\bigl(x+\frac t2,x-\frac t2\bigr) dt,
\end{equation}
where the integral sign means  that we take the Fourier transform with respect to $t$
of the distribution $k(x+\frac t2,x-\frac t2)$ on $\RZ$.
\end{proposition}
\begin{proof}
 With $u,v\in \mathscr S(\R^n)$, we have defined $\opw{a}$ via Formula \eqref{eza654} 
 and using  Definition \ref{wigne+}, we get
 \begin{align*}
 \poscal{\opw{a}u}{\bar v}_{\mathscr S'(\R^{n}), \mathscr S(\R^{n})}&=\poscal{a(x, \xi)}{
\widehat{\Omega}^{[2]}(x,\xi)
 }_{\mathscr S'(\RZ), \mathscr S(\RZ)}
\\&=
\poscal{\widehat a^{[2]}(t, z)}{
u(t+\frac z2) \bar v(t-\frac z2) 
 }_{\mathscr S'(\RZ), \mathscr S(\RZ)}
 \\&=
\poscal{\widehat a^{[2]}(\frac{x+y}2, y-x)}{
u(y) \bar v(x) 
 }_{\mathscr S'(\RZ), \mathscr S(\RZ)},
\end{align*}
proving \eqref{reteey}. As a consequence, we find that 
$
k_{a}(x+\frac t2,x-\frac t2)={\widehat a^{[2]}}(x,-t),
$
and by Fourier inversion, this entails
\begin{equation}\label{symker}
a(x,\xi)=\text{Fourier}_{t}\bigl(k_{a}(x+\frac t2,x-\frac t2)\bigr)(\xi)=\int e^{-2\pi i t\cdot \xi}
k_{a}(x+\frac t2,x-\frac t2) dt,
\end{equation}
where the integral sign means  that we perform a Fourier transformation with respect to the variable $t$.
 \end{proof}
\subsubsection{Symplectic covariance of the Weyl calculus}
\begin{theorem}[Symplectic covariance of the Weyl calculus]\label{thm.3216++}
 Let $a$ be in $\mathscr S'(\RZ)$ and let $\chi$ be in $\text{\sl Sp}(2n,\R)$. Then for a metaplectic operator $M$ such that $\Psi(M)=\chi$, we have  
 \index{Segal formula}
 \begin{equation}\label{segal}
M^{*}\ \opw{a} M=\opw{a\circ \chi}.
\end{equation}
For $u,v\in \mathscr S(\R^{n}),$
we have 
\begin{equation}\label{segal+}
\mathcal W\left(M u, Mv\right)=\mathcal W(u,v)\circ \chi^{-1},
\end{equation}
where $\mathcal W$ is the Wigner distribution given in \eqref{wigner}.
\end{theorem}
\begin{proof}
 The first property follows  immediately from \eqref{1239} and 
 Definition \ref{def.41fdhh},
 whe\-reas \eqref{segal+} is a consequence of \eqref{eza654} and \eqref{segal}.
\end{proof}
We note also that for $Y=(y,\eta)\in \RZ$, the symmetry $S_{Y}$ is defined by
$S_{Y}(X)=2Y-X$ and is quantized by the phase symmetry $\sigma_{Y}$ as defined by \eqref{phsymm}
with the formula
\begin{equation}\label{segalsym}
\opw{a\circ S_{Y}}=\sigma_{Y}^{*}\opw{a}\sigma_{Y}=\sigma_{Y}
\opw{a}\sigma_{Y}.
\end{equation}
\index{phase translation}
\index{{~\bf Notations}!$\tau_{y,\eta}$}
Similarly, 
the translation $T_{Y}$ is defined on the phase space by $T_{Y}(X)=X+Y$ and is quantized by the {\it phase translation} $\tau_{Y}$,
\begin{equation}\label{phtrans}
(\tau_{(y,\eta)} u)(x)=u(x-y) e^{2i\pi(x-\frac y2)\cdot\eta},
\end{equation}
and we have
\begin{equation}\label{1224}
\opw{a\circ T_{Y}}=\tau_{Y}^{*}\opw{a}\tau_{Y}=\tau_{-Y}
\opw{a}\tau_{Y}.
\end{equation}
\section{Proofs of the main results}
\subsection{First cases: \texorpdfstring{$\Xi_{12}$}{xi12} invertible}
\begin{proposition}\label{pro.baba}
 Let $\mathcal M_{P,L,Q}\expo{m}$ be a metaplectic operator
as given by Definition \ref{def.lersym}.
Then with the variance defined by \eqref{variance}, we have for all $u\in \mathscr S(\R^{n})$, 
\begin{equation}\label{}
\sqrt{\mathcal V(\mathcal M u)\mathcal V( u)}\ge \frac{\norm{u}_{L^{2}(\R^{n})}^{2}}{4\pi}
\trace\left((L^{*-1} L^{-1})^{1/2}\right)= \frac{\norm{u}_{L^{2}(\R^{n})}^{2}}{4\pi}\trinorm{L^{-1}}_{{\mathtt{S}^{1}}},
\end{equation}
where 
$\trinorm{L^{-1}}_{{\mathtt{S}^{1}}}$ is the trace-norm of $L^{-1}$, that is the sum of the eigenvalues of $\val{L^{-1}}$
(see Section \ref{sec.11hhff} in our Appendix).
\end{proposition}
\begin{proof}[Proof of the Proposition]
 We start with a lemma.
\begin{lemma}\label{lem.241+++}
Let $\mathcal M_{P,L,Q}\expo{m}$ be a metaplectic operator
as given by Definition \ref{def.lersym}.
Let $a, b$ be linear forms on the phase space given by 
\eqref{233rez}.
Then we have,
with 
$
\Lambda_{P,L,Q}
$
defined in \eqref{327327},
\begin{multline}\label{}
2\re\Poscal{ (\zeta\cdot y-z\cdot D_{y})
(\mathcal M_{P,L,Q}\expo{m}u)(y)}{i \mathcal M_{P,L,Q}\expo{m}
\bigl(\tau\cdot x-t \cdot D_{x})u
\bigr)}_{L^{2}(\R^{n})}
\\
=2\re\Poscal{ \opw{a}\mathcal M_{P,L,Q}\expo{m}u}{i \mathcal M_{P,L,Q}\expo{m}\opw{b} u}
\\
=\frac1{2\pi}\Poi{a\circ \Lambda_{P,L,Q}}{b}\norm{u}_{L^{2}(\R^{n})}^{2}
\\
=\frac1{2\pi}
\Bigl[
(\zeta-Pz)\cdot\bigl(
L^{-1}(\tau+Qt)
\bigr)
+Lz\cdot t
\Bigr]
\norm{u}_{L^{2}(\R^{n})}^{2}.
\end{multline}
 \end{lemma}
 \begin{nb}
 We know from Proposition \ref{pro.9159++} that every element in the metaplectic group is the product
of two operators of type $\mathcal M_{P,L,Q}\expo{m}$ as given in Definition
\ref{def.lersym}  .
Of course it is \emph{not} true that all metaplectic operators are of type $\mathcal M_{P,L,Q}\expo{m}$
since the Identity itself has the singular kernel $\delta(x-y)$ which is different from all the imaginary Gaussian functions serving as kernels of the operators $\mathcal M_{P,L,Q}\expo{m}$.
\end{nb}
\begin{proof}[Proof of the lemma]
Let $\mathcal M_{P,L,Q}\expo{m}$ be a metaplectic operator
as given by Definition \ref{def.lersym}.
We would like to calculate for $u\in \mathscr S(\R^{n})$,
$a, b$ real valued linear forms on $\R^{n}_{x}\times \R^{n}_{\xi}$,
with $L^{2}(\R^{n})$ dot-products,
\begin{align}
2\re&\Poscal{ \opw{a}\mathcal M_{P,L,Q}\expo{m}u}{i \mathcal M_{P,L,Q}\expo{m}\opw{b} u}
\notag\\&=2\re\Poscal{ (\mathcal M_{P,L,Q}\expo{m})^{\hskip-1pt*} \opw{a}\mathcal M_{P,L,Q}\expo{m}u}{i\opw{b} u}
\notag\\\text{\scriptsize (from Lemma \ref{lem.kfd89e})}&=
2\re\Poscal{\opw{a\circ \Lambda_{P,L,Q}}u}{i\opw{b} u}
\notag\\
&=\Poscal{\bigl[\opw{a\circ \Lambda_{P,L,Q}}, i\opw{b} \bigr]u}{u}
\notag\\
&=\frac1{2\pi}\Poscal{\OPW{\poi{a\circ \Lambda_{P,L,Q}}{b}} u}{u}.
\label{tar654}
\end{align}
We may assume\footnote{We may keep in mind that the phase space is $\Phi=E\oplus E^{*}$
(here $E$ is the \emph{configuration space}, a real finite dimensional vector space) and that the symplectic form $\sigma$ is bilinear alternate on $\Phi$ with
\begin{equation}\label{233233}
\poscal{\sigma X}{Y}_{\Phi^{*}, \Phi}=[X,Y]=\poscal{\xi}{y}_{E^{*},E}-\poscal{\eta}{x}_{E^{*},E}, \quad X=x\oplus\xi, Y=y\oplus \eta,
\end{equation}
in such a way that  our notation $\xi\cdot y-\eta\cdot x$ is the difference of two brackets of duality.
We have also that the kernel of the operator $\mathcal M_{P,L,Q}\expo{m}$, given by \eqref{322322}
is
$$
e^{-\frac{i\pi n}4} e^{\frac{i\pi m(L)}{2}}{\val{\det L}}^{1/2} e^{i\pi\left\{\poscal{Px}{x}_{E^{*},E}-2\poscal{Lx}{y}_{E^{*},E}+\poscal{Qy}{y}_{E^{*},E}\right\}},
$$
with 
\begin{equation}\label{232232}
P:E\longrightarrow E^{*}, \ P=P^{*}, \qquad Q:E\longrightarrow E^{*}, \ Q=Q^{*},\qquad
L:E\longrightarrow E^{*} \text{ invertible.}
\end{equation}
}
 that, 
with $Y=(y,\eta)$, $Z=(z,\zeta)$, $T=(t,\tau)$, $X=(x,\xi)$, in $\R^{n}\times\R^{n}$,
\begin{equation}\label{233rez}
a(Y)=[Z,Y]=\zeta\cdot y-\eta\cdot z, \qquad b(X)=[T,X]=\tau\cdot x-\xi\cdot t.
\end{equation}
We have 
\begin{equation}\label{234234}
(a\circ \Lambda)(X)=[Z, \Lambda X]=[\Lambda^{-1}Z, X],
\end{equation}
and since $\Lambda$ is a symplectic mapping we have $\Lambda^{*}\sigma\Lambda=\sigma$,
that is from \eqref{327327} and \eqref{232232}
\begin{align*}
\Lambda^{-1}=\sigma^{-1}\Lambda^{*}\sigma&=
\mat22{0}{-I}{I}{0}
\mat22{QL^{*-1}}{QL^{*-1}P-L}{\hs L^{*-1}}{\hs L^{*-1}P}
\mat22{0}{I}{-I}{0}
\\&=
\mat22{0}{-I}{I}{0}
\mat22{L-QL^{*-1}P}{\hs QL^{*-1}}{-L^{*-1}P}{L^{*-1}}
\\&=\mat22{L^{*-1}P}{\hs -L^{*-1}}{L-QL^{*-1}P}{QL^{*-1}},
\end{align*}
and thus \eqref{234234} gives 
\begin{multline}\label{23523k}
(a\circ \Lambda)(X)=\Bigl[\Bigl(L^{*-1}Pz-L^{*-1}\zeta,
(L-QL^{*-1}P)z+QL^{*-1}\zeta
\Bigr), (x,\xi)\Bigr]
\\=\Bigl((L-QL^{*-1}P)z+QL^{*-1}\zeta\Bigr)\cdot x
-\xi\cdot\Bigl(L^{*-1}Pz-L^{*-1}\zeta\Bigr).
\end{multline}
We calculate now the Poisson bracket,
using \eqref{233rez}, \eqref{23523k},
\begin{align}\label{}
\Poi{a\circ \Lambda}{b}&=\bigl(-L^{*-1}Pz+L^{*-1}\zeta\bigr)\cdot \tau
+\bigl((L-QL^{*-1}P)z+QL^{*-1}\zeta\bigr)\cdot t
\\
&=\tau\cdot L^{*-1}(\zeta-Pz)
+\bigl(Lz+QL^{*-1}(\zeta-Pz)\bigr)\cdot t
\notag\\
&=\underbrace{(\zeta-Pz)}_{E^{*}}\cdot\bigl(
\underbrace{L^{-1}\underbrace{(\tau+Qt)}_{E^{*}}}_{E}
\bigr)
+\underbrace{Lz}_{E^{*}}\cdot \underbrace{t}_{E},
\notag
\end{align} 
which does not depend on the variable $X$ as the Poisson bracket of two linear forms in $X$.
We get thus from \eqref{tar654}
\begin{multline}\label{}
2\re\Poscal{ \opw{a}\mathcal M_{P,L,Q}\expo{m}u}{i \mathcal M_{P,L,Q}\expo{m}\opw{b} u}_{L^{2}(\R^{n})}
\\
=\frac1{2\pi}
\Bigl[
(\zeta-Pz)\cdot\bigl(
L^{-1}(\tau+Qt)
\bigr)
+Lz\cdot t
\Bigr]
\norm{u}_{L^{2}(\R^{n})}^{2},
\end{multline}
that is 
\begin{multline}\label{}
2\re\Poscal{ (\zeta\cdot y-z\cdot D_{y})
(\mathcal M_{P,L,Q}\expo{m}u)(y)}{i \mathcal M_{P,L,Q}\expo{m}
\bigl(\tau\cdot x-t \cdot D_{x})u
\bigr)}_{L^{2}(\R^{n})}
\\
=\frac1{2\pi}
\Bigl[
(\zeta-Pz)\cdot\bigl(
L^{-1}(\tau+Qt)
\bigr)
+Lz\cdot t
\Bigr]
\norm{u}_{L^{2}(\R^{n})}^{2},
\end{multline}
concluding the proof of Lemma \ref{lem.241+++}.
\end{proof}
Let us go on with the proof of Proposition \ref{pro.baba}.
 Choosing $z=t=0$ in Lemma \ref{lem.241+++}, we get 
 \begin{equation*}
 2\re\Poscal{ (\zeta\cdot y)
(\mathcal M_{P,L,Q}\expo{m}u)(y)}{i \mathcal M_{P,L,Q}\expo{m}
\bigl((\tau\cdot x)u
\bigr)}_{L^{2}(\R^{n})}
=\frac1{2\pi}(\zeta\cdot L^{-1}\tau)\norm{u}_{L^{2}(\R^{n})}^{2},
\end{equation*}
which implies
that 
\begin{equation}\label{key001}
\norm{ (\zeta\cdot y)
(\mathcal M_{P,L,Q}\expo{m}u)(y)}_{L^{2}(\R^{n})}
\norm{(\tau\cdot x)u}_{L^{2}(\R^{n})}
\ge \frac1{4\pi}\val{(\zeta\cdot L^{-1}\tau)}\norm{u}_{L^{2}(\R^{n})}^{2}.
\end{equation}
\begin{nb}
 We note at this stage of the proof, that Inequality 
 \eqref{key001} is already providing a non-trivial estimate by choosing $\zeta=L^{-1}\tau$: we get 
 \begin{equation*}
 \norm{L^{-1}\tau}\norm{\tau}
\norm{\val y
(\mathcal M_{P,L,Q}\expo{m}u)(y)}_{L^{2}(\R^{n})}
\norm{\val xu}_{L^{2}(\R^{n})}
\ge \frac1{4\pi}\norm{L^{-1}\tau}^{2}\norm{u}_{L^{2}(\R^{n})}^{2},
\end{equation*}
implying for any $\tau$,
\begin{equation*}
\sqrt{\mathcal V\bigl(\mathcal M_{P,L,Q}\expo{m}u\bigr)\mathcal V\bigl(u\bigr)}
\ge \frac1{4\pi}\frac{\norm{L^{-1}\tau}}{\norm{\tau}}\norm{u}_{L^{2}(\R^{n})}^{2},
\end{equation*}
and thus
\begin{equation}\label{key002}
\mu\bigl(\mathcal M_{P,L,Q}\expo{m}\bigr)
\ge \frac1{4\pi}\norm{L^{-1}}=\frac1{4\pi}\trinorm{L^{-1}}_{{\mathtt{S}^{\io}}}.
\end{equation}
However, we are looking for a  better estimate, and we want to  replace $\trinorm{L^{-1}}_{{\mathtt{S}^{\io}}}$ in the right-hand-side (just the operator-norm) by  the much larger
Schatten norm of index 1, 
$\trinorm{L^{-1}}_{{\mathtt{S}^{1}}}$
(see \eqref{3143ss} in our Appendix);
as a matter of fact,
we shall see below that this will give the equality 
$$
\mu\bigl(\mathcal M_{P,L,Q}\expo{m}\bigr)=\frac1{4\pi}\trinorm{L^{-1}}_{{\mathtt{S}^{1}}}.
$$
\end{nb}
Let  $(\zeta_{j})_{1\le j\le n}$,
$(\tau_{j})_{1\le j\le n}$ be two orthonormal basis of $\R^{n}$ for the canonical Euclidean structure on $\R^{n}$.
We find from \eqref{key001} that for each $j\in \llbracket 1,n \rrbracket$,
\begin{equation*}
\left(\int {\val{\zeta_{j}\cdot y}}^{2}{\val{(\mathcal M u)(y)}}^{2} dy\right)^{1/2}
\left(\int \val{\tau_{j}\cdot x}^{2}{\val{u(x)}}^{2} dx\right)^{1/2}\ge \frac{\norm{u}_{L^{2}(\R^{n})}^{2}}{4\pi}
\val{\poscal{\zeta_{j}}{L^{-1}\tau_{j}}},
\end{equation*}
so that
\begin{multline*}
\sum_{1\le j\le n}\left(\int {\val{\zeta_{j}\cdot y}}^{2}{\val{(\mathcal M u)(y)}}^{2} dy\right)^{1/2}
\left(\int {\val{\tau_{j}\cdot x}}^{2}{\val{u(x)}}^{2} dx\right)^{1/2}
\\\ge \frac{\norm{u}_{L^{2}(\R^{n})}^{2}}{4\pi}\sum_{1\le j\le n}
\val{\poscal{\zeta_{j}}{L^{-1}\tau_{j}}},
\end{multline*}
which implies that 
\begin{multline*}
\left(\sum_{1\le j\le n}\int {\val{\zeta_{j}\cdot y}}^{2}{\val{(\mathcal M u)(y)}}^{2} dy\right)^{1/2}
\left(\sum_{1\le j\le n}\int {\val{\tau_{j}\cdot x}}^{2}{\val{u(x)}}^{2} dx\right)^{1/2}
\\\ge \frac{\norm{u}_{L^{2}(\R^{n})}^{2}}{4\pi}\sum_{1\le j\le n}
\val{\poscal{\zeta_{j}}{L^{-1}\tau_{j}}},
\end{multline*}
i.e.
\begin{multline*}
\left(\int {\val{y}}^{2}{\val{(\mathcal M u)(y)}}^{2} dy\right)^{1/2}
\left(\int {\val{x}}^{2}{\val{u(x)}}^{2} dx\right)^{1/2}
\ge \frac{\norm{u}_{L^{2}(\R^{n})}^{2}}{4\pi}\sum_{1\le j\le n}
\val{\poscal{\zeta_{j}}{L^{-1}\tau_{j}}},
\end{multline*}
so that
\begin{equation}\label{pre456}
\mathcal V(\mathcal M u)\mathcal V( u)\ge \frac{\norm{u}_{L^{2}(\R^{n})}^{4}}{4^{2}\pi^{2}}
\left(\sum_{1\le j\le n}
\val{\poscal{\zeta_{j}}{L^{-1}\tau_{j}}}\right)^{2}.
\end{equation}
Thanks to our Proposition
\ref{pro.54128u} in our Appendix
we have  with a unitary matrix $W$ the identity 
$$
\poscal{\zeta_{j}}{L^{-1}\tau_{j}}=\poscal{\zeta_{j}}{ W\val{L^{-1}}\tau_{j}}
=\poscal{W^{*}\zeta_{j}}{ \val{L^{-1}}\tau_{j}}.
$$
The set $(W^{*}\zeta_{j})_{1\le j\le n}$ is also an orthonormal basis and we can choose $\tau_{j}=W^{*}\zeta_{j}$
so that we get 
\[
\sum_{1\le j\le n}
\val{\poscal{\zeta_{j}}{L^{-1}\tau_{j}}}=\sum_{1\le j\le n}\poscal{\val{L^{-1}}\tau_{j}}{\tau_{j}}=\trace\val{L^{-1}},
\]
and this, together with \eqref{pre456}, gives the result of Proposition \ref{pro.baba}.
\end{proof}
Let us now take a look at  a simple example.
\begin{ex}\label{ex.21jhgg}
 Let $n\ge 2$ and let $\mathcal F_{1}$ be the partial Fourier transformation
 with respect to the variable $x_{1}$.
 Then the operator $e^{-i\pi/4}\mathcal F_{1}$ belongs to the metaplectic group $\textsl{Mp}(n)$.
 Indeed, let us consider, with obvious notations, the symmetric $n\times n$ matrix,
 \begin{equation}\label{}
C=\mat22{-1}{0_{1,n-1}}{0_{n-1,1}}{I_{n-1}}, \quad e^{i\pi CD^{2}}= e^{-i\pi D_{1}^{2}} e^{i\pi D_{2}^{2}}, \quad D_{2}=(D_{x_{2}}, \dots, D_{x_{n}}).
\end{equation}
Using the notations \eqref{4.genmet}, \eqref{linkme}, 
we have 
\begin{equation}\label{}
e^{i\pi CD^{2}}=M_{0,I_{n}, C}\expo{0}=\mathcal M_{0,-I_{n},C}^{\scriptscriptstyle\{n\}}\mathcal F e^{-i\pi n/4}\in \textsl{Mp}(n),
\end{equation}
where $\mathcal F$ stands for the (full)
Fourier transform.
We calculate the kernel $\kappa$
of the operator 
$
e^{i\pi CD^{2}}e^{-i\pi/4}\mathcal F_{1}
$
and we find that
\begin{align*}\label{}
\kappa(x,y)&=\int e^{2i\pi(x-z)\cdot \eta} e^{i\pi(-\eta_{1}^{2}+\eta_{2}^{2})}
d\eta_{1}d\eta_{2}
e^{-2i\pi z_{1}y_{1}} \delta(z_{2}-y_{2})
dz_{1}dz_{2}
e^{-i\pi/4}
\\&=
\int e^{2i\pi[(x_{1}-z_{1})\eta_{1}+(x_{2}-y_{2})\cdot \eta_{2}]} e^{i\pi(-\eta_{1}^{2}+\eta_{2}^{2})}
d\eta_{1}d\eta_{2}
e^{-2i\pi z_{1}y_{1}} 
dz_{1}
e^{-i\pi/4}
\\
&=
e^{i\pi[-x_{2}^{2}-2(x_{1}y_{1}-x_{2}y_{2})-y_{1}^{2}-y_{2}^{2}]} e^{\frac{i\pi(n-2)}{4}}.
\end{align*}
With
\begin{equation*}
L=\mat22{1}{0}{0}{-I_{n-1}}, \quad \text{we have } (\det L)^{1/2}=\pm e^{\frac{i\pi(n-1)}{2}},
\end{equation*}
so that 
$$
e^{i\pi CD^{2}}e^{-i\pi/4}\mathcal F_{1}=\mathcal M_{P,L,Q}\expo{m}, \quad P=
\mat22{0}{0}{0}{-I_{n-1}}, Q=-I_{n},
$$
since 
$$
e^{-\frac{i\pi n}4} e^{\frac{i\pi m}2}=e^{\frac{i\pi(n-2)}{4}}\Longleftrightarrow m-\frac n2=\frac n2 -1,\mod 4\Longleftrightarrow m=n-1,\mod 4.
$$
Eventually we have $e^{i\pi(n-1)}=\sign(\det L)$
$$
e^{i\pi CD^{2}}e^{-i\pi/4}\mathcal F_{1}=\mathcal M_{P,L,Q}\expo{m_{L}},\quad m_{L}\in m(L), 
$$
proving that 
$e^{-i\pi/4}\mathcal F_{1}$ belongs to $\textsl{Mp}(n)$, although it is not\footnote{Indeed,
the kernel of $e^{-i\pi/4}\mathcal F_{1}$ is 
$e^{-i\pi/4}
e^{-2i\pi x_{1}y_{1}}\delta_{0,n-1}(x_{2}-y_{2}),
$
which is supported in a vector subspace with dimension $n+1$ in $\R^{2n}$,
whereas the kernel of $\mathcal M_{P,L,Q}\expo{m(L)}$ is the Gaussian function
$$
e^{-\frac{i\pi n}4} e^{\frac{i\pi m(L)}2} e^{i\pi(Px^{2}-2L x\cdot y+Qy^{2})},
$$
whose support is the whole $\RZ$.
} of the  type 
$\mathcal M_{P,L,Q}\expo{m}$.
We check now,
\begin{align*}
&\frac1{2\pi}\norm{u}_{L^{2}(\R^{n})}^{2}=
2\re\poscal{(D_{1}u)(x_{1},x_{2})}{ix_{1} u(x_{1},x_{2})}_{L^{2}(\R^{n})}
\\&=2\re\poscal{\xi_{1}e^{-i\pi/4}\mathcal F_{1}u}{ie^{-i\pi/4}\mathcal F_{1}(x_{1} u)}_{L^{2}(\R^{n})}
\\
&\le2\left(\iint {\val{\xi_{1}}}^{2}\Val{(e^{-i\pi/4}\mathcal F_{1}u)(\xi_{1},x_{2})}^{2}d\xi_{1}dx_{2}\right)^{1/2}
\left(\iint {\val{x_{1}}}^{2}\Val{u(x_{1},x_{2})}^{2}dx_{1}dx_{2}\right)^{1/2}
\\
&\le 2\mathcal V\bigl(e^{-i\pi/4}\mathcal F_{1} u\bigr)^{1/2}\mathcal V (u)^{1/2},
\end{align*}
so that, according to the notation \eqref{mesure++},
\begin{equation}\label{}
\mu(e^{-i\pi/4}\mathcal F_{1})\ge \frac1{4\pi}.
\end{equation}
\end{ex}
\begin{lemma}
 Let $n$ be an integer $\ge 1$ and let $n_{1}, n_{2}$ integers such that $n_{1}\ge 1$ and $n=n_{1}+n_{2}$.
 We define $\mathcal F_{n_{1}}$ as the Fourier transformation with respect to $n_{1}$ variables amongst 
 $\{x_{1},\dots, x_{n}\}$. Then we have, according to the notation \eqref{mesure++}, 
 \begin{equation}\label{2316}
\mu(e^{-i\pi n_{1}/4}\mathcal F_{n_{1}})= \frac{n_{1}}{4\pi}.
\end{equation}
\end{lemma}
\begin{proof}
 The proof of the inequality $\mu(e^{-i\pi n_{1}/4}\mathcal F_{n_{1}})\ge  \frac{n_{1}}{4\pi}$ follows the same lines as  in Example \ref{ex.21jhgg}. We do have with  the notation
 $$
 x=(\underbrace{x_{1}, \dots, x_{n_{1}}}_{x_{[n_{1}]}}, 
 \underbrace{x_{n_{1}+1}, \dots x_{n_{1}+n_{2}}}_{x_{[n_{2}]}}),
 $$
 \begin{align*}
&\frac{n_{1}}{2\pi}\norm{u}_{L^{2}(\R^{n})}^{2}=\sum_{1\le j\le n_{1}}
2\re\poscal{(D_{j}u)(x)}{ix_{j} u(x)}_{L^{2}(\R^{n})}
\\&=\sum_{1\le j\le n_{1}}2\re\poscal{\xi_{j}\mathcal F_{n_{1}}u}{i\mathcal F_{n_{1}}(x_{j} u)}_{L^{2}(\R^{n})}
\\
&\le2\sum_{1\le j\le n_{1}}
\left(\iint {\val{\xi_{j}}}^{2}\Val{(\mathcal F_{n_{1}}u)(\xi_{[n_{1}]},x_{[n_{2}]})}^{2}d\xi_{[n_{1}]}dx_{[n_{2}]}\right)^{1/2}
\left(\int {\val{x_{j}}}^{2}\Val{u(x)}^{2}dx\right)^{1/2}
\\
&\le 2
\left[\sum_{1\le j\le n_{1}}\iint{\val{\xi_{j}}}^{2}\Val{(\mathcal F_{n_{1}}u)(\xi_{[n_{1}]},x_{[n_{2}]})}^{2}d\xi_{[n_{1}]}dx_{[n_{2}]}\right]^{1/2}
\hskip-15pt\times\left[\sum_{1\le j\le n_{1}}\int{\val{x_{j}}}^{2}\Val{u(x)}^{2}dx\right]^{1/2}
\\
&\le 2\mathcal V\bigl(\mathcal F_{n_{1}} u\bigr)^{1/2}\mathcal V (u)^{1/2},
\end{align*}
proving the sought inequality. To prove the equality \eqref{2316}, we choose
$u$ as a tensor product 
$$
u_{\varepsilon}(x)= 2^{n_{1}/4}e^{-\pi\norm{x_{[n_{1}]}}^{2}} \phi\bigl(\frac{x_{[n_{2}]}}{\varepsilon}\bigr)\varepsilon^{-n_{2}/2},\quad \phi\in \mathscr S(\R^{n_{2}}),
\norm{\phi}_{L^{2}(\R^{n_{2}})}=1, \quad \varepsilon>0.
$$
We check easily that
\begin{align*}
\mathcal V(\mathcal F_{n_{1}}u_{\varepsilon})&=\int {\val{\xi_{[n_{1}]}}}^{2}
2^{n_{1}/2} e^{-2\pi\norm{\xi_{[n_{1}]}}^{2}}d\xi_{[n_{1}]}
+O(\varepsilon^{2}), 
\\
\mathcal V(u_{\varepsilon})
&=\int {\val{x_{[n_{1}]}}}^{2}
2^{n_{1}/2} e^{-2\pi\norm{x_{[n_{1}]}}^{2}}dx_{[n_{1}]}
+O(\varepsilon^{2}),
\end{align*}
and we already know  from \eqref{oi54ff} that
$$
\int {\val{\xi_{[n_{1}]}}}^{2}
2^{n_{1}/2} e^{-2\pi\norm{\xi_{n_{1}}}^{2}}d\xi_{[n_{1}]}\int {\val{x_{[n_{1}]}}}^{2}
2^{n_{1}/2} e^{-2\pi\norm{x_{n_{1}}}^{2}}dx_{[n_{1}]}=\frac{n_{1}^{2}}{2^{4}\pi^{2}},
$$
proving the lemma.
 \end{proof}
 \subsection{Using the symplectic covariance}
\begin{theorem}\label{lem.44kkqs} Let $\mathcal M$ be an element of the metaplectic group and let $\Xi=\Psi(\mathcal M),$
where $\Psi$ is the homomorphism defined in  \eqref{hqma28}.
Then for $a, b$ real-valued linear forms on $\R^{n}\times \R^{n}$,
we have,
for all $u\in \mathscr S(\R^{n})$, 
 \begin{align}\label{2419++}
2\re&\Poscal{ \opw{a}\mathcal Mu}{i \mathcal M\opw{b} u}_{L^{2}(\R^{n})}
=\frac{\poi{a\circ \Xi}{b}}{2\pi}\norm{u}^{2}_{L^{2}(\R^{n})}.
\end{align}
\end{theorem}
\begin{nb}
 The Poisson bracket $\poi{a\circ \Xi}{b}$ is a constant scalar function since 
 $a\circ \Xi$ and $b$ are both linear forms on $\RZ$.
\end{nb}
\begin{proof}
 The proof is obvious and follows indeed the same lines as the proof of \eqref{tar654} above.
 \end{proof}
 \begin{corollary}
 Let $\mathcal M$, $\Xi$ be as in Theorem \ref{lem.44kkqs} and let $b$ be 
 the linear form on $\RZ$ given by $b(X)=[T,X]$ for some $T\in \RZ$. Then there exists a real-valued linear form $a$ on $\RZ$ such that,
 for all $u\in \mathscr S(\R^{n})$, 
 \begin{equation}\label{2420++}
\norm{\opw{a}\mathcal Mu}_{L^{2}(\R^{n})}\norm{\opw{b}u}_{L^{2}(\R^{n})}\ge\frac{\norm{\Xi T}^{2}}{4\pi}
\norm{u}^{2}_{L^{2}(\R^{n})}.
\end{equation}
We may choose $a(Y)=[Y,\sigma \Xi T]$.
\end{corollary}
\begin{nb}
Assuming that $b$ is not the zero linear form,
 with the linear form $\tilde a$ defined by
 \begin{equation}\label{}
 \tilde a(Y)=\frac{[Y,\sigma \Xi T]}{\norm{\Xi T}^{2}},
\end{equation}
we get 
\begin{equation}\label{242099}
\norm{\opw{\tilde a}\mathcal Mu}_{L^{2}(\R^{n})}\norm{\opw{b}u}_{L^{2}(\R^{n})}\ge\frac{1}{4\pi}
\norm{u}^{2}_{L^{2}(\R^{n})},
\end{equation}
a somewhat general uncertainty principle for metaplectic transformations.
\end{nb}
\begin{proof}[Proof of the corollary]
  Let $Z\in \RZ$ and let $a$ be the linear form given by $a(Y)=[Z,Y]$.
We have seen that $(a\circ \Xi)(X)=[\Xi^{-1}Z, X]$.
  According to \eqref{2419++} and the Cauchy-Schwarz inequality, it is enough to find $a$ such that 
  $
 \val {\poi{a\circ \Xi}{b}}=\norm{\Xi T}^{2},
  $
 that is to find $Z\in \RZ$ such that
  \begin{equation}\label{00666}
  [\Xi^{-1} Z, T]=\norm{\Xi T}^{2}.
\end{equation}
We choose $Z=-\sigma \Xi T$,
and  this gives
$
[\Xi^{-1} Z, T]=[ Z, \Xi T]=\poscal{\sigma Z}{\Xi T}=\norm{\Xi T}^{2},
$
proving the corollary.
\end{proof}
\begin{proof}[Proof of Theorem \ref{thm.21poii}]
{\bf We assume first that $\Xi_{12}$ is invertible.}
The result  follows almost immediately from Proposition \ref{pro.baba}:
indeed,
using
\eqref{327327},
\eqref{4.eqsyma} and \eqref{4.eqsym'}, we choose
$$
L=\Xi_{12}^{-1}, \quad P=\Xi_{22}\Xi_{12}^{-1}, \quad Q=\Xi_{12}^{-1}\Xi_{11},
$$
and using the fact that $\Xi$ is a symplectic matrix, which is expressed by the equalities
\begin{align}\label{}
&\bigl(\Xi_{11}^{*}\Xi_{21}\bigr)^{*}=\Xi_{11}^{*}\Xi_{21},\label{symp01}
\\
&\bigl(\Xi_{12}^{*}\Xi_{22}\bigr)^{*}=\Xi_{12}^{*}\Xi_{22},\label{symp02}
\\
&\Xi_{11}^{*}\Xi_{22}-\Xi_{21}^{*}\Xi_{12}=I_{n},\label{symp03}
\end{align}
as well as\footnote{Note that the three conditions \eqref{symp01}, \eqref{symp02}, \eqref{symp03} are equivalent to $\Xi\in \textsl{Sp}(2n,\R)$, which is equivalent to 
$\Xi^{*-1}\in \textsl{Sp}(2n,\R)$,
itself equivalent to 
\eqref{symp04}, \eqref{symp05}, \eqref{symp06}.
As a result we can freely use these six consequences of the assumption $\Xi$ symplectic.} 
\begin{align}\label{}
&\bigl(\Xi_{11}\Xi_{12}^{*}\bigr)^{*}=\Xi_{11}\Xi_{12}^{*},\label{symp04}
\\
&\bigl(\Xi_{21}\Xi_{22}^{*}\bigr)^{*}=\Xi_{21}\Xi_{22}^{*},\label{symp05}
\\
&\Xi_{11}\Xi_{22}^{*}-\Xi_{12}\Xi_{21}^{*}=I_{n},\label{symp06}
\end{align}
we get that $P,Q$ are symmetric since
$$
P^{*}-P=\Xi_{12}^{*-1}\Xi_{22}^{*}-\Xi_{22}\Xi_{12}^{-1}=\Xi_{12}^{*-1}
\underbrace{\bigl(
\Xi_{22}^{*}\Xi_{12}-\Xi_{12}^{*}\Xi_{22}
\bigr)
}_{=0\text{ from \eqref{symp02}}}\Xi_{12}^{-1},
$$
$$
Q^{*}-Q=\Xi_{11}^{*}\Xi_{12}^{*-1}-\Xi_{12}^{-1}\Xi_{11}=\Xi_{12}^{-1}
\underbrace{\bigl(
\Xi_{12}\Xi_{11}^{*}-\Xi_{11}\Xi_{12}^{*}
\bigr)}_{=0\text{ from \eqref{symp04}}}
\Xi_{12}^{*-1},
$$
and we have moreover 
\begin{align*}
&PL^{-1}Q-L^{*}-\Xi_{21}=\Xi_{22}\Xi_{12}^{-1}\Xi_{12}\Xi_{12}^{-1}\Xi_{11}-\Xi_{12}^{*-1}-\Xi_{21}
\\
&=\Xi_{22}\Xi_{12}^{-1}\Xi_{11}-\Xi_{12}^{*-1}-\Xi_{21}
\\&\hskip-9pt \underbrace{=}_{\eqref{symp06}}\Xi_{22}\Xi_{12}^{-1}\Xi_{11}-\Xi_{12}^{*-1}-\bigl(\Xi_{22}\Xi_{11}^{*}-I_{n}\bigr)\Xi_{12}^{*-1}
\\
&=\Xi_{22}\Xi_{12}^{-1}\Xi_{11}-\Xi_{12}^{*-1}-\Xi_{22}\Xi_{11}^{*}\Xi_{12}^{*-1}+\Xi_{12}^{*-1}
\\
&=\Xi_{22}\bigl(\Xi_{12}^{-1}\Xi_{11}-\Xi_{11}^{*}\Xi_{12}^{*-1}\bigr)
\\
&=\Xi_{22}\Xi_{12}^{-1}
\underbrace{\bigl(\Xi_{11}\Xi_{12}^{*}-\Xi_{12}\Xi_{11}^{*}\bigr)\Xi_{12}^{*-1}}_{=0 \text{ from }\eqref{symp04}},
\end{align*}
and $PL^{-1}=\Xi_{22}\Xi_{12}^{-1}\Xi_{12}=\Xi_{22}$,
so that 
$\Xi=\Lambda_{P,L,Q}$ and we may apply Proposition \ref{pro.baba}
to get the inequality
\begin{equation}\label{ine001}
\mu(\mathcal M)\ge \frac1{4\pi}\trinorm{\Xi_{12}}_{\texttt{S}^{1}}.
\end{equation}
{\bf We would like now to prove  that the above inequality can be improved into an equality, while we still assume the invertibility of $\Xi_{12}$.}
To start with, we could 
take,
with $\lambda>0$, 
$$
v_{\lambda}(x)=\phi(x\lambda) \lambda^{n/2}, \quad\phi\in \mathscr S(\R^{n}), \quad
\norm{\phi}_{L^{2}(\R^{n})}=1=\norm{v_{\lambda}}_{L^{2}(\R^{n})}.
$$
To handle 
$
\mathcal M_{P,L,Q}, 
$
it is better to define
$
w_{\lambda, Q}(x)=v_{\lambda}(x) e^{-i\pi Qx^{2}},
$
which has $L^{2}$ norm 1 and is such that 
\begin{multline*}
\bigl(\mathcal M_{P,L,Q} w_{\lambda,Q})(x)=e^{-\frac{i\pi n}4+\frac{i\pi m}2} e^{i\pi P x^{2}}
{\val{\det L}}^{1/2}
\int e^{-2i\pi Lx\cdot y} v_{\lambda}(y) dy
\\
=e^{-\frac{i\pi n}4+\frac{i\pi m}2} e^{i\pi P x^{2}}
{\val{\det L}}^{1/2}(\mathcal F v_{\lambda})(Lx)
\\=e^{-\frac{i\pi n}4+\frac{i\pi m}2} e^{i\pi P x^{2}}
{\val{\det L}}^{1/2}\widehat \phi (Lx\lambda^{-1})\lambda^{-n/2},
\end{multline*}
entailing that
\begin{gather}
\mathcal V\bigl(\mathcal M_{P,L,Q} w_{\lambda,Q}\bigr)=\val{\det L}\int 
\val x^{2}
{\val{\widehat \phi (Lx\lambda^{-1})}}^{2} \lambda^{-n}dx=
\lambda^{2}\int 
{\val{L^{-1}\xi}}^{2}
{\val{\widehat \phi (\xi)}}^{2} d\xi
\label{est001}
\\
\mathcal V\bigl(w_{\lambda,Q}\bigr)=\lambda^{-2}\int {\val{x}}^{2}{\val{\phi(x)}}^{2} dx.
\label{est002}
\end{gather}
We obtain thus with $\Xi_{12}=L^{-1}$,
\begin{multline}\label{87iuff}
\sqrt{\mathcal V\bigl(\mathcal M_{P,L,Q} w_{\lambda,Q}\bigr)\mathcal V\bigl(w_{\lambda,Q}\bigr)}
\\=\left(\int
{\val{\Xi_{12} \xi}}^{2}
{\val{\widehat \phi (\xi)}}^{2} 
 d\xi\right)^{1/2}
\left(\int
{\val{x}}^{2}
{\val{\phi (x)}}^{2} dx\right)^{1/2}
\\
\le \frac12\Bigl\langle
\OPW{{\val{\Xi_{12} \xi}}^{2}+{\val x}^{2}}\phi, \phi\Bigr\rangle_{L^{2}(\R^{n})}.
\end{multline}
Using the polar decomposition of
Proposition \ref{pro.54128u}
in our Appendix,
 we have 
$$
\Xi_{12}=U\val{\Xi_{12}}, \quad\text{where $U$ is a unitary matrix},
$$
so that 
\begin{equation}\label{fd55qq}
\OPW{{\val{\Xi_{12} \xi}}^{2}+{\val x}^{2}}=\OPW{\bigl\vert
\val{\Xi_{12}} \xi
\bigr\vert^{2}+{\val x}^{2}},
\end{equation}
and the symplectic covariance of the Weyl calculus (see Theorem \ref{thm.3216++}) 
shows that 
\eqref{fd55qq} is unitarily equivalent to the operator
\begin{equation}\label{}
\OPW{\bigl\vert
\val{\Xi_{12}} B^{*}\xi
\bigr\vert^{2}+{\val {B^{-1}x}}^{2}},\quad\text{where $B$ belongs to $\textsl{Gl}(n,\R)$}.
\end{equation}
We may choose 
\begin{equation}\label{}
B={\val{\Xi_{12}}}^{-1/2}=B^{*},
\end{equation}
and we get that \eqref{fd55qq} is unitarily equivalent to the operator
\begin{multline}\label{tr88ss}
\OPW{\bigl\vert
{\val{\Xi_{12}}}^{1/2}\xi
\bigr\vert^{2}+\bigl\vert
{\val{\Xi_{12}}}^{1/2}x
\bigr\vert^{2}}
=\OPW{
\poscal{\val{\Xi_{12}} \xi}{\xi}+\poscal{\val{\Xi_{12}} x}{x}
}
\\
=\OPW{\sum_{1\le j\le n}
\lambda_{j}\bigl(\val{\Xi_{12}}\bigr)(x_{j}^{2}+\xi_{j}^{2})},
\end{multline}
where 
$
\{\lambda_{j}\bigl(\val{\Xi_{12}}\bigr)\}_{1\le j\le n}
$
are the eigenvalues of $\val{\Xi_{12}}$,
that is the singular values of $\Xi_{12}$.
A consequence of the inequality \eqref{87iuff} and the identity \eqref{tr88ss}   is that 
\begin{multline}\label{9854hh}
\mu(\mathcal M_{P,L,Q})
\le \inf_{\substack{\phi\in \mathscr S(\R^{n})\\
\norm{\phi}_{L^{2}(\R^{n})}=1}}
\sqrt{\mathcal V\bigl(\mathcal M_{P,L,Q} w_{\lambda,Q}\bigr)\mathcal V\bigl(w_{\lambda,Q}\bigr)}
\\\le \frac12
\inf\Bigl[
\OPW{\sum_{1\le j\le n}
\lambda_{j}\bigl(\val{\Xi_{12}}\bigr)(x_{j}^{2}+\xi_{j}^{2})}
\Bigr].
\end{multline}
On the  other hand, its is easy to calculate the right-hand-side of \eqref{9854hh},
which is
\begin{equation}\label{}
\frac1{4\pi}\sum_{1\le j\le n}
\lambda_{j}\bigl(\val{\Xi_{12}}\bigr)=\frac1{4\pi}\trinorm{\Xi_{12}}_{{\mathtt{S}^{1}}},
\end{equation}
proving that 
\begin{equation}\label{}
\mu(\mathcal M_{P,L,Q})
\le \frac1{4\pi}\trinorm{\Xi_{12}}_{{\mathtt{S}^{1}}},
\end{equation}
and eventually with \eqref{ine001}
the sought equality in  of  Theorem \ref{thm.21poii} in the case where $\Xi_{12}$ is invertible.
We are thus left with proving Theorem \ref{thm.21poii} in the case where $\Xi_{12}$ is a singular matrix.
Example \ref{exa.444} shows that $\Xi_{12}$ could be singular
and, even worse, it could happen that the four blocks $\{\Xi_{jk}\}_{1\le j,k\le 2}$ in the decomposition of $\Xi$
in Theorem  \ref{thm.21poii} are all singular $n\times n$ matrices when $n\ge 2$.
\begin{remark}\label{rem.54jupp}
 Before tackling the general singular case, it is interesting to check first  that we have $\mu(\mathcal M)=0$ when $\Xi_{12}=0$: indeed in that case
 we find that, thanks to \eqref{symp03},
 $$
 \Xi_{11}^{*}\Xi_{22}=I_{n},
 $$
 and thus $\Xi_{11}$ is invertible proving that 
 $$
 \Xi=\Xi_{A,B,C}=
{\arraycolsep=8pt     
\begin{pmatrix*}[l]
{B^{-1}}&{-B^{-1}C}\\
{AB^{-1}}&{B^*-AB^{-1}C}
\end{pmatrix*}}
=
\mat22{I_{n}}{0}{A}{I_{n}}
\mat22{B^{-1}}{0}{0}{B^*}
\mat22{I_{n}}{-C}{0}{I_{n}},
 $$
 where $B$ is invertible and $A,C$ symmetric. The condition $\Xi_{12}=0$ forces $C=0$
 and we find
 $$
 \Xi=\mat22{I_{n}}{0}{A}{I_{n}}
\mat22{B^{-1}}{0}{0}{B^*},
$$
and $M=M_{A,B,0}^{\scriptscriptstyle\{m\}}$ with 
\begin{multline*}
(M_{A,B,0}^{\scriptscriptstyle\{m\}}v)(x)=e^{\frac{i\pi m}2}
{\val{\det B}}^{1/2}
e^{i\pi Ax^{2}}
\int_{\R^{n}}e^{i\pi (2Bx\cdot \eta)}\hat v(\eta) d\eta
\\
=e^{i\pi Ax^{2}}(\det B)^{1/2} v(Bx).
\end{multline*}
In particular we have 
\begin{multline*}
\mathcal V(Mu)\mathcal V(u)=
\iint \val{\det B}{\val{x}}^{2}{\val{v(Bx)}}^{2}
{\val{y}}^{2}{\val{v(y)}}^{2}dxdy
\\
=
\iint {\val{B^{-1}x}}^{2}{\val{v(x)}}^{2}
{\val{y}}^{2}{\val{v(y)}}^{2}dxdy\le \norm{B^{-1}}^{2}
\mathcal V(u)\mathcal V(u).
\end{multline*}
We already know that 
$$
\inf_{\substack{u\in \mathscr S(\R^{n})\\
\norm{u}_{L^{2}(\R^{n})}=1}}\mathcal V(u)=0,
$$
and this gives $\mu(M)=0$.
\end{remark}
\subsection{The singular case}
{We need now to tackle the case where $\Xi_{12}$ is a singular matrix with positive rank. We may assume that $n\ge 2$,
since the singular case in one dimension is treated by Remark \ref{rem.54jupp}.}
According to Theorem \ref{lem.44kkqs},
we find that,
choosing
$$
a(y,\eta)=\zeta\cdot y, \quad b(x,\xi)=\tau\cdot x,
$$
we have with a ``vertical'' $Z=(0,\zeta)$
$$
\Xi X=(\Xi_{11}x+\Xi_{12}\xi,\Xi_{21}x+\Xi_{22}\xi), \quad 
a(\Xi X)=[Z, \Xi X]=\zeta\cdot(\Xi_{11}x+\Xi_{12}\xi),
$$
and thus, with a unitary operator $\mathcal U$ (thanks to Proposition \ref{pro.54128u} in our Appendix),
$$
\poi{a\circ \Xi}{b}=\poscal{\Xi_{12}^{*}\zeta}{\tau}=\poscal{\zeta}{\mathcal U\val{\Xi_{12}}\tau}.
$$
We get then 
$$\sum_{1\le j\le n}
\norm{\opw{\zeta_{j}\cdot y}\mathcal M v}\norm{\opw{\tau_{j}\cdot x} v}\ge \frac1{4\pi}\norm{v}^{2}
\sum_{1\le j\le n}\val{\poscal{\mathcal U^{*}\zeta_{j}}{\val{\Xi_{12}}\tau_{j}}}.
$$
Let $(\tau_{j})_{1\le j\le n}$ be an orthonormal basis of $\R^{n}$.
We choose $\zeta_{j}=\mathcal U\tau_{j}$ 
and we obtain
that $(\zeta_{j})_{1\le j\le n}$ is also an orthonormal basis.
Moreover, we have 
\begin{multline*}
\frac1{4\pi}\sum_{1\le j\le n}\val{\poscal{\tau_{j}}{\val {\Xi_{12}}\tau_{j}}}\norm{v}^{2}
\\=
\frac1{4\pi}\sum_{1\le j\le n}\val{\poscal{\mathcal U^{*}\mathcal U\tau_{j}}{\val {\Xi_{12}}\tau_{j}}}\norm{v}^{2}
=\frac1{4\pi}\sum_{1\le j\le n}\val{\poscal{U^{*}\zeta_{j}}{\val {\Xi_{12}}\tau_{j}}}\norm{v}^{2}
\\
\le \left(\sum_{1\le j\le n}\norm{{\zeta_{j}\cdot y}\mathcal M v}^{2}\right)^{1/2}
 \left(\sum_{1\le j\le n}\norm{{\tau_{j}\cdot x} v}^{2}\right)^{1/2}
 \le \mathcal V(\mathcal M v)^{1/2}\mathcal V( v)^{1/2},
\end{multline*}
proving indeed that, even in the case where $\Xi_{12}$ is a singular matrix, we have 
\begin{equation}\label{ine002}
\mu(\mathcal M)\ge\frac1{4\pi} \trinorm{\Xi_{12}}_{{\mathtt{S}^{1}}}.
\end{equation}
{\bf We are now left with proving that the reverse inequality of \eqref{ine002} holds true,
when $\Xi_{12}$ is a singular matrix with positive rank.
}
It is helpful to begin with a simple particular case.
\begin{lemma}
 Let $n\ge 2$ be an integer, let $r\in \llbracket 1, n-1\rrbracket$ and let $C$ be a real symmetric $n\times n$ matrix with rank $r$.
 Then with $M=e^{i\pi \poscal{C D}{D}}$, we have 
 $$
 \mu(M)=\frac1{4\pi}\sum_{1\le j\le n}\val{\lambda_{j}(C)},
 $$
 where $\{\lambda_{j}(C)\}_{1\le j\le n}$ are the eigenvalues of $C$ (listed with their multiplicities).
\end{lemma}
\begin{proof}
 We note first of all that the eigenvalues of $\sqrt{C^{*}C}=\val C$ are 
 $\{\val{\lambda_{j}(C)}\}_{1\le j\le n}$ so that 
 $
 \trinorm{C}_{\texttt{S}^{1}}=\sum_{1\le j\le n}\val{\lambda_{j}(C)}.
 $
 As a consequence, from \eqref{ine002}, we get the inequality 
 \begin{equation}\label{45kkjj}
  \mu(M)\ge \frac1{4\pi}\sum_{1\le j\le n}\val{\lambda_{j}(C)},
\end{equation}
since $e^{i\pi \poscal{C D}{D}}$ is a metaplectic transformation such that 
 $$
 \Psi(e^{i\pi \poscal{C D}{D}})=\mat22{I_{n}}{-C}{0}{I_{n}}.
 $$ 
 Since $C$ is real symmetric, there exists an orthogonal matrix $\Omega$ such that 
 $$
 C=\tr{\Omega}\Delta \Omega,\quad \text{$\Delta$ diagonal 
 $(\mu_{1}, \dots, \mu_{r}, \underbrace{0, \dots, 0}_{\text{$(n-r)$ zeros}})$},
 $$
 where the $(\mu_{l})_{1\le l\le r}$ are the non-zero eigenvalues of $C$ (counted with their multipli\-cities).
 As a consequence, we have 
 $
 M=e^{i\pi \poscal{\Delta\Omega D}{\Omega D}}
 $
 and thus
 \begin{align*}
 (Mv)(x)&=\int_{\R^{n}} e^{2i\pi x\cdot \xi}e^{i\pi \poscal{\Delta\Omega \xi}{\Omega \xi}}\hat v(\xi) d\xi
 \\&=
 \int_{\R^{n}} e^{2i\pi \Omega x\cdot \eta}e^{i\pi \poscal{\Delta\eta}{\eta}}\hat v(\Omega^{-1}\eta) d\eta
 =\bigl(e^{i\pi \poscal{\Delta D_{y}}{D_{y}}}(v\circ \Omega^{-1})\bigr)(\Omega x)
 \\
 \text{\scriptsize (using the notation \eqref{4.genmet})}&=\bigl(M_{0, \Omega, 0}\expo{m}e^{i\pi \poscal{\Delta D_{y}}{D_{y}}} M_{0, \Omega^{-1}, 0}\expo{-m}v\bigr)(x),
 \end{align*}
 implying that 
 \begin{equation}\label{2439++}
 MM_{0, \Omega, 0}\expo{m}=M_{0, \Omega, 0}\expo{m}e^{i\pi \poscal{\Delta D}{D}}.
\end{equation}
We check now the operator $e^{i\pi \poscal{\Delta D}{D}}=e^{i\pi
\sum_{1\le j\le r} \mu_{j}D_{x_{j}}^{2}
}$.
With the notation
$$
y=(x_{1}, \dots ,x_{r}), \quad z=(x_{r+1}, \dots, x_{n}),
$$
we consider for $v\in \mathscr S(\R^{r}), \phi\in \mathscr S(\R^{n-r})$, 
$$
u(y, z)= v(y) \phi(z/\varepsilon)\varepsilon^{-\frac{(n-r)}2}, \quad \varepsilon>0, \quad \norm{\phi}_{L^{2}(\R^{n-r})}=1=\norm{v}_{L^{2}(\R^{r})}.
$$
We have 
\begin{align*}
\mathcal V( u)&=\iint_{\R^{r}\times \R^{n-r}}
\bigl({\val y}^{2}+\varepsilon^{2}{\val t}^{2}\bigr){\val{v(y)}}^{2} {\val{\phi(t)}}^{2}dy dt,
\\
\mathcal V( e^{i\pi \poscal{\Delta D}{D}}u)&=\iint_{\R^{r}\times \R^{n-r}}
\bigl({\val y}^{2}+\varepsilon^{2}{\val t}^{2}\bigr)
{\val{
e^{i\pi
\sum_{1\le j\le r} \mu_{j}D_{x_{j}}^{2}
}v(y)
}}^{2} {\val{\phi(t)}}^{2}dy dt, 
\end{align*}
so that, with obvious notations, 
\begin{multline*}\label{}
\mathcal V_{n}( e^{i\pi \poscal{\Delta D}{D}}u)\mathcal V_{n}( u)=
\\
\Bigl(\int {\val y}^{2}
{\val{e^{i\pi\sum_{1\le j\le r} \mu_{j}D_{x_{j}}^{2}}v(y)}}^{2}
+\varepsilon^{2}\int {\val t}^{2}\val {\phi(t)}^{2}dt\Bigr)
\times
\Bigl(\int \val y^{2}{\val{v(y)}}^{2}+\varepsilon^{2}\int {\val t}^{2}{\val {\phi(t)}}^{2}dt\Bigr)
\\
=\left(\mathcal V_{r}( e^{i\pi \poscal{\Delta D}{D}}v)+\varepsilon^{2}\int {\val t}^{2}{\val {\phi(t)}}^{2} dt 
\right)
\times\left(\mathcal V_{r}( v)+\varepsilon^{2}\int {\val t}^{2}\val {\phi(t)}^{2} dt \right),
\end{multline*}
which means with
$$
\gamma_{\phi}=\left(\int {\val t}^{2}{\val {\phi(t)}}^{2} dt\right)^{1/2},
$$
\begin{multline*}
\mathcal V_{n}( e^{i\pi \poscal{\Delta D}{D}}u)\mathcal V_{n}( u)
=\mathcal V_{r}( e^{i\pi \poscal{\Delta D}{D}}v)\mathcal V_{r}( v)
\\+\varepsilon^{2}\gamma_{\phi}^{2}
\bigl(\mathcal V_{r}( e^{i\pi \poscal{\Delta D}{D}}v)+\mathcal V_{r}( v)\bigr)
+\varepsilon^{4}\gamma_{\phi}^{4}.
\end{multline*}
This proves that for any $\varepsilon>0, v_{0}\in \mathscr S(\R^{r}), \phi_{0}\in \mathscr S(\R^{n-r}),$
\begin{multline*}
\inf_{\substack{u\in \mathscr S(\R^{n})\\\norm{u}_{L^{2}(\R^{n})}=1}}\mathcal V_{n}( e^{i\pi \poscal{\Delta D}{D}}u)\mathcal V_{n}( u)
\le 
\inf_{\substack{v\in \mathscr S(\R^{r})\\\norm{v}_{L^{2}(\R^{r})}=1}}\mathcal V_{r}( e^{i\pi \poscal{\Delta D}{D}}v)\mathcal V_{r}(v)
\\
+\varepsilon^{2}\gamma_{\phi_{0}}^{2}
\bigl(\mathcal V_{r}( e^{i\pi \poscal{\Delta D}{D}}v_{0})+\mathcal V_{r}( v_{0})\bigr)
+\varepsilon^{4}\gamma_{\phi_{0}}^{4},
\end{multline*}
and thus that 
\begin{multline}\label{}
\inf_{\substack{u\in \mathscr S(\R^{n})\\\norm{u}_{L^{2}(\R^{n})}=1}}\mathcal V_{n}( e^{i\pi \poscal{\Delta D}{D}}u)\mathcal V_{n}( u)\le 
\inf_{\substack{v\in \mathscr S(\R^{r})\\\norm{v}_{L^{2}(\R^{r})}=1}}\mathcal V_{r}( e^{i\pi \poscal{\Delta D}{D}}v)\mathcal V_{r}(v)
\\
=\mu_{r}\bigl(
e^{i\pi \poscal{\Delta D}{D}}
\bigr)=\frac1{4\pi}\trinorm{\Delta}_{\texttt{S}^{1}}=\frac1{4\pi}\sum_{1\le j\le r}\val{\mu_{j}}
=
\frac1{4\pi}\sum_{1\le j\le r}\val{\lambda_{j}(C)},
\end{multline}
proving that, from \eqref{2439++}
\begin{equation}\label{df55ar}
\mu_{n}\bigl( M_{0, \Omega^{-1}, 0}\expo{-m}MM_{0, \Omega, 0}\expo{m}\bigr)=
\frac1{4\pi}\sum_{1\le j\le r}\val{\lambda_{j}(C)}.
\end{equation}
\begin{claim}
 Let $\Omega\in O(n)$. Then we have for $u\in \mathscr S(\R^{n})$,  
 $$
 \mathcal V( M_{0, \Omega, 0}\expo{m} u)=\mathcal V(u).
 $$
\end{claim}\par\no
We have indeed
$$
\mathcal V( M_{0, \Omega, 0}\expo{m} u)=\int_{\R^{n}}{\val{x}}^{2} {\val{u(\Omega x)}}^{2}dx
=\int_{\R^{n}}{\val{y}}^{2}{\val{u(y)}}^{2}dy=\mathcal V( u),
$$
proving the claim.
As a consequence we get that for $u\in \mathscr S(\R^{n})$,  
$$
\mathcal V(M_{0, \Omega^{-1}, 0}\expo{-m}MM_{0, \Omega, 0}\expo{m} u) \mathcal V(u)
=\mathcal V(MM_{0, \Omega, 0}\expo{m} u) \mathcal V(u),
 $$
 and thus
 \begin{multline*}
\inf_{\substack{u\in \mathscr S(\R^{n})\\\norm{u}_{L^{2}(\R^{n})=1}}}
\mathcal V(M_{0, \Omega^{-1}, 0}\expo{-m}MM_{0, \Omega, 0}\expo{m} u) \mathcal V(u)
=
\inf_{\substack{u\in \mathscr S(\R^{n})\\\norm{u}_{L^{2}(\R^{n})=1}}}
\mathcal V(MM_{0, \Omega, 0}\expo{m} u) \mathcal V(u)
\\
=
\inf_{\substack{v\in \mathscr S(\R^{n})\\\norm{v}_{L^{2}(\R^{n})=1}}}
\mathcal V(Mv) \mathcal V(M_{0, \Omega^{-1}, 0}\expo{-m} v)
=
\inf_{\substack{v\in \mathscr S(\R^{n})\\\norm{v}_{L^{2}(\R^{n})=1}}}
\mathcal V(Mv) \mathcal V( v)=\mu(M)
\end{multline*}
 which eventually proves the lemma, using \eqref{df55ar}.
 \end{proof}
{\bf We want finally to deal with  the general case where $\Xi_{12}$ is singular with rank $r\in\llbracket 1, n-1\rrbracket$.}
According to Proposition \ref{pro.9159++},
it is enough to tackle the case where
\begin{equation}\label{54lknn}
\mathcal M=
\mathcal M_{P_{1}, L_{1}, Q_{1}}\expo{m_{1}}\mathcal M_{P_{2}, L_{2}, Q_{2}}\expo{m_{2}},
\end{equation}
where 
$P_{j},Q_{j}$ are symmetric $n\times n$
matrices and $L_{j}$ are invertible $n\times n$
matrices.
\begin{lemma}\label{lem.nor456}
Let $\mathcal M$ be given by \eqref{54lknn}.
We have, using the notations \eqref{4.genmet},
\begin{gather}
\mathcal M
=M_{P_{1},I_{n},0}
M_{0,-I_{n},0}\expo{-n}
M_{0,B,0}\expo{m_{B}}
M_{0,I_{n},C}\expo{0}
M_{Q_{2}, I_{n},0}\expo{0},\label{mm0011}
\\
B=L_{2}^{*-1}L_{1}\in\textsl{Gl}(n,\R), \quad C
= L_{2}^{*-1}(Q_{1}+P_{2})L_{2}^{-1}\in\textsl{Sym}(n,\R).\label{mm0022}
\end{gather}
Moreover, defining
\begin{equation}\label{}
\wt{\mathcal M}=M_{0,B,0}\expo{m_{B}}
M_{0,I_{n},C}\expo{0}=M_{0,B,C}\expo{m_{B}},
\end{equation}
we have
\begin{align}
\Psi(\wt{\mathcal M})&=\wt\Xi=\mat22{B^{-1}}{-B^{-1}C}{0}{B^*},
\label{mm0033}
\\
\Psi({\mathcal M})&=\Xi=
\begin{pmatrix*}[l]
{B^{-1}CQ_{2}-B^{-1}}&{B^{-1}C}
\\
{P_{1}(B^{-1}C Q_{2}-B^{-1})-B^{*}Q_{2}}\hs&{P_{1}B^{-1}C-B^*}
\end{pmatrix*},
\label{mm003+}
\\
\mu(\mathcal M)&=\mu(\wt{\mathcal M}),\label{mm0044}
\end{align}
where $\mu$ is defined in \eqref{mesure++}.
\end{lemma}
\begin{proof}
 We have 
 \begin{align*}
&(\mathcal M v)(x)
 =\overbrace{e^{i\pi \frac{m_{1}+m_{2}}2} e^{-\frac{i\pi n}2}
 {\val{\det L_{1}}}^{1/2}{\val{\det L_{2}}}^{1/2}
}^{\omega_{12}}
\\&\hskip75pt
 \times
 \iint
 e^{i\pi(P_{1}x^{2}-2L_{1}x\cdot\xi+Q_{1}\xi^{2})}
  e^{i\pi(P_{2}\xi^{2}-2L_{2}\xi\cdot y+Q_{2}y^{2})}
v(y)
  d\xi dy 
\\
&=\omega_{12}e^{i\pi P_{1}x^{2}}
\iiint e^{i\pi(Q_{1}+P_{2})\xi^{2}}
e^{-2i\pi \xi\cdot (L_{1}x+L_{2}^{*}y)}
\mathcal F(e^{i\pi Q_{2}y^{2}} v(y))(\eta) e^{2i\pi y\cdot \eta}
d\eta dyd\xi
   \end{align*}
 \begin{align*}
&=\omega_{12} e^{i\pi P_{1}x^{2}}
\iint
\delta_{0}(\eta-L_{2}\xi)
e^{i\pi(Q_{1}+P_{2})\xi^{2}}\mathcal F(e^{i\pi Q_{2}y^{2}} v(y))(\eta)
e^{-2i\pi \xi\cdot L_{1}x} d\eta d\xi
\\
&=\omega_{12}e^{i\pi P_{1}x^{2}}
\int
e^{i\pi(Q_{1}+P_{2})\xi\cdot \xi}
\mathcal F(e^{i\pi Q_{2}y^{2}} v(y))
(L_{2}\xi)
e^{-2i\pi \xi\cdot L_{1}x}  d\xi
\\
&=
\omega_{12}{\val{\det L_{2}}}^{-1}e^{i\pi P_{1}x^{2}}
\int
e^{i\pi L_{2}^{*-1}(Q_{1}+P_{2})L_{2}^{-1}\xi^{2}}
\mathcal F(\underbrace{e^{i\pi Q_{2}y^{2}} v(y)}_{v_{Q_{2}}(y)})
(\xi)
e^{-2i\pi \xi\cdot L_{2}^{*-1}L_{1}x}  d\xi
\\
&=
e^{i\pi \frac{m_{1}+m_{2}}2} e^{-\frac{i\pi n}2}
 {\val{\det L_{1}}}^{1/2}
{\val{\det L_{2}}}^{-1/2}e^{i\pi P_{1}x^{2}}
\bigl(e^{i\pi L_{2}^{*-1}(Q_{1}+P_{2})L_{2}^{-1}D^{2}}v_{Q_{2}}\bigr)(-L_{2}^{*-1}L_{1}x).
\end{align*}
We have
thus, using the notations \eqref{4.genmet}, \eqref{sigma0}, \eqref{sig+++},
\begin{align}\label{}
\mathcal M&= M_{P_{1},I_{n},0}
M_{0,-I_{n},0}\expo{-n}M_{0,L_{2}^{*-1}L_{1},0}\expo{m_{1}+m_{2}}
e^{i\pi L_{2}^{*-1}(Q_{1}+P_{2})L_{2}^{-1}D^{2}}
M_{Q_{2}, I_{n},0}\expo{0}
\\
&=M_{P_{1},I_{n},0}
M_{0,-I_{n},0}\expo{-n}
M_{0,B,0}\expo{m_{B}}
e^{i\pi CD^{2}}
M_{Q_{2}, I_{n},0}\expo{0},
\\
\text{with } B&=L_{2}^{*-1}L_{1}, \quad C
= L_{2}^{*-1}(Q_{1}+P_{2})L_{2}^{-1},
\end{align}
proving
\eqref{mm0011}, \eqref{mm0022}.
Moreover we have 
\begin{align*}
&\Psi(\wt{\mathcal M})=\mat22{B^{-1}}{0}{0}{B^*}\mat22{I_n}{-C}{0}{I_n}
=\mat22{B^{-1}}{-B^{-1}C}{0}{B^*}=\wt\Xi
\\
&\wt\Xi_{11}=B^{-1}, \quad \wt\Xi_{12}=-B^{-1}C, \quad \wt\Xi_{21}=0,
\quad \wt\Xi_{22}=B^*,
\end{align*}
proving  \eqref{mm0033};
we calculate now
$$
\Psi(\mathcal M)= 
\mat22{I_{n}}{0}{P_{1}}{I_{n}}\mat22{-I_{n}}{0}{0}{-I_{n}}\wt{\Xi}\mat22{I_{n}}{0}{Q_{2}}{I_{n}},
$$
and we get\footnote{The most important point in \eqref{mm0033},
\eqref{mm003+}
is the fact that $\Xi_{12}=-\wt\Xi_{12}$, which is compatible with 
\eqref{mm0044} and the statement of Theorem \ref{thm.21poii}.
} easily  \eqref{mm003+}.
We note also that 
$$
\mathcal V\bigl(\mathcal M v\bigr)
=\mathcal V\bigl(M_{P_{1},I_{n},0}
M_{0,-I_{n},0}\expo{-n}
\wt{\mathcal M}
 M_{Q_{2}, I_{n},0}\expo{0}v\bigr)=
 \mathcal V\bigl(
\wt{\mathcal M}
 M_{Q_{2}, I_{n},0}\expo{0}v\bigr),
$$
so that we find 
\begin{multline*}\mu(\mathcal M)^{2}=
\inf_{\substack{v\in \mathscr S(\R^{n})\\\norm{v}_{L^{2}(\R^{n})=1}}}
\mathcal V(\mathcal M v)\mathcal V(v)
=
\inf_{\substack{v\in \mathscr S(\R^{n})\\\norm{v}_{L^{2}(\R^{n})=1}}}
\mathcal V(\wt{\mathcal M}
 M_{Q_{2}, I_{n},0}\expo{0}v)\mathcal V(v)
\\
=
\inf_{\substack{u\in \mathscr S(\R^{n})\\\norm{u}_{L^{2}(\R^{n})=1}}}
\mathcal V(\wt{\mathcal M}
 u)\mathcal V(M_{-Q_{2}, I_{n},0}\expo{0}u)
 =
\inf_{\substack{u\in \mathscr S(\R^{n})\\\norm{u}_{L^{2}(\R^{n})=1}}}
\mathcal V(\wt{\mathcal M}
 u)\mathcal V(u)
 =\mu(\wt{\mathcal M})^{2},
\end{multline*}
proving \eqref{mm0044} and the lemma.
\end{proof}
\begin{claim}\label{cla.251}
Let $\mathcal M, \wt{\mathcal M}, B,C$ be given by Lemma \ref{lem.nor456}.
To complete the proof of 
Theorem \ref{thm.21poii}, it is enough to prove that 
\begin{equation}\label{}
\mu\bigl(\wt{\mathcal M}\bigr)\le \frac1{4\pi}\trinorm{B^{-1}C}_{\texttt{S}^{1}}.
\end{equation}
Moreover, we have with $v\in \mathscr S(\R^{n})$,
 $w=e^{i\pi CD^{2}}v$, the inequality
\begin{equation}\label{ineq01}
\sqrt{\mathcal V(\wt{\mathcal M} v)\mathcal V(v)}
\le \frac12\poscal{\opw{\norm{B^{-1}x}^{2}+\norm{x+C\xi}^{2}}w}{w}_{L^{2}(\R^{n})}.
\end{equation}
\end{claim}
\begin{proof}[Proof of the claim]
 The first statement in the above claim is obvious,
  since we already know from \eqref{ine002}, \eqref{mm0044}
 that 
 $$
  \frac1{4\pi}\trinorm{\Xi_{12}=B^{-1}C}_{\texttt{S}^{1}}\le \mu(\mathcal M)
  =\mu(\wt{\mathcal M}).
 $$
 On the other hand, we have
 \begin{align*}
 &\sqrt{\mathcal V(\wt{\mathcal M} v)\mathcal V(v)}\le 
 \frac12\left(
\int \norm{B^{-1}x}^{2}{\val{(e^{i\pi CD^{2}}v)(x)}}^{2}dx
+\int \norm{x}^{2}{\val{v(x)}}^{2}dx
 \right)
 \\
&=\frac12\left(
\int \norm{B^{-1}x}^{2}{\val{w(x)}}^{2}dx
+\int \norm{x}^{2}
{\val{(e^{-i\pi C D^{2}}w)(x)}}^{2}dx
 \right)
 \\
& =\frac12\left(
\int \norm{B^{-1}x}^{2}{\val{w(x)}}^{2}dx
+\sum_{1\le j\le n}\int {x_{j}}^{2}
{\val{(e^{-i\pi C D^{2}}w)(x)}}^{2}dx
 \right)
 \\
 &=\frac12\left(
\int \norm{B^{-1}x}^{2}
{\val{w(x)}}^{2}dx
+\sum_{1\le j\le n}\norm{x_{j}e^{-i\pi C D^{2}}w}_{L^{2}(\R^{n})}^{2}
 \right)
 \\
 &=\frac12\left(
\int \norm{B^{-1}x}^{2}{\val{w(x)}}^{2}dx
+\sum_{1\le j\le n}\norm{e^{i\pi C D^{2}} x_{j}e^{-i\pi C D^{2}}w}_{L^{2}(\R^{n})}^{2}
 \right),
\end{align*}
and applying the symplectic covariance of the Weyl calculus
as given by Theorem \ref{thm.3216++} 
and Claim \ref{cla.lin001},
we find that,
with $(\mathbf e_{j}
)_{1\le j\le n}$ the canonical basis of $\R^{n}$,
\begin{multline}\label{ine003}
 \sqrt{\mathcal V(\wt{\mathcal M} v)\mathcal V(v)}\le 
 \frac12\poscal{
 \OPW{\norm{B^{-1}x}^{2}+\sum_{1\le j\le n}\poscal{x+C\xi}{\mathbf e_{j}}_{\R^{n}}^{2}}
 w}{w}_{L^{2}(\R^{n})}
\\
=\frac12\poscal{
 \OPW{
 \norm{\val{B^{-1}}x}^{2}+\norm{x+C\xi}^{2}
 }
 w}{w}_{L^{2}(\R^{n})},
\end{multline}
proving the claim.
\end{proof}
\begin{claim}\label{cla.252}
Let $\mathcal M, \wt{\mathcal M}, B,C$ be given by Lemma \ref{lem.nor456}.
We define the following quadratic form on $\RZ$, 
\begin{equation}\label{}
Q(x,\xi)= \norm{\val{B^{-1}}x}^{2}+\norm{x+C\xi}^{2}.
\end{equation}
We have then 
\begin{equation}\label{ine005}
\mu\bigl(\wt{\mathcal M}\bigr)\le\frac12 
\inf\Bigl[\spectrum 
\bigl\{\opw{Q}\bigr\}\Bigr].
\end{equation}
\end{claim}
\begin{proof}
 Inequality \eqref{ine005} follows immediately from \eqref{ine003} and the definition of $w$ in
 Claim \ref{cla.251}.
 \end{proof}
\begin{lemma}
 Let $C$ be a symmetric real $n\times n$ matrix and let $\Omega$ be a symmetric
 positive-definite $n\times n$ matrix. We define the operator
 \begin{equation}\label{}
\mathcal H_{C,\Omega}=\OPW{\norm{\Omega x}^{2}+\norm{C \xi+x}^{2}},
\end{equation}
where $\norm{\cdot}$ stands for the canonical Euclidean norm on $\R^{n}$.
Then we have
\begin{equation}\label{}
\inf_{\substack{v\in \mathscr S(\R^{n})\\ \norm{v}_{L^{2}(\R^{n})}=1}}
\poscal{\mathcal H_{C,\Omega}v}{v}_{L^{2}(\R^{n})}\le \frac1{2\pi}\trace{\sqrt{C \Omega^{2}C}}.
\end{equation}
\end{lemma}
\begin{nb}
 The lemma is already proven  when $C$ is $0$ from Remark \ref{rem.0013},
 and also when $n=1$ from Lemma \ref{lem.appytr} in our Appendix.
\end{nb}
\begin{proof}
Taking into account the above {\it Nota Bene},
we may assume that $C$ is a symmetric matrix with rank $r\in \llbracket 1,n\rrbracket$ with $n\ge 2$.
We may as well suppose that 
$$
C=\begin{pmatrix*}[l]
\Lambda&0_{r,n-r}
\\
0_{r, n-r}&0_{n-r,n-r}
\end{pmatrix*},
\qquad
\Omega=\begin{pmatrix*}[l]
\beta_{11}^{r,r}&\beta_{12}^{r,n-r}
\\
{\beta_{12}^{*}}^{\!\!n-r,r}&\beta_{22}^{n-r,n-r}
\end{pmatrix*},
$$
with $\Lambda=\diag(\lambda_{1}, \dots, \lambda_{r})$ with all $\lambda_{j}\not=0$,
$\Omega$ symmetric positive-definite, which implies in particular that $\beta_{11}$ is positive-definite.
Now the Weyl symbol of $\mathcal H_{C,\Omega}$ is,
with obvious notations\footnote{ The norm $\norm{\cdot}_{r}$ stands for an Euclidean norm on $\R^{r}$.}
$$
a(x_{1}, x_{2}, \xi_{1})=\norm{\beta_{11}x_{1}+\beta_{12}x_{2}}^{2}_{r}
+\norm{\beta_{12}^{*}x_{1}+\beta_{22}x_{2}}_{n-r}^{2}+\norm{\Lambda \xi_{1}+x_{1}}^{2}_{r}
+\norm{x_{2}}^{2}_{n-r}.
$$
Since it does not depend on $\xi_{2}$,
we may apply Proposition 
\ref{pro.app123}
and take advantage of the fact that
\begin{equation}\label{ine008}
\inf\bigl[\spectrum\bigl(\mathcal H_{C,\Omega}\bigr)\bigr]\le \inf
\bigl[\spectrum\bigl(
\ops{w,r}{a(x_{1}, 0, \xi_{1})}\bigr)\bigr].
\end{equation}
We have in fact 
\begin{multline*}
a(x_{1}, 0, \xi_{1})=
\norm{\beta_{11}x_{1}}^{2}_{r}
+\norm{\beta_{12}^{*}x_{1}}_{n-r}^{2}+\norm{\Lambda \xi_{1}+x_{1}}^{2}_{r}
\\=
\poscal{(\beta_{11}^{2}+\beta_{12}\beta_{12}^{*})x_{1}}{x_{1}}_{r}+
\norm{\Lambda \xi_{1}+x_{1}}^{2}_{r}.
\end{multline*}
 We note that $(\beta_{11}^{2}+\beta_{12}\beta_{12}^{*})$ is positive-definite,
 so applying the result in dimension $r<n$, we get that
 \begin{multline}\label{}
 \inf
\bigl[\spectrum\bigl(
\ops{w,r}{a(x_{1}, 0, \xi_{1})}\bigr)\bigr]
\le 
 \frac1{2\pi}\trace\sqrt{\Lambda(\beta_{11}^{2}+\beta_{12}\beta_{12}^{*})\Lambda}
 \\
 = \frac1{2\pi}\trace\sqrt{C\Omega^{2}C},
\end{multline}
and thanks to \eqref{ine008},
this gives the proof of the lemma.
 \end{proof}
 We can now apply this lemma  to the case where
$\Omega=\val{B^{-1}}$, where $B$ is the invertible matrix given in Lemma
\ref{lem.nor456} and we find that,
with $Q$ defined in Claim \ref{cla.252}, we have 
\begin{multline*}
\inf\Bigl[\spectrum 
\bigl\{\opw{Q}\bigr\}\Bigr]
\le \frac1{2\pi}\trace{\sqrt{C{\val{B^{-1}}}^{2}C}}
=\frac1{2\pi}\trace{\sqrt{C B^{*-1}B^{-1}C}}
\\
=\frac1{2\pi}\trinorm{B^{-1}C}_{\mathtt{S}^{1}},
\end{multline*}
which implies,
thanks to Claim \ref{cla.252},
that 
$$
\mu\bigl(\wt{\mathcal M}\bigr)\le\frac1{4\pi}\trinorm{B^{-1}C}_{\mathtt{S}^{1}},
$$
and according to Claim \ref{cla.251}, this finally gives the proof of Theorem \ref{thm.21poii}.
\end{proof}
\section{Appendix}
\subsection{Classical analysis}
\subsubsection{Fourier transform.}\label{sec.app001}
\index{Fourier transform}
 We use in this paper  the following normalization for the Fourier transform and inversion formula: for $u\in \mathscr S(\R^{n})$,
 \begin{equation}\label{fourier}
\hat u(\xi)=\int_{\R^{n}} e^{-2i \pi x\cdot \xi} u(x) dx,\quad u(x)=\int_{\R^{n}} e^{2i \pi x\cdot \xi} \hat u(\xi) d\xi,
\end{equation}
a formula that can be extended to $u\in \mathscr S'(\R^{n})$,
with defining the distribution $\hat u$ by the duality bracket  
\begin{equation}\label{}
\poscal{\hat u}{\phi}_{\mathscr S'(\R^{n}), \mathscr S(\R^{n})}=\poscal{u}{\hat \phi}_{\mathscr S'(\R^{n}), \mathscr S(\R^{n})}.
\end{equation}
Checking \eqref{fourier} for $u\in \mathscr S'(\R^{n})$ is then easy, that  is
\begin{equation}\label{inve66}
\check{\hat{\hat u}}=u,
\end{equation}
where the distribution $\check u$ is defined by
\begin{equation}\label{}
\poscal{\check u}{\phi}_{\mathscr S'(\R^{n}), \mathscr S(\R^{n})}=\poscal{u}{\check \phi}_{\mathscr S'(\R^{n}), \mathscr S(\R^{n})}, \quad \text{with}\ \check \phi(x)=\phi(-x).
\end{equation}
It is useful to notice that for $u\in \mathscr S'(\R^{n})$,
\begin{equation}\label{inve55}
\check{\hat u}=\hat{\check u}.
\end{equation}
Our  normalization yields simple formulas for the Fourier transform of Gaussian functions: for $A$ a real-valued symmetric  positive definite  $n\times n$ matrix, we define the function $v_{A}$ in the Schwartz space by
\begin{equation}\label{gau135}
v_{A}(x)= e^{-\pi\poscal{Ax}{x}},\quad \text{and we have }\quad \widehat{v_{A}}(\xi)=(\det A)^{-1/2}
 e^{-\pi\poscal{A^{-1}\xi}{\xi}}.
\end{equation}
Similarly when $B$ is 
a real-valued symmetric  non-singular  $n\times n$ matrix, the function $w_{B}$
defined by
$$
w_{B}(x)=e^{i\pi\poscal{Bx}{x}}
$$
is in $L^{\io}(\R^{n})$ and thus a tempered distribution and we have 
\index{signature}
\index{index}
\begin{equation}\label{foimga}
 \widehat{w_{B}}(\xi)={\val{\det B}}^{-1/2} e^{\frac{i\pi}{4}\sign B}
 e^{-i\pi\poscal{B^{-1}\xi}{\xi}},
\end{equation}
where $\sign B$
stands for the {\it signature of $B$} that is, with $\mathtt E$ the set of eigenvalues of $B$
(which are real and non-zero),
\begin{equation}\label{ind001}
\sign B=\underbrace{\card (\mathtt E\cap \R_{+})}_{\nu_{+}(B)}-\underbrace{\card (\mathtt E\cap \R_{-})}_{\nu_{-}(B)=\inde{(B)}}.
\end{equation}
The integer $\nu_{-}(B)$ is called the \emph{index} of $B$, noted $\inde{(B)}$; Formula 
\eqref{foimga} can be written as 
\begin{equation}\label{foimga++}
e^{-i\pi n/4}\mathcal F\left(e^{i\pi\poscal{Bx}{x}}\right)=i^{-{\inde{B}}}
{\val{\det B}}^{-1/2} 
e^{-i\pi\poscal{B^{-1}\xi}{\xi}},
\end{equation}
since $\nu_{+}+\nu_{-}=n$ (as $B$ is non-singular), 
$$
e^{\frac {i\pi n}4} e^{-\frac{i\pi\nu_{-}}2}=e^{\frac{i\pi}4(\nu_{+}+\nu_{-}-2\nu_{-})}
=e^{\frac{i\pi}4\sign (B)}.
$$
We note also that
\begin{equation}\label{ind002}
\sign(\det B)=(-1)^{\inde{B}},
\end{equation}
so that
$$
\left(i^{-{\inde{B}}}
{\val{\det B}}^{-1/2}
\right)^{2}=(-1)^{\nu_{-}}{\val{\det B}}^{-1}=\sign(\det B)
{\val{\det B}}^{-1}=(\det B)^{-1},
$$
and thus  the prefactor $i^{-{\inde{B}}}{\val{\det B}}^{-1/2} $ in the rhs of  \eqref{foimga++} is a square root of
$1/\det B$. 
\subsubsection{Schatten norms, Polar decomposition}\label{sec.11hhff}
\begin{definition}\label{def.54lkjj}
 Let $\mathbb H$ be a (real or complex)
 Hilbert space,
 let $\mathcal B(\mathbb H)$ be the Banach space of bounded linear endomorphisms of $\mathbb H$
 and let $T\in \mathcal B(\mathbb H)$. Using the fact that 
 $T^{*}T$ is a self-adjoint non-negative bounded operator, we may define, using the spectral theorem,
 \begin{equation}\label{}
\val{T}=\sqrt{T^{*}T}.
\end{equation}
Elements of the spectrum of $\val T$ will be called the \emph{singular values} of $T$. Let $p\in [1,+\io]$. 
We define the
\emph{Schatten norm of index $p$ of the operator $T$}, as
\begin{equation}\label{3143ss}
\trinorm{T}_{\mathtt{S}^{p}}=\bigl(\trace(\val T^{p})\bigr)^{1/p}\text{ for $p$ finite},
\quad
\trinorm{T}_{\mathtt{S}^{\io}}=\norm{\val T}_{\mathcal B(\mathbb H)}=\norm{T}_{\mathcal B(\mathbb H)}.
\end{equation}
\end{definition}
\begin{lemma}
 Let $\mathbb H$ and $T$ be as in Definition \ref{def.54lkjj}.
 Then we have 
 \begin{equation}\label{}
\ker T=\ker\val{T}.
\end{equation}
\end{lemma}
\begin{proof}
We have, 
 $$
Tx=0\Longrightarrow T^{*}Tx=0\Longrightarrow0=\poscal{T^{*}T x}{x}=\norm{\val Tx}^{2}\Longrightarrow \val Tx =0,
$$
$$
\val Tx=0\Longrightarrow 0= \val T^{2}x=T^{*}Tx\Longrightarrow0=\poscal{T^{*}T x}{x}=\norm{Tx}^{2}\Longrightarrow Tx =0,
$$
concluding the proof.
\end{proof}
\begin{remark}
 The notation $\val T$ could be misleading. In particular we may have $\val T\not=\val{T^{*}}$:
 take for instance the right shift operator $S_{r}$  defined on $\ell ^{2}(\mathbb N)$ by
 $$
x=(x_{0}, x_{1}, \dots, x_{k},\dots)\mapsto(0, x_{0}, x_{1}, \dots, x_{k},\dots)=S_{r}x.
 $$
 The operator $S_{r}$ is isometric since we have obviously $\norm{S_{r}x}_{\ell ^{2}(\mathbb N)}=\norm{x}_{\ell ^{2}(\mathbb N)}$ and thus
 $$
 S_{r}^{*}S_{r}=I, \quad\text{which implies }\val{S_{r}}=I.
 $$
 On the other hand we have with $S_{l}=S_{r}^{*}$
 $$
 \poscal{S_{l}\mathbf e_{j}}{\mathbf e_{k}}_{\ell ^{2}(\mathbb N)}= \poscal{\mathbf e_{j}}{\mathbf e_{k+1}}_{\ell ^{2}(\mathbb N)}=\delta_{j,k+1},
 $$
 so that 
 $
 S_{l}\mathbf e_{0}=0$, and $\forall j\ge 1,   S_{l}\mathbf e_{j}=\mathbf e_{j-1},
 $
 i.e.
  $$
x=(x_{0}, x_{1}, \dots, x_{k},\dots)\mapsto(x_{1}, \dots, x_{k},\dots)=S_{l}x.
 $$
 We have thus
 $
 S_{l}^{*}S_{l}=S_{r}S_{r}^{*}
 $
 and $S_{r}S_{r}^{*}$ is the orthogonal projection onto $\ell ^{2}(\N^{*})$,
 since
 $$\forall j\ge 1, \ 
 S_{r}S_{r}^{*}(\mathbf e_{j})=S_{r}(\mathbf e_{j-1})=\mathbf e_{j},
 \quad
 S_{r}S_{r}^{*}(\mathbf e_{0})=0,
 $$
 implying that 
 $$
 \val{S_{r}^{*}}=I-\mathbb P_{0}, \quad \text{where $\mathbb P_{0}$ is the orthogonal projection onto $\mathcal V(\mathbf e_{0}).$}
 $$
 More generally if $T:\mathbb H\longrightarrow \mathbb H_{0}$ is isometric onto a proper subspace $\mathbb H_{0}$ of $\mathbb H$,
 we have 
 $$
 T^{*}T= \Id_{\mathbb H},\quad\text{implying }\val{T}=\Id_{\mathbb H},
 $$
 whereas $T T^{*}$ is the orthogonal projection onto $\mathbb H_{0}$: indeed, it is an orthogonal projection since $TT^{*}TT^{*}=T T^{*}$ and its range is included in the range of $T$ (which is $\mathbb H_{0}$); also if $y\in \mathbb H_{0}$, we have $y=Tx$ with $x\in \mathbb H$ and thus
 $$
TT^{*}Tx=Tx=y,\text{ implying that $\mathbb H_{0}\subset \range(TT^{*})\subset \range T=\mathbb H_{0}$,}
 $$
 proving that
 $TT^{*}=\mathbb P_{\mathbb H_{0}}$ and $\val{ T^{*}}=\mathbb P_{\mathbb H_{0}}$.
 \par
 Another caveat about the notation $\val{T}$ is that the commutator $[\val T, T]$ can be  different from $0$, even in two dimensions.
We consider the matrix 
 $$
 T=\mat22{1}{1}{0}{1}, \quad\text{so that }T^{*}T=\mat22{1}{1}{1}{2}.
 $$ 
 Then $[\val T, T]\not=0$, otherwise we would have $[T^{*}T, T]=[\val T^{2}, T]=0$ and in fact
$$
 [T^{*}T, T]=\mat22{1}{1}{1}{2}\mat22{1}{1}{0}{1}-\mat22{1}{1}{0}{1}\mat22{1}{1}{1}{2}
 =\mat22{-1}{-1}{0}{1}
 \not=0.
$$
Moreover the operators $\val{T}$ and $\val{T^{*}}$ could be different, even in two dimensions:
using the previous example, we get 
$$
\val T^{2}=T^{*}T=\mat22{1}{1}{1}{2}, \quad 
\val {T^{*}}^{2}=TT^{*}=\mat22{2}{1}{1}{1},
$$
which are two different matrices which are not even commuting
since 
$$
\mat22{1}{1}{1}{2}\mat22{2}{1}{1}{1}=\mat22{3}{2}{4}{3}, \quad
\mat22{2}{1}{1}{1}\mat22{1}{1}{1}{2}=\mat22{3}{4}{2}{3}.
$$
\end{remark}
\begin{nb}
 Nevertheless, we can note that the singular values of $T$ and $T^{*}$ are the same
 (cf. Corollary \ref{cor.310310}).
 \end{nb}
\begin{lemma}\label{lem.2154ee}
 Let $\mathbb H$,
 $\mathcal B(\mathbb H)$ be as in Definition \ref{def.54lkjj}.
 Let $T$ be an invertible operator in $\mathcal B(\mathbb H)$.
 Then the operator $\val T$ is invertible and 
 $$
 U= T{\val{T}}^{-1}\quad \text{is a unitary operator and we have $T=U\val T.$}
 $$
\end{lemma}
\begin{proof}[Proof of the lemma]
 Since $T$ is invertible, we find that $T^{*}$ is invertible as well as  $T^{*}T $: we have 
that
$$(T^{*}T )^{-1}=T^{-1}(T^{*})^{-1}=T^{-1}(T^{-1})^{*}.
$$
As a consequence with $\val T=\sqrt{T^{*} T}$, we find that there exists $\alpha>0$ such that for all $x\in \mathbb H$, we have 
$$
\norm{\val{T}x}^{2}=\poscal{T^{*}T x}{x}\ge \alpha\norm{x}^{2},
$$
and thus $\val T$ is invertible.
Now we get also that 
$$
U^{*}U={\val{T}}^{-1}T^{*} T{\val{T}}^{-1}=(T^{*}T)^{-1/2}T^{*} T(T^{*} T)^{-1/2}=I,
$$
and
$
U U^{*}=T{\val{T}}^{-1}{\val{T}}^{-1} T^{*}= T (T^{*}T)^{-1} T^{*}=TT^{-1}T^{*-1}T^{*}=I,
$
so that $U$ is indeed unitary.
\end{proof}
\begin{proposition}\label{pro.54128u}
 Let $\mathbb H$
 be a finite-dimensional Hilbert space.
 Let $T$ be in $\mathcal B(\mathbb H)$.
 Then there exists a unitary operator $\mathcal U$ such that 
\begin{equation}\label{316app}
 T= \mathcal U\val{T}. 
\end{equation}
\end{proposition}
\begin{proof}
According to Lemma \ref{lem.2154ee}, we may assume that the kernel of $T$ has dimension
$r\in \llbracket 1, n-1\rrbracket$, where $n=\dim \mathbb H$. Then we have 
$$
\mathbb H= \underbrace{\ker T}_{\text{dim }r}\oplus 
\underbrace{(\ker T)^{\perp}}_{\text{dim }n-r}= \underbrace{(\range T)^{\perp}}_{\text{dim } r}\oplus\underbrace{\range T}_{\text{dim }n-r},
$$
and thus
$$\begin{matrix*}[r]
\tilde T: (\ker T)^{\perp}&\longrightarrow& \range T
\\
x&\mapsto&Tx
\end{matrix*}
\quad
\text{ is an isomorphism.}
$$
We have thus, with $P$ standing for the orthogonal projection on $(\ker T)^{\perp}$
(note that $P^{*}$ is the canonical injection\footnote{Indeed we have for $x_{2}\in (\ker T)^{\perp}$, $x_{1}\in \mathbb H$,
$$
\poscal{P^{*}x_{2}}{x_{1}}_{\mathbb H}=\poscal{x_{2}}{Px_{1}}_{(\ker T)^{\perp}}
=\poscal{x_{2}}{x_{1}}_{\mathbb H},\quad\text{since $x_{2}\in(\ker T)^{\perp}.$}
$$ } of $(\ker T)^{\perp}$ into $\mathbb H$)
and $J$ for the canonical injection of $\range T$ into $\mathbb H$
(note that $J^{*}$ is the orthogonal projection from $\mathbb H$ onto $\range T$),
\begin{equation}\label{gf54vv}
T=J \tilde T P,
\end{equation}
so that 
$
T^{*}T=P^{*}\tilde T^{*}\underbrace{J^{*}J}_{\Id_{\range T}} \tilde T P,
$
and for $x=x'\oplus x'', x''=Px$,
$$
T^{*}Tx=P^{*}\tilde T^{*} J^{*}J\underbrace{\tilde T Px}_{\in \range T}=P^{*}{\val{\tilde T}}^{2}P.
$$
We note also that
$
P^{*}\val{\tilde T}PP^{*}\underbrace{\val{\tilde T}Px}_{\in (\ker T)^{\perp}}
=P^{*}\val{\tilde T}\val{\tilde T} Px,
$
proving that
$
\val{T}=P^{*}\val{\tilde T}P.
$
We have in fact, thanks to Lemma \ref{lem.2154ee},
 with a unitary $W$, the following commutative diagram:{\Large
\vs
\[
\xymatrix@R=4pc@C=3pc{\mathbb{H} \ar[rr]^T \ar[d]_P && \mathbb{H}
\\
(\ker T)^{\perp} \ar[dr]_{\val{\tilde{T}}} \ar[rr]^{\tilde{T}} && \range T \ar[u]_J
\\
&(\ker T)^{\perp} \ar[ur]_W
}
\]
\vs}
\no
We have also  from \eqref{gf54vv} and the above  diagram 
$$
T=J W\val{\tilde T} P,
$$
and defining with $R_{0}$ isometric from $\ker T$ onto $(\range T)^{\perp}$
(which are finite-dimensional spaces with the same dimension),
\begin{equation}\label{ggff44}
\mathcal Ux=\mathcal U ((I-P) x\oplus Px)= \underbrace{R_{0}\overbrace{(I-P) x}^{\in \ker T}}_{\in (\range T)^{\perp}}+\underbrace{W \hskip-7pt \overbrace{Px}^{\in (\ker T)^{\perp}}}_{\in \range T},
\end{equation}
we get
$$
\mathcal U \val T x=\mathcal U \overbrace{ P^{*}\underbrace{\val {\tilde T } Px}_{\in (\ker T)^{\perp}}
}^{\in (\ker T)^{\perp}}
=\underbrace{W\val{\tilde T}     \overbrace{Px}^{\in (\ker T)^{\perp}}      }_{\in \range T}=JW\val{\tilde T} Px=Tx,$$
so that 
$$
T=\mathcal U \val{T},
$$
and $\mathcal U$ is unitary
since, using \eqref{ggff44} and the isometries $R_{0}$ and $W$, we get 
$$
\norm{\mathcal U x}^{2}=\norm{R_{0}(I-P)x}^{2}+\norm{W Px}^{2}
=\norm{(I-P)x}^{2}+\norm{Px}^{2}=\norm{x}^{2},
$$
proving that
$
\mathcal U^{*}\mathcal U=I,
$
providing the result 
since $\mathbb H$ is finite dimensional.
\end{proof}
\begin{corollary}\label{cor.310310}
Let $\mathbb H$, $T$ be as in Proposition \ref{pro.54128u}. Then $\val{T}$ and $\val{T^{*}}$ are unitarily conjugate and in particular have  the same spectrum, so that $T$ and $T^{*}$ have the same singular values
and for all $p\in [1,+\io]$, we have 
\begin{equation}\label{}
\trinorm{T^{*}}_{\mathtt{S}^{p}}=\trinorm{T}_{\mathtt{S}^{p}}.
\end{equation}
\end{corollary}
\begin{proof}
 We have
 $
 T T^{*}=\mathcal U \val{T}\val{T}\mathcal U^{*}=
 \mathcal U {\val{T}}^{2}\mathcal U^{*}
 $
 and 
 $$
  \mathcal U \val{T}\mathcal U^{*} \mathcal U \val{T}\mathcal U^{*}=
   \mathcal U {\val{T}}^{2}\mathcal U^{*},
 $$
 so that 
 $
 \val{T^{*}}=\mathcal U \val{T}\mathcal U^{*},
 $
 proving the sought result.
\end{proof}
\begin{remark}
 As it is apparent from the proof of Proposition \ref{pro.54128u},
 the decomposition \eqref{316app} is not unique,
 since we have many possible choices for the unitary operator $\mathcal U$,
 in particular with the choice of $R_{0}$ in \eqref{ggff44}.
 As a consequence, except in the case where $T$ is invertible,
 it does not coincide with the so-called \emph{polar decomposition} 
 (see e.g. Theorem {\footnotesize VI.10} in the B.~Simon \& M.~Reed' book \cite{MR0751959} and 
 Section 3.9 in J.~Conway's book \cite{MR1721402}) providing a decomposition
 $T=U\val{T}$ where $U$ is a partial isometry.
 However, the fact that $\mathcal U$ in \eqref{316app}
 can be chosen as unitary is helpful for us at several places,
 in particular for the precise computation of the lower bound of an Harmonic Oscillator.
\end{remark}
\subsection{Classical phase space analysis}
\subsubsection{The Classical Heisenberg Uncertainty Principle}\label{sec.ugtf88}
Let $u\in \mathscr S(\R)$. We have,  with $D_{t}=\frac{\p}{2i\pi \p t}$,
\begin{multline}\label{1259pp}
2\re\poscal{D_{t} u}{i t u}_{L^{2}(\R)}=\poscal{D_{t}u}{itu}+\poscal{itu}{D_{t}u}
=\poscal{D_{t}it u-it D_{t}u}{u}
\\
=\poscal{[D_{t}, it] u}{u}_{L^{2}(\R)}=\frac{1}{2\pi}\norm{u}^{2}_{L^{2}(\R)}, 
\end{multline}
implying in particular
\begin{equation}\label{112}
\norm{D_{t}u}_{L^{2}(\R)}\norm{tu}_{L^{2}(\R)}\ge \frac{1}{4\pi}\norm{u}^{2}_{L^{2}(\R)}, 
\end{equation}
which is an equality for $u(t)=e^{-\pi t^{2}}2^{1/4}$; moreover we infer also from 
\eqref{1259pp} that 
\begin{align}\label{}
\poscal{\pi(D_{t}^{2}+t^{2}) u}{u}&\ge \frac{1}{2}\norm{u}^{2}_{L^{2}(\R)}, 
\end{align}
proving that 
\begin{equation}\label{}
\min\bigl[ \spectrum{\pi(D_{t}^{2}+t^{2})}\bigr]=\frac12.
\end{equation}
The core argument in the proof of \eqref{112} is the integration by parts given by the third equality,
taking advantage of the non-commutation  of the operators of momentum $D_{t}$ and of position $t$. 
We can infer a non-trivial generalization of \eqref{112} by noticing that  for $u\in \mathscr S(\R^{n})$, we have 
\[
2\re\poscal{D_{x_{j}}u}{i x_{j}u}_{L^{2}(\R^{n})}=\frac1{2\pi}\norm{u}^{2}_{L^{2}(\R^{n})},
\] 
which implies for $p\in [1,+\io]$, $\frac 1{p}+\frac1{p'}=1$,
\begin{align*}
&\frac n{4\pi}\norm{u}^{2}_{L^{2}(\R^{n})}=\sum_{1\le j\le n}
\re\poscal{D_{x_{j}}u}{i x_{j}u}_{L^{2}(\R^{n})}
\le \sum_{1\le j\le n}\norm{D_{x_{j}} u}_{L^{p}(\R^{n})}
\norm{x_{j} u}_{L^{p'}(\R^{n})}
\\
&\le  \Bigl(\sum_{1\le j\le n}\norm{D_{x_{j}} u}^{p^{}}_{L^{p}(\R^{n})}\Bigr)^{1/p}
  \Bigl(\sum_{1\le j\le n}\norm{x_{j} u}^{{p'}^{}}_{L^{p'}(\R^{n})}\Bigr)^{1/{p'}}
  =\norm{Du}_{L^{p}(\R^{n})}\norm{xu}_{L^{p'}(\R^{n})},
\end{align*}
where for a function $\R^{n}\ni x\mapsto v(x) =\bigl(v_{j}(x)\bigr)_{1\le j\le n}\in\R^{n}$,
we have defined 
\begin{equation}\label{}
\norm{v}_{L^{p}(\R^{n})}=
\left(\int\sum_{1\le j\le n}{\val{v_{j}(x)}}^{p}
dx\right)^{1/p}=
\left(\int\norm{v(x)}_{p}^{p}dx\right)^{1/p}.
\end{equation}
We get thus
\begin{equation}\label{}
\Norm{Du}_{L^{p}(\R^{n})} 
  \Norm{xu(x) }_{L^{p'}(\R^{n})}\ge \frac n{4\pi}\norm{u}^{2}_{L^{2}(\R^{n})},
\end{equation}
and in particular for $p=2$,
since we have with the standard Euclidean norm on $\R^{n}$, and the vector 
$(Du)(x)=\bigl((D_{x_{j}} u)(x)\bigr)_{1\le j\le n}$,
\begin{multline*}
\Norm{D u}_{L^{2}(\R^{n})}^{2}=\int \sum_{1\le j\le n}
{\val{(D_{{x_{j}}}u)(x)}}^{2}dx
\\=\int \poscal{(Du)(x)}{(Du)(x)}_{\R^{n}} dx
=\Norm{\val{Du}}_{L^{2}(\R^{n})}^{2},
\end{multline*}
we get 
\begin{equation}\label{116}
\Norm{\val{D}u}_{L^{2}(\R^{n})} 
  \Norm{\val xu(x) }_{L^{2}(\R^{n})}\ge \frac n{4\pi}\norm{u}^{2}_{L^{2}(\R^{n})},
\end{equation}
with the Fourier multiplier
\begin{equation}\label{}
\val{D}=\bigl(\sum_{1\le j\le n} D_{x_{j}}^{2}\bigr)^{1/2}.\end{equation}
The inequality \eqref{116} can be written as 
\begin{equation}\label{}
\Bigl(\int{\val{\xi}}^{2}{\val{\hat u(\xi)}}^{2} d\xi\Bigr)
\Bigl(\int{\val{x}}^{2}{\val{u(x)}}^{2}dx\Bigr)
\ge \frac {n^{2}}{2^{4}\pi^{2}}\norm{u}^{2}_{L^{2}(\R^{n})},
\end{equation}
and is an equality for the Gaussian function 
\begin{equation}\label{}
u(x)= 2^{n/4} e^{-\pi\norm{x}_{2}^{2}}.
\end{equation}
Indeed we have 
$
\hat u(\xi)=2^{n/4}e^{-\pi\norm{\xi}_{2}^{2}},
$
and 
\begin{align}
2^{n}\Bigl(\int{\val{y}}^{2}e^{-2\pi\val y^{2}}\Bigr)^{2}&=2^{n}
\Bigl(\int_{0}^{+\io} r^{n+1} e^{-2\pi r^{2}} dr\Bigr)^{2}\frac{4 \pi^{n}}{\Gamma(n/2)^{2}}
\notag\\
\text{\scriptsize (using $r=s^{1/2}(2\pi)^{-1/2}$) }\hs&=2^{n}\frac{4 \pi^{n}}{\Gamma(n/2)^{2}}
\frac{1}{(2\pi)^{n+1}}
\Bigl(\int_{0}^{+\io} s^{n/2} e^{-s} \frac{ds}{2^{3/2}\sqrt{\pi}}\Bigr)^{2}
\notag\\
&=
\frac{1}{4\pi^{2}\Gamma(n/2)^{2}}
\Gamma(\frac n2+1)^{2}=\frac{1}{4\pi^{2}}\frac{n^{2}}4=\frac{n^2}{2^4\pi^2}.
\label{oi54ff}
\end{align}
\begin{nb}\textit{
We note that the proof of \eqref{116} in dimensions larger than $2$ does not follow imme\-diately from an integration by parts since we have to deal with proving the inequality
\[
\Norm{\val D u}_{L^2(\R^n)}\Norm{\val x u}_{L^2(\R^n)}\ge \frac{n}{4\pi}\Norm{ u}_{L^2(\R^n)}^2.
\]
When writing
\[
2\re\poscal{\val D u}{i\val x u}=\poscal{[\val D,i\val x] u}{u},
\]
the bracket $[\val D,i\val x]$ is not so easy to calculate
and we are not really able to find a direct $n$-dimensional argument to handle Inequality \eqref{116}.
}
\end{nb}
\begin{remark}\label{rem.kjda99}
 If $M$ is a unitary operator in $L^{2}(\R^{n})$, then we have
 \begin{equation}\label{}
\mu(M)=\mu(M^{*}),
\end{equation}
where $\mu$ is defined by \eqref{mesure++}.
Indeed, we have 
\begin{align*}
\mu(M^{*})^{2}&=\inf_{\substack{u\in L^{2}(\R^{n})\\\norm{u}_{L^{2}(\R^{n})=1}}}
\bigl(\mathcal V(M^{*}u)\mathcal V(u)\bigr)
\\
&=\inf_{\substack{v\in L^{2}(\R^{n})\\\norm{v}_{L^{2}(\R^{n})=1}}}
\bigl(\mathcal V(M^{*}Mv)\mathcal V(Mv)\bigr)
=\inf_{\substack{v\in L^{2}(\R^{n})\\\norm{v}_{L^{2}(\R^{n})=1}}}
\bigl(\mathcal V(v)\mathcal V(Mv)\bigr)=\mu(M)^{2}.
\end{align*}
\end{remark}
\begin{proposition}\label{pro.app123}
 Let $n_{1}, n_{2}\in \N^{*}$ and let $n=n_{1}+n_{2}$. Let $Q(x_{1}, x_{2},\xi_{2})$ be a non-negative quadratic form on $\R^{n_{1}}\times \R^{n_{2}}\times \R^{n_{2}}$. We define
 \begin{equation}\label{}
\lambda_{0}(x_{1})=\inf\Bigl[\spectrum 
\bigl\{\OPS{w,n_{2}}{Q(x_{1}, \cdot,\cdot)}\bigr\}\Bigr],
\end{equation}
where $\ops{w,n_{2}}{b}$ stands for the Weyl quantization of the Hamiltonian $b$
defined on $ \R^{n_{2}}\times \R^{n_{2}}$.
Then we have
\begin{equation}\label{}
\inf\Bigl[\spectrum 
\bigl\{\OPS{w,n}{Q}\bigr\}\Bigr]=\inf_{x_{1}\in \R} \lambda_{0}(x_{1}).
\end{equation}
\end{proposition}
\begin{proof}
 Let $v\in \mathscr S(\R^{n_{1}+n_{2}})$: we have
 \begin{align*}
& \poscal{\OPS{w,n}{Q} v}{v}_{L^{2}(\R^{n})} =
 \iiint\hskip-5pt\iiint e^{2i\pi(\poscal{x_{1}-y_{1}}{\xi_{1}}+\poscal{(x_{2}-y_{2})}{\xi_{2}})}
 \\
 &\hskip108pt \times Q\bigl(\frac{x_{1}+y_{1}}{2}, \frac{x_{2}+y_{2}}{2}, \xi_{2}\bigr) v(y_{1}, y_{2})
\overline{v(x_{1}, x_{2})}
 dy_{1}dy_{2}d\xi_{1}d\xi_{2} dx_{1}dx_{2}
 \\
 &=\iint\hskip-5pt\iint e^{2i\pi\poscal{x_{2}-y_{2}}{\xi_{2}}}
 Q\bigl(x_{1}, \frac{x_{2}+y_{2}}{2}, \xi_{2}\bigr) v(x_{1}, y_{2}) 
 \overline{v(x_{1}, x_{2})}dy_{2}dx_{2} d\xi_{2} dx_{1}
 \\
 &=\int_{\R^{n_{1}}}\poscal{\OPS{w,n_{2}}{Q(x_{1},\cdot,\cdot)}v_{x_{1}}}{v_{x_{1}}}_{L^{2}(\R^{n_{2}})} dx_{1},
\end{align*}
with
$
v_{x_{1}}(x_{2})=v(x_{1},x_{2}).
$
As a consequence, we have
\begin{multline*}
 \poscal{\OPS{w,n}{Q} v}{v}_{L^{2}(\R^{n})} \ge \int_{\R^{n_{1}}}
 \lambda_{0}(x_{1})
 \norm{v_{x_{1}}}_{L^{2}(\R^{n_{2}})}^{2}dx_{1}
\\ \ge\inf_{x_{1}\in \R} \lambda_{0}(x_{1})
  \int_{\R^{n_{1}}}\norm{v_{x_{1}}}_{L^{2}(\R^{n_{2}})}^{2}dx_{1}
=
\bigl(\inf_{x_{1}\in \R} \lambda_{0}(x_{1})\bigr)
\norm{v}_{L^{2}(\R^{n})}^{2},
\end{multline*}
proving that
\begin{equation}\label{3345++}
\inf\Bigl[\spectrum 
\bigl\{\OPS{w,n}{Q}\bigr\}\Bigr]\ge \inf_{x_{1}\in \R} \lambda_{0}(x_{1}).
\end{equation}
Let $z_{1}\in \R^{n_{1}}$.
We check with $\varepsilon>0$, 
$$
v_{1}(z_{1},\varepsilon,x_{1})=\phi\bigl((x_{1}-z_{1})/\varepsilon\bigr) \varepsilon^{-n_{1}/2}, \quad \phi\in \mathscr S(\R^{n_{1}}), \quad
\norm{\phi}_{L^{2}(\R^{n_{1}})}=1,
$$
using the first computation in this proof, for $v_{2}\in \mathscr S(\R^{n_{2}})$, 
\begin{multline*}
 \poscal{\OPS{w,n}{Q} (v_{1}(z_{1},\varepsilon,\cdot)\otimes v_{2})}{ (v_{1}(z_{1},\varepsilon,\cdot)\otimes v_{2})}_{L^{2}(\R^{n})} 
 \\=
 \int_{\R^{n_{1}}}\poscal{\OPS{w,n_{2}}{Q(x_{1},\cdot,\cdot)}v_{2}}{v_{2}}_{L^{2}(\R^{n_{2}})} 
{\val{ \phi\bigl((x_{1}-z_{1})/\varepsilon\bigr)}}^{2} 
\varepsilon^{-n_{1}}
 dx_{1}
 \\=
 \int_{\R^{n_{1}}}\poscal{\OPS{w,n_{2}}{Q(z_{1}+\varepsilon t_{1},\cdot,\cdot)}v_{2}}{v_{2}}_{L^{2}(\R^{n_{2}})} 
{\val{ \phi\bigl(t_{1}\bigr)}}^{2} 
 dt_{1},
\end{multline*}
so that
\begin{multline*}
 \lim_{\varepsilon\rightarrow 0_{+}}\poscal{\OPS{w,n}{Q} (v_{1}\otimes v_{2})}{ (v_{1}\otimes v_{2})}_{L^{2}(\R^{n})} 
 \\=
 \int_{\R^{n_{1}}}\poscal{\OPS{w,n_{2}}{Q(z_{1},\cdot,\cdot)}v_{2}}{v_{2}}_{L^{2}(\R^{n_{2}})} 
{\val{ \phi\bigl(t_{1}\bigr)}}^{2} 
 dt_{1}\\
 =\poscal{\OPS{w,n_{2}}{Q(z_{1},\cdot,\cdot)}v_{2}}{v_{2}}_{L^{2}(\R^{n_{2}})}.
\end{multline*}
Let $\gamma>0$ be given.
We may choose
$v_{2}(x_{2})=w(z_{1},\gamma, x_{2})$ such that 
$$\norm{v_{2}(z_{1},\gamma,  \cdot)}_{L^{2}(\R^{n_{2}})} =1,$$ and verifying 
$$
\poscal{\OPS{w,n_{2}}{Q(z_{1},\cdot,\cdot)}v_{2}}{v_{2}}_{L^{2}(\R^{n_{2}})}\le \gamma+\lambda_{0}(z_{1}).
$$
As a consequence, we find that $$
1=\norm{v_{1}(z_{1},\varepsilon, \cdot)\otimes v_{2}(z_{1},\gamma, \cdot)}_{L^{2}(\R^{n})},
$$
and
\begin{multline*}
 \lim_{\varepsilon\rightarrow 0_{+}}\poscal{\OPS{w,n}{Q} (v_{1}(z_{1},\varepsilon, \cdot)\otimes w(z_{1},\gamma, \cdot))}{ v_{1}(z_{1}, \varepsilon, \cdot)\otimes w(z_{1}, \gamma, \cdot)}_{L^{2}(\R^{n})} \\\le \gamma+\lambda_{0}(z_{1}),
\end{multline*}
so that 
\begin{multline*}
\poscal{\OPS{w,n}{Q} (v_{1}(z_{1},\varepsilon, \cdot)\otimes w(z_{1},\gamma, \cdot))}{ v_{1}(z_{1}, \varepsilon, \cdot)\otimes w(z_{1}, \gamma, \cdot)}_{L^{2}(\R^{n})} 
\\
\le \gamma+\lambda_{0}(z_{1})+\beta(z_{1},\gamma, \varepsilon), \quad
\lim_{\varepsilon\rightarrow 0_{+}}\beta(z_{1},\gamma, \varepsilon)=0.
\end{multline*}
The previous inequality implies that for $\gamma>0,\varepsilon>0, z_{1}\in \R^{n_{1}}$,
\begin{equation}\label{reza55}
\inf\Bigl[\spectrum 
\bigl\{\OPS{w,n}{Q}\bigr\}\Bigr]\le  \gamma+\lambda_{0}(z_{1})+\beta(z_{1},\gamma, \varepsilon). 
\end{equation}
Taking the $\mathtt{\limsup}$ with respect to $\varepsilon\rightarrow 0_{+}$ on both sides of 
\eqref{reza55},
we obtain  for 
 $\gamma>0, z_{1}\in \R^{n_{1}}$,
\begin{equation}\label{reza66}
\inf\Bigl[\spectrum 
\bigl\{\OPS{w,n}{Q}\bigr\}\Bigr]\le  \gamma+\lambda_{0}(z_{1}).
\end{equation}
Taking then the $\mathtt{\limsup}$ with respect to $\gamma\rightarrow 0_{+}$ on both sides of 
\eqref{reza66},
we get for 
 $z_{1}\in \R^{n_{1}}$,
\begin{equation}\label{reza77}
\inf\Bigl[\spectrum 
\bigl\{\OPS{w,n}{Q}\bigr\}\Bigr]\le \lambda_{0}(z_{1}).
\end{equation}
proving the converse inequality of \eqref{3345++}
and the proposition.
\end{proof}
\begin{corollary}
 Let $C$ be a symmetric real $n\times n$ matrix and let $\Omega$ be a symmetric
 positive-definite $n\times n$ matrix. We define the operator
 \begin{equation}\label{}
\mathcal H_{C,\Omega}=\OPW{\norm{\Omega x}^{2}+\norm{C \xi+x}^{2}},
\end{equation}
where $\norm{\cdot}$ stands for the canonical Euclidean norm on $\R^{n}$.
Then we have
\begin{equation}\label{}
\inf_{\substack{v\in \mathscr S(\R^{n})\\ \norm{v}_{L^{2}(\R^{n})}=1}}
\poscal{\mathcal H_{C,\Omega}v}{v}_{L^{2}(\R^{n})}\le \frac1{2\pi}\trace{\sqrt{C \Omega^{2}C}}.
\end{equation}
\end{corollary}
\begin{proof}
We may assume that $C$ is a symmetric matrix with rank $r\in \llbracket 1,n\rrbracket$.
We may as well suppose that 
$$
C=\begin{pmatrix*}[l]
\Lambda&0_{r,n-r}
\\
0_{r, n-r}&0_{n-r,n-r}
\end{pmatrix*},
\qquad
\Omega=\begin{pmatrix*}[l]
\beta_{11}^{r,r}&\beta_{12}^{r,n-r}
\\
{\beta_{12}^{*}}^{\!\!n-r,r}&\beta_{22}^{n-r,n-r}
\end{pmatrix*},
$$
with $\Lambda=\diag(\lambda_{1}, \dots, \lambda_{r})$ with all $\lambda_{j}\not=0$,
$\Omega$ symmetric positive-definite, which implies in particular that $\beta_{11}$ is positive-definite.
Now the Weyl symbol of $\mathcal H_{C,\Omega}$ is
$$
a(x_{1}, x_{2}, \xi_{1})=\norm{\beta_{11}x_{1}+\beta_{12}x_{2}}^{2}_{r}
+\norm{\beta_{12}^{*}x_{1}+\beta_{22}x_{2}}_{n-r}^{2}+\norm{\Lambda \xi_{1}+x_{1}}^{2}_{r}
+\norm{x_{2}}^{2}_{n-r}.
$$
Since it is a quadratic form independent of $\xi_{2}$,
we may apply 
Proposition \ref{pro.app123}.
We have for instance 
\begin{multline*}
a(x_{1}, 0, \xi_{1})=
\norm{\beta_{11}x_{1}}^{2}_{r}
+\norm{\beta_{12}^{*}x_{1}}_{n-r}^{2}+\norm{\Lambda \xi_{1}+x_{1}}^{2}_{r}
\\=
\poscal{(\beta_{11}^{2}+\beta_{12}\beta_{12}^{*})x_{1}}{x_{1}}_{r}+
\norm{\Lambda \xi_{1}+x_{1}}^{2}_{r}.
\end{multline*}
 We note that $(\beta_{11}^{2}+\beta_{12}\beta_{12}^{*})$ is positive-definite,
 so applying the result when $C$ is invertible we get that
 $$
 \inf\bigl[\spectrum(\mathcal H_{C,\Omega})\bigr]\le 
 \frac1{2\pi}\trace\sqrt{\Lambda(\beta_{11}^{2}+\beta_{12}\beta_{12}^{*})\Lambda}
 $$
 which is indeed equal to 
 $
 \frac1{2\pi}\trace\sqrt{C\Omega^{2}C},
 $
 proving the sought result.
\end{proof}
\subsubsection{On some Gaussian functions}\label{sec.gauss1}
We consider for $x\in\R, t>0$,
\begin{equation}\label{}
v_{t}(x)= e^{-\pi t x^{2}}(2t)^{1/4}, \quad\text{so that }\norm{v_{t}}_{L^{2}(\R)}=1,
\end{equation}
and 
\begin{equation}\label{}
(\mathcal M_{p,l,q} v_{t})(x)=e^{-i\pi/4}e^{i\pi m/2}{\val{l}}^{1/2}\int_{\R}
e^{i\pi(px^{2}-2l x y+qy^{2})}
e^{-\pi t y^{2}}(2t)^{1/4} dy,
\end{equation}
with $l\not=0$, $e^{i\pi m}=\sign l$. We find that 
\begin{align*}
{\val{(\mathcal M_{p,l,q} v_{t})(x)}}^{2}
&=(2t)^{1/2}\val l
\iint e^{i\pi(px^{2}-2l x y)} e^{-\pi (t-iq) y^{2}}
e^{-i\pi(px^{2}-2l x z)} e^{-\pi (t+iq) z^{2}} dy dz
\\
&=(2t)^{1/2}{\val{l}}
{\Val{\int
e^{-2i\pi lxy} e^{-\pi (t-iq) y^{2}} dy
}}^{2}
\\
&=(2t)^{1/2}{\val l} {\val{e^{-\pi(t-iq)^{-1} l^{2}x^{2}}}}^{2}(t^{2}+q^{2})^{-1/2},
\end{align*}
and thus we have 
\begin{align*}
\mathcal V\bigl(\mathcal M_{p,l,q} v_{t}\bigr)&=
\int x^{2}
e^{-2\pi\frac{t l^{2} x^{2}}{t^{2}+q^{2}}} (2t)^{1/2}\val l (t^{2}+q^{2})^{-1/2}dx
\\
&=\int y^{2}(2t)^{-1}\val l^{-2} (t^{2}+q^{2}) e^{-\pi y^{2}} dy,
\\
\text{and \quad}\mathcal V\bigl(v_{t}\bigr)&=\int x^{2} e^{-2\pi t x^{2}} (2t)^{1/2}dx 
\\
&=\int y^{2}(2t)^{-1}e^{-\pi y^{2}} dy.
\end{align*}
We get then
\begin{align*}\label{}
\mathcal V\bigl(\mathcal M_{p,l,q} v_{t}\bigr)\mathcal V\bigl(v_{t}\bigr)
&=(2t)^{-2}\val l^{-2} (t^{2}+q^{2})\left(\int y^{2}e^{-\pi y^{2}} dy\right)^{2}
\\
&=(2t)^{-2}\val l^{-2} (t^{2}+q^{2})\Bigl(\frac1{\pi^{3/2}} \Gamma(3/2)\Bigr)^{2}
\\&=\val l^{-2}(1+q^{2}t^{-2})\frac\pi{4^{2}\pi^{3}}
=\val l^{-2}(1+q^{2}t^{-2})\frac1{(4\pi)^{2}},
\end{align*}
and thus
\begin{equation}\label{}
\sqrt{\mathcal V\bigl(\mathcal M_{p,l,q} v_{t}\bigr)\mathcal V\bigl(v_{t}\bigr)}
=\tend{\frac{1}{4\pi \val l}\sqrt{1+q^{2}t^{-2}}}{\frac{1}{4\pi \val l}}{t}{+\io}.
\end{equation}
\subsubsection{Warm-up in one dimension}\label{sec.warmup}
 As a warm-up,
 let us check thoroughly the one-dimensional case. We start with a symplectic matrix  in $\textsl{Sp}(2,\R)=\textsl{Sl}(2,\R)$, say
 $$
 \Xi=\mat22{\alpha}{\beta}{\gamma}{\delta}, \quad \alpha \delta-\beta \gamma=1.
 $$
 $\bullet$ {\it We assume first that $\beta\not=0$} and we find that, using the Notation \eqref{327327}
 $$
  \Xi=\Lambda_{p,l,q}, \quad \alpha=l^{-1}q, \beta=l^{-1},  \gamma=pl^{-1}q-l, \delta=pl^{-1}.
 $$
 Looking for a metaplectic transformation $M$ in the fiber of $\Xi$,
 we get
 \begin{equation}\label{re87uu}
 \mathcal M_{p,l,q}\text{ with kernel\quad} e^{-\frac{i\pi}4} e^{\frac{i\pi m}2}{\val{l}}^{1/2}
 e^{i\pi(px^{2}-2l xy+qy^{2})}.
\end{equation}
 We consider now, with $a, b$ linear form on $\R^{2}$,
 $$
 2\re\poscal{\opw{a} \mathcal M_{p,l,q}v}{i \mathcal M_{p,l,q}\opw{b}v}
 =\frac1{2\pi}\poi{a\circ \Xi}{b}\norm{v}^{2}
 $$
 We choose $b(x,\xi)=x$ and we look for $a$ satisfying
 $$
 (a\circ \Xi)(x,\xi)=\xi, \quad\text{i.e.\quad}
 a(\alpha x+\beta \xi, \gamma x+\delta \xi)=\xi+\rho x.
 $$
 We choose then
 $$
 a(y,\eta)= \beta^{-1} y\quad \text{so that \quad}
 a(\alpha x+\beta \xi, \gamma x+\delta \xi)=\beta^{-1} (\alpha x+\beta \xi)=\xi+\rho x.
 $$
 We obtain thus that 
 $$
 2{\val{\beta}}^{-1}\mathcal V\bigl(\mathcal M v\bigr)^{1/2}\mathcal V(v)^{1/2}\ge \frac1{2\pi}\norm{v}^{2}
 $$
 that is 
 $$
 \mu(\mathcal M)\ge \frac{\val{\beta}}{4\pi}.
 $$
 Let us now show the equality $\mu(\mathcal M)= \frac{\val{\beta}}{4\pi}.$
 With $\mathcal M$ given by \eqref{re87uu}, we follow the calculations of Section 
 \ref{sec.gauss1} in our Appendix
 and we get 
 $\mu(M)\le \frac{\val{\beta}}{4\pi}$, proving the sought result.
 \par\no
  $\bullet$ {\it We assume now that $\beta=0$}, so that 
  $$
  \Xi=\mat22{\alpha}{0}{\gamma}{\alpha^{-1}},\quad \alpha\not=0.
  $$
  We find then that 
  $$
    \Xi=\Xi_{a,b,0}, \quad b=\alpha^{-1}, a=b\gamma.
  $$
  The operator
  $$
  (M_{a,b,0} v)(x)=e^{\frac{i\pi m}2} \val b^{1/2}\int  e^{i\pi (a x^{2} +2 bx\eta)}\hat v(\eta) d\eta
  $$
  Since $\mathcal V(e^{i\pi a x^{2}}w)=\mathcal V (w)$, we may assume that $a=0$
   and 
   $$
   (Mv)(x)= b^{1/2}v(bx),
   $$ 
   yielding
   $
   \mu(M)=0.
   $
   We have thus proven the following result.
\begin{theorem}
 Let $M\in \textsl{Mp}(1)$. Then we have $\Xi=\Psi(M)\in \textsl{Sp}(2,\R)=
 \textsl{Sl}(2,\R) $ and 
  $$\Xi=
  \mat22{\alpha}{\beta}{\gamma}{\delta}.$$
  Then we have 
  $
  \mu(M)= \frac{\val{\beta}}{4\pi}.
  $
\end{theorem}
\begin{lemma}\label{lem.appytr}
 Let $\omega, c$ be real numbers and let 
 $
 \mathcal H_{c,\omega}=\opw{(\omega x)^{2}+(c\xi+x)^{2}}.
 $
 Then we have 
 \begin{equation}\label{}
 \inf\bigl[\spectrum \mathcal H_{c,\omega}\bigr]=\frac{\val{\omega c}}{2\pi}.
\end{equation}
\end{lemma}
\begin{proof}
 We may assume that $c\not=0$, since the result is obvious when $c=0$.
 (see e.g. Remark \ref{rem.0013}.)
 We define then 
 $
 \alpha=\sqrt{\val{\omega/c}},
 $
 and we obtain that,
 thanks to the symplectic covariance of the Weyl calculus,
 $ \mathcal H_{c,\omega}$ is unitarily equivalent 
 to the operator with Weyl symbol
 $$
 \omega^{2}\alpha^{-2}x^{2}+\bigl(c \alpha\xi+\alpha^{-1} x\bigr)^{2}=\val{c\omega}
 \bigl(x^{2}+(\xi+c^{-1}\alpha^{-2}x)^{2}\bigr),
 $$
 which
 is unitarily equivalent 
 to the harmonic oscillator
 $
 \val{c\omega}\opw{\xi^{2}+x^{2}}
 $
 (also from the symplectic covariance),
 whose spectrum is 
 $$
\frac { \val{c\omega}}\pi\Bigl(\frac12+\N\Bigr),
 $$
 concluding the proof of the lemma.
\end{proof}
%%%%%%%%%%%%%%%%%%%%%%%%%%%%%%%%%%%%
\providecommand{\bysame}{\leavevmode\hbox to3em{\hrulefill}\thinspace}
\providecommand{\MR}{\relax\ifhmode\unskip\space\fi MR }
\providecommand{\MRhref}[2]{%
  \href{http://www.ams.org/mathscinet-getitem?mr=#1}{#2}
}
\providecommand{\href}[2]{#2}

\end{document}